\documentclass{article}
\usepackage[utf8]{inputenc}
\usepackage{amsmath,amssymb,amsthm, bbold, xfrac, graphicx, tabu, tikz, framed, caption, subcaption, verbatim, url, upgreek, wasysym, bbm, dictsym, staves, physics, algorithmic, calligra, mathrsfs, breqn, placeins, hyperref, pgfplots}
\usepackage[toc,page]{appendix}
\hypersetup{
    colorlinks=true,
    citecolor=blue,
    urlcolor=violet,
    linkcolor=red
}
\usepackage[margin=1 in]{geometry}
\usetikzlibrary{decorations.pathreplacing}
\usepackage[shortlabels]{enumitem}

\numberwithin{equation}{section}
\numberwithin{figure}{section}

\newtheorem{theorem}{Theorem}[section]
\newtheorem*{theorem*}{Theorem}
\newtheorem{lemma}[theorem]{Lemma}
\newtheorem{prop}[theorem]{Proposition}
\newtheorem{corollary}[theorem]{Corollary}

\theoremstyle{definition}
\newtheorem{definition}[theorem]{Definition}
\newtheorem{example}[theorem]{Example}

\newtheorem{remark}[theorem]{Remark}

\newcommand{\Var}{\mathrm{Var}}
\newcommand\inner[1]{\left\langle #1 \right\rangle}
\newcommand{\brackets}[1]{\left\{ #1 \right\}}
\newcommand{\floor}[1]{\left\lfloor #1 \right\rfloor}
\newcommand{\ceil}[1]{\left\lceil #1 \right\rceil}

\newcommand{\indicator}{\mathbf{1}}
\newcommand{\indicatorWithSetBrackets}[1]{\indicator \brackets{#1}}

\renewcommand{\P}{\mathbb{P}}
\newcommand{\E}{\mathbb{E}}
\newcommand{\R}{\mathbb{R}}
\newcommand{\dee}[1]{\; \mathrm{d} #1}
\newcommand{\eps}{\varepsilon}

\newcommand{\EHI}{\hyperref[EHI:d:ehi]{\mathrm{EHI}}}
\newcommand{\EHIPoisson}{\hyperref[EHI:d:ehiPoisson]{\mathrm{EHI}_{\mathrm{PK}}}}
\newcommand{\EHIp}{\hyperref[EHI:d:EquivalentEhis]{\mathrm{\EHI'}}}
\newcommand{\EHIpp}{\hyperref[EHI:d:EquivalentEhis]{\mathrm{\EHI''}}}

\title{\textbf{Towards a characterization of elliptic Harnack inequality for jump processes}}
\author{Jens Malmquist\\
\textit{Department of Mathematics, The University of British Columbia}\\
\textit{Vancouver, BC, Canada}\\
\textit{ORCID: 0000-0002-5439-2367}\\
\texttt{jens@math.ubc.ca}}
\date{}

\begin{document}

\maketitle

\begin{abstract}
    Let $X$ be an isotropic unimodal L\'{e}vy jump process on $\mathbb{R}^d$. We develop probabilistic methods which in many cases allow us to determine whether $X$ satisfies the elliptic Harnack inequality (EHI), by looking only at the jump kernel of $X$, and its truncated second moments. Both our positive results and our negative results can be applied to subordinated Brownian motions (SBMs) in particular. We produce the first known example of an SBM that does \textit{not} satisfy EHI. We show that for many SBMs that were previously known to satisfy EHI (such as the geometric stable process, the iterated geometric stable process, and the relativistic geometric stable process), bounded perturbations of them also satisfy EHI (which was not previously clear). We show that certain SBMs with Laplace exponent $\phi(\lambda) = \tilde{\Omega}(\lambda)$ satisfy EHI, which previous methods were unable to determine.
\end{abstract}

\tableofcontents

\section{Introduction}

Harnack inequalities are the subject of significant research in probability, harmonic analysis, and partial differential equations. The earliest Harnack inequality, proved by Carl Gustav Axel von Harnack \cite{harnack}, was for harmonic functions on the plane: if $B(x_0, r)$ and $B(x_0, R)$ are concentric balls in $\mathbb{R}^2$, with $R > r > 0$, and $u$ is a non-negative solution to the Laplace equation ($\Delta u = 0$) on $B(x_0, R)$, then
\begin{equation} \label{STAB:EarliestEhi}
    \sup_{B(x_0, r)} u \leq C \inf_{B(x_0, r)} u,
\end{equation}
where $C$ is a constant depending on $R/r$ but not on $u$.
Pini \cite{pini} and Hadamard \cite{hadamard} established an analogous inequality for caloric functions: if $R>r>0$, $t_4 > t_3 > t_2 > t_1 > 0$, and $u \in C^\infty((0, \infty) \times \mathbb{R}^d)$ is a non-negative solution to the heat equation ($\frac{\partial u}{\partial t} - \Delta u = 0$) on $(0, t_4) \times B(0, R)$, then
\begin{equation}\label{STAB:EarliestPhi}
    \sup_{B(0, r) \times [t_1, t_2]} u(t, x) \leq C \inf_{B(0, r) \times [t_3, t_4]} u(t, x),
\end{equation}
where $C$ is a constant depending on $d$, $r$, $R$, $t_1$, $t_2$, $t_3$, and $t_4$. Both inequalities have since been generalized to broader settings and to different operators playing the role of $\Delta$. Generalizations of \eqref{STAB:EarliestEhi} are called \textit{elliptic} Harnack inequalities and generalizations of \eqref{STAB:EarliestPhi} are called \textit{parabolic} Harnack inequalities. For a more detailed introduction to the history and basic theory of Harnack inequalities, we refer the reader to \cite{kassman}.

A major use of Harnack inequalities is that they often imply H\"{o}lder continuity for harmonic or caloric functions. 
This theory is well-developed in the case of diffusions (or equivalently, local Dirichlet forms), where Harnack inequalities are a central ingredient in the De Giorgi-Nash-Moser theory in harmonic analysis and partial differential equations. In the case of jump processes (or equivalently, non-local Dirichlet forms), the theory is still in its infancy. Chen, Kumagai, and Wang \cite{CkwElliptic, chen-kumagai-wang2020, ckw2, ckw} have done much of the early work in this area.
In \cite{ckw2}, they established that the parabolic Harnack inequality implies H\"{o}lder regularity for a certain class of jump processes on metric measure spaces. In \cite{Malm}, the author showed that the conditions of \cite{ckw2} can be relaxed, in particular to allow for graphs. It is unknown whether the elliptic Harnack inequality implies H\"{o}lder regularity in the non-local case.

In this paper, we explore the following question: given a jump process $X$ on $\R^d$, can we determine whether $X$ satisfies the elliptic Harnack inequality by looking at the jump kernel of $X$? For simplicity, we assume $X$ is an isotropic unimodal L\'{e}vy process.
We prove strong positive results and find a revealing counterexample. We do not achieve a full characterization of elliptic Harnack inequality for jump processes in terms of the jump kernel, but many naturally-arising jump kernels fall under the perview of our results, and we hope that this paper is the first step towards a more complete picture of what the jump kernel reveals about elliptic Harnack inequality. We also hope that our key constructions (the small/large and small/flat Meyer decompositions\textemdash see Section \ref{ss:zooOfMeyer}) will be useful tools be adapted by future researchers.

Before we discuss the results, let us give some background information on isotropic unimodal L\'{e}vy jump processes and past work on the elliptic Harnack inequality for jump processes.
A \emph{L\'{e}vy process} is a stochastic process with independent and stationary increments. Such a process is called \emph{isotropic unimodal} if for all $t>0$, there exists a non-increasing kernel $m_t : (0, \infty) \to [0, \infty)$ such that $\P(X_t \in A | X_0=x) = \int_A m_t(|x|) \dee{x}$ for all Lebesgue-measurable $A \subseteq \R^d$. So an \emph{isotropic unimodal L\'{e}vy jump process} on $\R^d$ is a process that only moves via jumps, whose increments are independent, stationary, and radially symmetric in distribution, with a density that decreases with distance from the starting point.
Such a process always has a jump kernel $j(r)$, a non-increasing function from $(0, \infty)$ to $[0, \infty)$ such that for all Lebesgue-measurable $A \subseteq \R^d$, the rate at which jumps such that $X_t - X_{t-} \in A$ occur is $\int_A j(|x|) \dee{x}$. The jump kernel always satisfies the \emph{L\'{e}vy-Khintchine condition}:
\begin{equation*}
    \int_{\R^d} \min\{1, |x|^2\} j(|x|) \dee{x} < \infty.
\end{equation*}

We are particularly interested in a subclass of isotropic unimodal L\'{e}vy processes called \emph{subordinate Brownian motions}. A non-decreasing L\'{e}vy process $S=(S_t)_{t \geq 0}$ on $[0, \infty)$ is known as a \emph{subordinator}. A subordinate Brownian motion is a process of the form $X=(X_t)_{t \geq 0} = (B_{S_t})_{t \geq 0}$, where $S$ is a subordinator and $B$ is a standard Brownian motion (sped up or slowed down by a constant factor if desired) independent of $S$.

There are three key prior works to draw from when it comes to positive results establishing elliptic Harnack inequality for jump processes.
Kim and Mimica \cite{KM} provide a set of sufficient conditions that imply elliptic Harnack inequality at small scales for a subordinate Brownian motion. (We cite this result in its entirety as our Theorem \ref{EHI:t:KM}.) Grzywny \cite{G} provides a set of conditions (involving the characteristic exponent of $X$) that imply elliptic Harnack inequality at small scales for isotropic unimodal L\'{e}vy processes.
In the much more general setting of jump processes on abstract metric measure spaces, Chen, Kumagai, and Wang \cite[Corollary 1.2]{CkwElliptic} characterize elliptic Harnack inequality in terms of Faber-Krahn, the Poincar\'{e} inequality, and cutoff-Sobolev. All three of these works use highly analytic proof techniques, and the statements of their sufficient conditions involve more sophisticated machinery than the jump process of $X$. By contrast, our proofs are mostly probabilistic and in many cases our results allow for verifying elliptic Harnack inequality via a glance at the jump kernel.

The exact statements of our main results are somewhat complicated, but let us summarize some of the highlights:

\begin{itemize}
    \item We show that if there exist exponents $\alpha$ and $\beta$ such that if $j(r) \asymp r^{d-2+\alpha} \left(\log(r^{-1}) \right)^{-1-\beta}$ for small $r$, and there exists a constant $c>0$ such that $j(2r) \geq c j(r)$ for all $r$, then $X$ satisfies the elliptic Harnack inequality at small scales. Note that in order for the L\'{e}vy-Khintchine condition to be met, we must have either $\alpha>0$ or $\alpha=0<\beta$. Otherwise, this result places no restrictions on $\alpha$ and $\beta$.
    \item Kim and Mimica \cite{KM} proved that geometric stable processes, iterated geometric stable processes, and relativistic geometric stable processes (see \eqref{EHI:GeomStablelaplaceExp}-\eqref{EHI:RelGeomStablelaplaceExp}) all satisfy the elliptic Harnack inequality at small scales. However, their method does not clearly extend to bounded perturbations of these processes. We show that bounded perturbations of these processes do indeed satisfy the elliptic Harnack inequality at small scales.
    \item In Example \ref{EHI:ex:NoA3example}, we produce a subordinate Brownian motion that satisfies elliptic Harnack inequality at small scales despite clearly falling outside the scope of the results of \cite{KM} and \cite{G}.
\end{itemize}

We are also interested in negative results.
Until this work, it was an open question whether there exists a subordinate Brownian motion that does not satisfy the elliptic Harnack inequality.
Bass and Chen \cite[Section 3]{BC} give an example of a L\'{e}vy process on $\mathbb{R}^d$ that does not satisfy the elliptic Harnack inequality. The process they consider has a high degree of non-regularity (it is not isotropic unimodal; its L\'{e}vy measure is singular with respect to the Lebesgue measure), and they exploit this in their proof. Grzywny and Kwa\'{s}nicki \cite[Example 5.5]{GK} give another process, this one with more regularity, that also fails to satisfy the elliptic Harnack inequality. The process considered in \cite[Example 5.5]{GK} is an isotropic unimodal L\'{e}vy process, but not a subordinated Brownian motion as the jump sizes are uniformly bounded from above, which is crucial to their proof. In Example \ref{EHI:ex:counterexample}, we provide the first known example of a subordinate Brownian motion for which the elliptic Harnack inequality fails.

\subsection{Setting and main results}

Let $X=(X_t)_{t \geq 0}$ be a continuous-time, L\'{e}vy jump process on $\mathbb{R}^d$ with isotropic unimodal increments. In other words, $X$ has each of the following:
\begin{itemize}
    \item Stationary increments: For all $t>s>0$, $X_t-X_s \overset{d}{=} X_{t-s}-X_0$.
    \item Independent increments: For all $t_n>\cdots>t_1>t_0$, the collection $\left\{ X_{t_i} - X_{t_{i-1}} : 1 \leq i \leq n \right\}$ is independent.
    \item Isotropic increments: For all $t>0$, the distribution of $X_t-X_0$ is radially symmetric.
    \item Unimodal increments: For all $t>0$, $X_t-X_0$ has a density which is non-increasing in $|\cdot|$.
\end{itemize}
We use the notation $X_t$ and $X(t)$ interchangeably.
We assume that $X$ is right-continuous, with left-limits (\textit{cadlag}). For all $t>0$, we use the notation $X_{t-}$ or $X(t-)$ to denote the left-limit $\lim_{s \nearrow t} X_s$.

Given such a process, let $\mathbb{P}_x$ and $\mathbb{E}_x$ denote probabilities and expectations for the initial value $X_0=x$.
Let us denote hitting times and exit times as follows:
\begin{equation*}
    T_A := \min\{t\geq 0:X_t \in A\}, \qquad \tau_A := \min\{t \geq 0:X_t \notin A\}.
\end{equation*}
For all $t>0$, let $\Delta X(t) := X_t - X_{t-}$ denote the displacement vector of the jump taken at time $t$ (or $0$, if there is not a jump taken at this time).
For all $r>0$, let $T^{(r)}$ denote the first time that the process takes a jump of magnitude greater than or equal to $r$:
\begin{equation*}
    T^{(r)} := \inf\{ t>0 : |\Delta X(t)| \geq r \}.
\end{equation*}
Such a process has an associated regular Dirichlet form $(\mathcal{E}, \mathcal{F})$.
By the Beurling-Deny formula \cite[Theorem 3.2.1]{FOT}, $\mathcal{E}$ can be decomposed into a strongly local component, a jumping component, and a killing component:
\begin{equation*}
    \mathcal{E}(f, g) = \mathcal{E}^c(f, g) + \int_{\mathbb{R}^d \times \mathbb{R}^d \setminus \mathrm{diag}} (f(x)-f(y))(g(x)-g(y)) \, \widehat{J}(dx, dy) + \int_{\mathbb{R}^d} f(x) g(x) \, \widehat{k}(dx)
\end{equation*}
where $\mathrm{diag}$ denotes the diagonal of $\mathbb{R}^d \times \mathbb{R}^d$, $\mathcal{E}^c$ is a strongly local symmetric form, $\widehat{J}$ is a symmetric non-negative Radon measure (which we call the \textit{jumping measure}), and $\widehat{k}$ is a non-negative Radon measure (which we call the \textit{killing measure}).

In our setting, $X$ is pure-jump, so the killing and strongly local components are identically $0$.
Since $X$ is isotropic unimodal, it follows from the $(1)\Longrightarrow(3)$ implication of the proposition on page 488 of \cite{W} that there exists a non-increasing function $j : (0, \infty) \to [0, \infty)$ such that $\widehat{J}(dx, dy) = j(|x-y|) \, dx \, dy$.
We refer to the function $j(r)$ as the \emph{jump kernel} of $X$. We also sometimes use the notation $J(x, y):=j(|x-y|)$ and $J(x, U):=\int_U j(|x-y|) \dee{y}$.

If $X(t) \neq X(t-)$ for some $t>0$, we refer to $X(t)-X(t-)$ as the \emph{displacement} of $X$ at time $t$, which we denote by $\Delta X(t)$. We refer to $|\Delta X(t)|$ as the \emph{magnitude} of the jump at time $t$. There is a probabilistic interpretation of the jump kernel: for all Lebesgue-measurable $A \subseteq \mathbb{R}^d$, jumps in $X$ with displacement in $A$ occur as a Poisson point process with rate $\int_A j(|x|) \dee{x}$.

\begin{definition}[Harmonicity]
    For all open $D \subseteq \mathbb{R}^d$, we say that a function $h:\mathbb{R}^d \to \mathbb{R}$ is \emph{harmonic} on $D$ (with respect to $X$) if for all open $U \Subset D$ and for all $x$, we have
    \begin{equation*}
        h(x) = \mathbb{E}_x \left[ h(X_{\tau_U}) \right].
    \end{equation*}
    (Note that when $X$ is a jump process, the harmonicity of $h$ on $D$ depends on the values $h$ takes on $D^c$, because for $U \Subset D$, $\tau_U$ need not be in $D$.)
\end{definition}

For all $x \in \mathbb{R}^d$ and $r>0$, let $B(x, r)$ denote the open ball in $\mathbb{R}^d$ with center $x$ and radius $r$. Given a Lebesgue-measurable set $A \subseteq \mathbb{R}^d$, let $|A|$ denote the Lebesgue measure of $A$.
    
\begin{definition}[Elliptic Harnack inequality] \label{EHI:d:ehi}
    We say that $X$ satisfies the \emph{elliptic Harnack inequality} ($\EHI$) if there exist constants $C \geq 1$ and $\kappa \in (0, 1)$ such that for all $x_0 \in \mathbb{R}^d$ and $r>0$, if a function $h$ is non-negative everywhere and harmonic with respect to $X$ on $B(x_0, r)$, then
    \begin{equation}\label{EHI:e:ehiDef}
        C^{-1} h(y) \leq h(x) \leq C h(y)
    \end{equation}
    for all $x, y \in B(x_0, \kappa r)$.

    Given a set $S \subseteq (0, \infty)$, let us say that $\EHI(r \in S)$ holds if there exist $C$ and $\kappa$ such that \eqref{EHI:e:ehiDef} holds for all $x_0 \in \R^d$, $r \in S$, all functions $h$ that are non-negative everywhere and harmonic on $B(x_0, r)$, and $x, y \in B(x_0, \kappa r)$. Let us also write $\EHI(r \leq 1)$ for $\EHI(r \in (0, 1])$, $\EHI(r \geq 1)$ for $\EHI(r \in [1, \infty))$, etc.
\end{definition}

In many applications, the elliptic Harnack inequality is primarily of interest at small scales. In fact, some other authors such as Kim and Mimica \cite{KM} \textit{define} the elliptic Harnack inequality as what we call $\EHI(r \leq 1)$.

For our main positive results, we make a simplifying assumption that there exists a $c_j>0$ such that for all $r>0$,
    \begin{equation} \label{EHI:e:regularJumps}
        j(2r) \geq c_j j(r).
    \end{equation}
In Section \ref{EHI:s:GreenPoisson}, we use a compactness argument to show that if \eqref{EHI:e:regularJumps} holds, then $X$ satisfies $\EHI(r \in [a, b])$ for all $0<a<b<\infty$. Therefore, when \eqref{EHI:e:regularJumps} is met, $\EHI$ only needs to be verified at very small and very large scales.

For all $r>0$, let us define the \emph{truncated second moment} of $X$ at $r$ as
\begin{equation}\label{EHI:e:m2def}
    m_2(r) := \int_{B(0, r)} |x|^2 j(|x|) \dee{x}
\end{equation}
and the \emph{tail} of $X$ at $r$ as
\begin{equation}\label{EHI:e:lambdadef}
    \lambda(r) := \int_{B(0, r)^c} j(|x|) \dee{x}.
\end{equation}
Our main results are expressed in terms of comparisons between the quantities $m_2(r)$, $r^{d+2} j(r)$, and $r^2 \lambda(r)$, all of which vary with $r$ and can be thought of as having the same ``units." In the following lemma, we see that when \eqref{EHI:e:regularJumps} holds, $r^{d+2} j(r)$ is at most a constant multiple of each of the other two quantities $m_2(r)$ and $r^2 \lambda(r)$.

\begin{lemma}\label{EHI:l:SmallestOfThreeQuantities}
    Let $X$ be an isotropic unimodal L\'{e}vy process on $\R^d$. There exists a universal constant $C_1=C_1(d)>0$ (depending only on $d$) such that for all $r>0$,
    \begin{equation*}
        r^{d+2}j(r) \leq C_1 m_2(r).
    \end{equation*}
    If $X$ satisfies \eqref{EHI:e:regularJumps}, then there exists a $C_2=C_2(c_j, d)>0$ (depending only on $c_j$ and $d$) such that for all $r>0$,
    \begin{equation*}
        r^{d+2} j(r) \leq C_2 r^2 \lambda(r).
    \end{equation*}
\end{lemma}

\begin{proof}
    For all $r>0$,
    \begin{align*}
        m_2(r) &\geq \int_{B(0, r) \setminus B(0, r/2)} |x|^2 j(|x|) \dee{x} \geq \left|B(0, r) \setminus B\left(0, \frac{r}{2} \right)\right| \cdot \left(\frac{r}{2}\right)^2 j(r) \\
        &= \frac14 (1-2^{-d}) |B(0, 1)| r^{d+2} j(r).
    \end{align*}
    If $X$ satisfies \eqref{EHI:e:regularJumps}, then
    \begin{equation*}
        \lambda(r) \geq \int_{B(0, 2r) \setminus B(0, r)} j(|x|) \dee{x} \geq |B(0, 2r) \setminus B(0, r)| \cdot c_j j(r) \geq (2^d-1)|B(0, 1)| c_j r^d j(r).
    \end{equation*}
\end{proof}

Our main result, in its full generality, is the following theorem.

\begin{theorem} \label{EHI:t:main}
    Let $X$ be an isotropic unimodal L\'{e}vy process on $\R^d$ satisfying \eqref{EHI:e:regularJumps}.
    Let $j(r)$ be its jump kernel, and let $m_2(r)$ and $\lambda(r)$ be as defined in \eqref{EHI:e:m2def} and \eqref{EHI:e:lambdadef}. Let $M, C, c > 0$ and $\eps \in (0, 1]$, and let $(r_n) \subseteq (0, \infty)$ be a sequence of positive numbers such that $M^{-1} \leq \frac{r_{n+1}}{r_n} \leq M$ for all $n$.
    Suppose that for all $n$, we have either
    \begin{equation}\label{EHI:e:bigm2}
        m_2(r_n) \geq c r_n^2 \lambda(r_n)
    \end{equation}
    or
    \begin{equation}\label{EHI:e:smallm2}
        m_2(r_n) \leq C \left( r_n^{d+2} j(r_n) \right)^\eps \left( r_n^2 \lambda(r_n) \right)^{1-\eps}.
    \end{equation}
    If $r_n \to 0$, then $X$ satisfies $\EHI(r \leq 1)$. If $r_n \to \infty$, then $X$ satisfies $\EHI(r \geq 1)$.
\end{theorem}

Note that each $r_n$ need only satisfy one of \eqref{EHI:e:bigm2} and \eqref{EHI:e:smallm2}, and it need not be the same one for all $n$. 
However, in practice, it will usually be easiest to verify that just one of \eqref{EHI:e:bigm2} or \eqref{EHI:e:smallm2} holds as $r \to 0^+$ or $r \to \infty$.

Before we state some corollaries of Theorem \ref{EHI:t:main}, let us establish some notation for comparing quantities. Given functions $f$ and $g$ taking positive values, let us say that $f(x) \lesssim g(x)$ on some set $S$ if there exists a constant $C>0$ such that $f(x) \leq Cg(x)$ for all $x \in S$.
We refer to the constant $C$ as the ``implicit constant" in $\lesssim$.
If the set $S$ is unspecified, assume it to be the entire domain of $f$ and $g$.
We say that $f(x) \lesssim g(x)$ as $x \to 0^+$ if there exists a $\delta>0$ such that $f(x) \lesssim g(x)$ on $(0, \delta]$, and that $f(x) \lesssim g(x) $ as $x \to \infty$ if there exists a $C>0$ such that $f(x) \lesssim g(x)$ on $[C, \infty)$.
We use the notation $f(x) \gtrsim g(x)$ when $g(x) \lesssim f(x)$, and $f(x) \asymp g(x)$ when we have both $f(x) \lesssim g(x)$ and $g(x) \lesssim f(x)$. Let us also say that $f(x) \ll g(x), x \to a$ (or equivalently, $g(x) \gg f(x), x \to a$) if $\lim_{x \to a} f(x)/g(x)=0$. For sequences $(a_n), (b_n) \subseteq (0, \infty)$, we similarly say that $a_n \ll b_n$ (or $b_n \gg a_n)$ if $\lim_{n \to \infty} a_n/b_n=0$.
We also apply the notation $\lesssim$, $\gtrsim$, and $\asymp$ to sequences, when thought of as functions on $\mathbb{N}$.

The following corollary describes the most common ways we check the conditions of Theorem \ref{EHI:t:main} in practice.

\begin{corollary} \label{EHI:c:main}
    Let $X$ be an isotropic unimodal L\'{e}vy process on $\R^d$ with jump kernel $j(r)$ satisfying \eqref{EHI:e:regularJumps}, and let $m_2(r)$ and $\lambda(r)$ be as defined in \eqref{EHI:e:m2def} and \eqref{EHI:e:lambdadef}.
    \begin{itemize}
        \item (a) If $m_2(r) \gtrsim r^2 \lambda(r)$ as $r \to 0^+$, then $X$ satisfies $\EHI(r \leq 1)$.
        \item (b) If $m_2(r) \gtrsim r^2 \lambda(r)$ as $r \to \infty$, then $X$ satisfies $\EHI(r \geq 1)$.
        \item (c) If there exists an $\eps \in (0, 1]$ such that $m_2(r) \lesssim (r^{d+2} j(r))^\eps (r^2 \lambda(r))^{1-\eps}$ as $r \to 0^+$, then $X$ satisfies $\EHI(r \leq 1)$.
        \item (d) If there exists an $\eps \in (0, 1]$ such that $m_2(r) \lesssim (r^{d+2} j(r))^\eps (r^2 \lambda(r))^{1-\eps}$ as $r \to \infty$, then $X$ satisfies $\EHI(r \geq 1)$.
    \end{itemize}
\end{corollary}

\begin{proof}
    For (a) and (c), we can take $r_n=2^{-n}$ for all $n$. For (b) and (d), we can take $r_n=2^n$ for all $n$. If necessary, discard finitely many terms. Then the result is a direct application of Theorem \ref{EHI:t:main}. 
\end{proof}

Note that the assumption of \eqref{EHI:e:regularJumps} must hold at all scales, not just as $r \to 0^+$ or as $r \to \infty$, because for jump processes, the jump that causes the process to exit a bounded open region can take it very far away.
Most of the time, when we check (c) or (d) in Corollary \ref{EHI:c:main}, we will simply take $\eps=1$, so that $m_2(r) \lesssim (r^{d+2}j(r))^\eps (r^2 \lambda(r))^{1-\eps}$ becomes $m_2(r) \lesssim r^{d+2} j(r)$. However, in light of Lemma \ref{EHI:l:m2lambdaComparable}, the existence of an 
$\eps \in (0, 1]$ such that $m_2(r) \lesssim (r^{d+2} j(r))^\eps (r^2 \lambda(r))^{1-\eps}$ is in principle a weaker condition to check than $m_2(r) \lesssim r^{d+2} j(r)$.

The following corollary illustrates the power of Theorem \ref{EHI:t:main}. Consider the ratio $\frac{r^{d+2}j(r)}{s^{d+2}j(s)}$, where $0<s \leq r$. If the growth of this ratio is polynomial or faster, or polylogarithmic or slower, then we have $\EHI$. The only way for $\EHI$ to fail is if the ratio oscillates between fast-growing and slow-growing, or if its growth is slower than polynomial but faster than polylogarithmic (for example, if $j(r) \asymp r^{-d-2} \exp(-\sqrt{\log(r^{-1})})$ as $r \to 0^+$). Most naturally-arising isotropic unimodal L\'{e}vy jump processes such that \eqref{EHI:e:regularJumps} holds seem to satisfy one of the conditions of this corollary, and therefore $\EHI$.

\begin{corollary}\label{EHI:c:mainNoM2}
    Let $X$ be an isotropic unimodal L\'{e}vy jump process with jump kernel $j(r)$, satisfying \eqref{EHI:e:regularJumps}.
    \begin{itemize}
        \item (a) If there exist $c, R,\alpha>0$ such that
        \begin{equation*}
            \frac{r^{d+2} j(r)}{s^{d+2} j(s)} \geq c \left(\frac{r}{s}\right)^\alpha \qquad\mbox{for all $0 < s \leq r \leq R$},
        \end{equation*}
        then
        \begin{equation*}
            m_2(r) \lesssim r^{d+2} j(r) \qquad\mbox{as $r \to 0^+$}
        \end{equation*}
        and therefore $X$ satisfies $\EHI(r \leq 1)$.
        \item (b) If $j(r) \gtrsim r^{-d-2}$ as $r \to \infty$, and there exist $c, R,\alpha>0$ such that
        \begin{equation*}
            \frac{r^{d+2} j(r)}{s^{d+2} j(s)} \geq c \left(\frac{r}{s}\right)^\alpha \qquad\mbox{for all $R \leq s \leq r$},
        \end{equation*}
        then
        \begin{equation*}
            m_2(r) \lesssim r^{d+2} j(r) \qquad\mbox{as $r \to \infty$}
        \end{equation*}
        and therefore $X$ satisfies $\EHI(r \geq 1)$.
        \item (c) If $j(r) \gtrsim r^{-d}$ as $r \to 0^+$, and there exist $C, R, \alpha>0$ such that
        \begin{equation*}
            \frac{r^{d+2} j(r)}{s^{d+2} j(s)} \leq C \left( \frac{\log(s^{-1})}{\log(r^{-1})} \right)^{1+\alpha} \qquad\mbox{for all $0 < s \leq r \leq R$},
        \end{equation*}
        then
        \begin{equation*}
            m_2(r) \gtrsim r^2 \lambda(r) \asymp r^{d+2} j(r) \qquad\mbox{as $r \to 0^+$}
        \end{equation*}
        and therefore $X$ satisfies $\EHI(r \leq 1)$.
        \item (d) If there exist $C, R, \alpha>0$ such that
        \begin{equation*}
            \frac{r^{d+2} j(r)}{s^{d+2} j(s)} \leq C \left( \frac{\log r}{\log s} \right)^{1+\alpha} \qquad\mbox{for all $R \leq s \leq r$},
        \end{equation*}
        then
        \begin{equation*}
            m_2(r) \gtrsim r^2 \lambda(r) \asymp r^{d+2} j(r) \qquad\mbox{as $r \to \infty$}
        \end{equation*}
        and therefore $X$ satisfies $\EHI(r \geq 1)$.
    \end{itemize}
\end{corollary}

Let us consider some examples.

\begin{example}
    Let $X$ be an isotropic unimodal L\'{e}vy jump process with jump kernel $j(r)$, satisfying \eqref{EHI:e:regularJumps}, such that $j(r) \asymp r^{-d-
    \alpha}$ for some $\alpha \in (0, 2)$. (Note that $\alpha \notin (0, 2)$ would cause the L\'{e}vy-Khintchine condition to fail.) This example was already known to satisfy $\EHI$, but it is a good starting point to familiarize the reader with the three quantities $r^{d+2}j(r)$, $m_2(r)$, and $r^2 \lambda(r)$, and how we compare them to check the conditions of our main results. A simple calculation shows that
    \begin{equation*}
        m_2(r) \asymp r^2 \lambda(r) \asymp r^{d+2} j(r) \asymp r^{2-\alpha}.
    \end{equation*}
    This satisfies all four conditions of Corollary \ref{EHI:c:main}, so $X$ satisfies $\EHI$.
\end{example}

\begin{example}\label{EHI:ex:GeneralExampleSmallR}
    Let $X$ be an isotropic unimodal L\'{e}vy jump process with jump kernel $j(r)$, satisfying \eqref{EHI:e:regularJumps}, such that $j(r) \asymp r^{-d-2+\alpha} \left(\log(r^{-1})\right)^{-1-\beta}$ as $r \to 0^+$, for some $\alpha, \beta \in \R$. Then $X$ satisfies $\EHI(r \leq 1)$. To see this, consider the following argument.
    
    First, note that $\alpha$ must be non-negative, or else the L\'{e}vy-Khintchine condition fails.
    If $\alpha>0$, choose some $\alpha' \in (0, \alpha)$. For all small values $s$ and $r$ such that $0<s \leq r$, we have
    \begin{equation*}
        \frac{r^{d+2}}{s^{d+2}} \asymp \left( \frac{r}{s} \right)^\alpha \left( \frac{\log(r^{-1})}{\log(s^{-1})} \right)^{-1-\beta} \gtrsim \left( \frac{r}{s} \right)^{\alpha'}
    \end{equation*}
    so $X$ satisfies $\EHI(r \leq 1)$ by Corollary \ref{EHI:c:mainNoM2}(a).  
    If $\alpha=0$, then $\beta$ must be positive, or else the L\'{e}vy-Khintchine condition fails. For all small values $s$ and $r$ such that $0<s \leq r$, we have
    \begin{equation*}
        \frac{r^{d+2}}{s^{d+2}} \asymp \left( \frac{\log(r^{-1})}{\log(s^{-1})} \right)^{-1-\beta} \asymp \left( \frac{\log(s^{-1})}{\log(r^{-1})} \right)^{1+\beta}.
    \end{equation*}
    We also have $j(r) \asymp r^{-d-2} \left(\log(r^{-1}) \right)^{-1-\beta} \gtrsim r^{-d}$ as $r \to 0^+$.
    Thus, $X$ satisfies $\EHI(r \leq 1)$ by Corollary \ref{EHI:c:mainNoM2}(c).
\end{example}

\begin{example}
    Let $X$ be an isotropic unimodal L\'{e}vy jump process with jump kernel $j(r)$, satisfying \eqref{EHI:e:regularJumps}, such that $j(r) \asymp r^{-d-\alpha} (\log r)^\beta$ as $r \to \infty$, for some $\alpha \in [0, 2)$ and $\beta \in \R$. Then $X$ satisfies $\EHI(r \geq 1)$.
    
    To see this, choose an $\alpha' \in (0, 2-\alpha)$. For all large values $s$ and $r$ such that $0<s \leq r$, we have
    \begin{equation*}
        \frac{r^{d+2}j(r)}{s^{d+2}j(s)} \asymp \left(\frac{r}{s}\right)^{2-\alpha} \left(\frac{\log r}{\log s} \right)^\beta \gtrsim \left(\frac{r}{s}\right)^{\alpha'}
    \end{equation*}
    so $X$ satisfies $\EHI(r \geq 1)$ by Corollary \ref{EHI:c:mainNoM2}(b).
    
\end{example}

\subsection{Subordinate Brownian motions}\label{EHI:ss:sbm}

We are particularly interested in applying our results to a class of processes called subordinate Brownian motions, which we introduce in this subsection. We show that there exists a subordinate Brownian motion that does not satisfy $\EHI(r \leq 1)$. As far as we could tell, this was previously an open question. Our counterexample is highly regular, and sheds light on what must go wrong in order for a subordinate Brownian motion to not satisfy $\EHI(r \leq 1)$. We also prove $\EHI$ for some subordinate Brownian motions for which it was previously unknown. 

Let $S=(S_t)_{t \geq 0}$ be a right-continuous, non-decreasing L\'{e}vy process on $[0, \infty)$ with $S_0=0$. We refer to such a process as a \emph{subordinator}. Associated with $S$ is a function $\phi : (0, \infty) \to (0, \infty)$ called the \emph{Laplace exponent}, which satisfies
\begin{equation*}
    \E \left[e^{-\lambda S_t} \right] = e^{-\phi(\lambda) t} \qquad\mbox{for all $\lambda>0$, $t>0$}.
\end{equation*}
The Laplace exponent is always of the form
\begin{equation} \label{EHI:e:driftAndLevy}
    \phi(\lambda) = \gamma\lambda + \int_{(0, \infty)} (1-e^{-\lambda t}) \, \mu(\mathrm{d}t)
\end{equation}
for some $\gamma>0$ (called the \emph{drift}) and a measure $\mu$ on $(0, \infty)$ (called the L\'{e}vy measure of $S$) such that
\begin{equation} \label{EHI:e:LevyKhintchineSubordinators}
    \int_{(0, \infty)} (1 \wedge t) \, \mu(\mathrm{d}t) < \infty.
\end{equation}
Equation \eqref{EHI:e:LevyKhintchineSubordinators} is called the L\'{e}vy-Khintchine condition for subordinators. The L\'{e}vy measure of $S$ also has a probabilistic interpretation: for all Lebesgue-measurable $A \subseteq (0, \infty)$, $\mu(A)$ is the rate at which $S$ takes jumps such that $S_t - S_{t-} \in A$.
If $0<\mu(A)<\infty$, then the set of times $t$ such that $S_t - S_{t-} \in A$ is a Poisson point process with intensity $\mu(A)$.

The Laplace exponent is always a \emph{Bernstein function}: $\phi \in C^\infty(0, \infty)$ with $\phi^{(k)} \geq 0$ for all $k \in \{0, 1, 3, 5, 7, 9, \dots\}$ and $\phi^{(k)} \leq 0$ for all $k \in \{2, 4, 6, 8, \dots\}$. Conversely, every Bernstein function is the Laplace exponent of some subordinator, and every $\gamma>0$ and $\mu$ satisfying the L\'{e}vy-Khintchine condition are the drift and L\'{e}vy measure of some subordinator.

The potential measure of a subordinator is defined as
\begin{equation*}
    U(A) := \int_0^\infty \P(S_t \in A) \dee{t}.
\end{equation*}

Let $B=(B_t)_{t \geq 0}$ be a Brownian motion on $\R^d$, independent of $S$, with twice the speed of a standard Brownian motion. We define a new process $X=(X_t)_{t \geq 0} := (B_{S_t})_{t \geq 0}$, and refer to $X$ as a \emph{subordinate Brownian motion} (with subordinator $S$). This new process $X$ is an isotropic unimodal L\'{e}vy process. If the drift $\gamma$ of $S$ is $0$, then $X$ is pure-jump. The jump kernel of $X$ is given by
\begin{equation} \label{EHI:e:JumpKernelIntermsofMu}
    j(r) = \int_{(0, \infty)} (4\pi t)^{-d/2} \exp(-\frac{r^2}{4t}) \mu(\mathrm{d}t).
\end{equation}

Kim and Mimica \cite{KM} prove that a large class of subordinate Brownian motions satisfy $\EHI(r \leq 1)$.

\begin{theorem}[Kim-Mimica]\label{EHI:t:KM}
    Let $X=(X_t)_{t \geq 0}=(B_{S_t})_{t \geq 0}$ be a subordinate Brownian motion, where $S$ is a subordinator and $B$ is a Brownian motion on $\R^d$, independent of $S$. Assume there exists a $C>0$ such that the jump kernel of $X$ satisfies
    \begin{equation*}
        j(r+1) \leq j(r) \leq C j(r+1) \qquad\mbox{for all $r>1$}.
    \end{equation*}
    Suppose also that the subordinator $S$ satisfies the three following conditions:
    \begin{enumerate}[label=A-\arabic*]
        \item\label{EHI:KM:A1} The potential measure of $S$ has decreasing density. In other words, there exists a decreasing function $u(t)$ such that $U(dt)=u(t) \dee{t}$.
        \item\label{EHI:KM:A2} The L\'{e}vy measure of $S$ is infinite (in other words, $\mu((0, \infty))=\infty$) and has a decreasing density.
        \item\label{EHI:KM:A3} There exist constants $\sigma>0$, $\lambda_0>0$, and $\delta \in (0, 1]$ such that
        \begin{equation*}
            \frac{\phi'(\lambda x)}{\phi'(\lambda)} \leq \sigma x^{-\delta} \qquad\mbox{for all $x \geq 1$ and $\lambda \geq \lambda_0$}.
        \end{equation*}
\end{enumerate}
    Then $X$ satisfies $\EHI(r \leq 1)$.
\end{theorem}

Theorem \ref{EHI:t:KM} can be used to show that each of the following subordinate Brownian motions satisfies $\EHI(r \leq 1)$:
\begin{itemize}
    \item Geometric stable processes:
    \begin{equation} \label{EHI:GeomStablelaplaceExp}
        \phi(\lambda) = \log(1+\lambda^{\beta/2}), \qquad 0 < \beta \leq 2.
    \end{equation}
    \item Iterated geometric stable processes:
    \begin{align}
        \phi_1(\lambda) &= \log(1+\lambda^{\beta/2}), \quad 0 < \beta \leq 2, \notag\\
        \phi_{n+1} &=\phi_n \circ \phi_1 \qquad\qquad\mbox{for all $n$}. \label{EHI:IteratedGeomStablelaplaceExp}
    \end{align}
    \item Relativistic geometric stable processes:
    \begin{equation} \label{EHI:RelGeomStablelaplaceExp}
        \phi(\lambda)=\log(1+\left( \lambda+m^{\beta/2} \right)^{2/\beta} -m) \qquad 0 < \beta \leq 2.
    \end{equation}
\end{itemize}
Note that condition \ref{EHI:KM:A3} implies $\gamma=0$ (by taking $x \to \infty$ and applying \eqref{EHI:e:driftAndLevy}), so Theorem \ref{EHI:t:KM} only applies to jump processes.

We say that two isotropic unimodal L\'{e}vy jump processes $X$ and $Y$, with jump kernels $j_X(r)$ and $j_Y(r)$ respectively, are \emph{bounded perturbations} of one another if there exist constants $C \geq c>0$ such that $cj_X(r) \leq j_Y(r) \leq C j_X(r)$ for all $r>0$. The results of \cite{KM} do not extend immediately to bounded perturbations of the above examples (geometric stable, iterated geometric stable, and relativistic geometric stable), but our results do.

\begin{example}\label{EHI:ex:geometric}
    Let $X$ be a bounded perturbation of the geometric stable process, with parameter $\beta \in (0, 2]$. 
    \v{S}iki\'{c}, Song, and Vondra\v{c}ek \cite[Theorem 3.4]{SSV} showed that  $j(r) \asymp r^{-d}$ as $r \to 0^+$. Thus, for small values $0<s\leq r$, we have
    \begin{equation*}
        \frac{r^{d+2}j(r)}{s^{d+2}j(s)} \asymp \left(\frac{r}{s}\right)^2,
    \end{equation*}
    so by Corollary \ref{EHI:c:mainNoM2}(a) (with $\alpha=2$), $X$ satisfies $\EHI(r \leq 1)$. If $\beta<2$, then by  \cite[Theorem 3.5]{SSV}, we also have $j(r) \asymp r^{-d-\beta}$ as $r \to \infty$, which means that 
    \begin{equation*}
        \frac{r^{d+2}j(r)}{s^{d+2}j(s)} \asymp \left(\frac{r}{s}\right)^{2-\beta}
    \end{equation*}
    for large $s<r$,
    so we also have $\EHI(r \geq 1)$ by Corollary \ref{EHI:c:mainNoM2}(b).
\end{example}

To show that bounded perturbations of the iterated geometric stable and relativistic geometric stable processes satisfy $\EHI(r \leq 1)$,  we use \cite[Proposition 4.2]{KM} to get the asymptotic form of $j(r)$ for $r \to 0^+$. The details of these calculations are shown in Section \ref{EHI:ss:IteratedRelativistic}.

\begin{example}\label{EHI:ex:relativistic}
    Let $X$ be a bounded perturbation of the relativistic geometric stable process. Then $j(r) \asymp r^{-d}$ as $r \to 0^+$, so $X$ satisfies $\EHI(r \leq 1)$.
\end{example}

For all $n$, let $\log^{(n)}$ denote the $n$-fold composition of $\log$.

\begin{example}\label{EHI:ex:iterated}
    Let $X$ be a bounded perturbation of the iterated geometric stable process, with $n$ iterations. Then 
    \begin{equation*}
        j(r) \asymp \frac{r^{-d}}{\log(r^{-2}) \log^{(2)}(r^{-2}) \cdots \log^{(n-1)}(r^{-2})} \qquad \mbox{as $r\to 0^+$}.
    \end{equation*}
     For large $s<r$,
     \begin{equation*}
         \frac{r^{d+2}j(r)}{s^{d+2}j(s)} \asymp \frac{r^2 \left(\log(r^{-2}) \right)^{-1} \left(\log^{(2)}(r^{-2}) \right)^{-1} \cdots \left(\log^{(n-1)}(r^{-2}) \right)^{-1}}{s^2 \left(\log(s^{-2}) \right)^{-1} \left(\log^{(2)}(s^{-2}) \right)^{-1} \cdots \left(\log^{(n-1)}(s^{-2}) \right)^{-1}} \geq \left(\frac{r}{s}\right)^2
     \end{equation*}
     so we have $\EHI(r \leq 1)$ by Corollary \ref{EHI:c:mainNoM2}(a).
\end{example}

For our last example, we identify a subordinate Brownian motion that does not satisfy condition \ref{EHI:KM:A3} of Kim and Mimica, but does satisfy $\EHI(r \leq 1)$.

\begin{example} \label{EHI:ex:NoA3example}
    Let $X=(X_t)_{t \geq 0} = (B_{S_t})_{t \geq 0}$ be a subordinate Brownian motion on $\R^d$ satisfying \eqref{EHI:e:regularJumps}, such that the subordinator $S=(S_t)_{t \geq 0}$ has drift $0$ and a L\'{e}vy measure satisfying $\mu(dt) \asymp t^{-2} \left(\log(t^{-1}) \right)^{-2} \dee{t}$ as $t \to 0^+$.
    In Section \ref{EHI:ss:NoA3example}, we show that $S$ has Laplace exponent $\phi(\lambda) \gtrsim \lambda/\log\lambda$, which implies that $X$ can not satisfy \ref{EHI:KM:A3}, and we calculate that $X$ has jump kernel
    \begin{equation*}
        j(r) \asymp r^{-d-2} \left(\log(r^{-1}) \right)^{-2} \qquad\mbox{as $r \to 0^+$}.
    \end{equation*}
    This is a special case of Example \ref{EHI:ex:GeneralExampleSmallR}, so $X$ satisfies $\EHI(r \leq 1)$.
\end{example}

\subsection{Counterexample and negative result}

In light of Theorem \ref{EHI:t:main}, its corollaries, and the above examples, a natural question to ask is whether every isotropic unimodal L\'{e}vy jump process on $\R^d$ that satisfies \eqref{EHI:e:regularJumps} also satisfies $\EHI(r \leq 1)$. Another question that has been open until this work (and was the original motivator that led to the research in this chapter) is whether every subordinate Brownian motion satisfies $\EHI(r \leq 1)$. We show that the answer to both of these questions is ``no," as the following process is a counterexample to both.

\begin{example}\label{EHI:ex:counterexample}
    Let $S=(S_t)_{t \geq 0}$ be a subordinator with drift $0$ and L\'{e}vy measure
    \begin{equation*}
        \mu(\mathrm{d}t) = \frac23 \sum_{n=1}^\infty 2^{3n^2} \min\left\{ 1, \left(2^{2n^2} t \right)^{-3}\right\} \dee{t}.
    \end{equation*}
    Let $X=(X_t)_{t \geq 0} = (B_{S_t})_{t \geq 0}$ be a subordinate Brownian motion with subordinator $S$.
    Let $\mu(t) := \frac23 \sum_{n=1}^\infty 2^{3n^2} \min\left\{ 1, \left(2^{-2n^2} t \right)^{-3}\right\}$, so that $\mu(\mathrm{d}t) = \mu(t)\dee{t}$. Note that $\mu(t) \asymp \mu(2t)$, since $\mu$ is the sum of infinitely many non-increasing functions that decay no faster than $t^{-3}$. From this, it is easy to show that the jump kernel of $X$ satisfies \eqref{EHI:e:regularJumps}. However, we show in Section \ref{EHI:ss:Counterexample} that $X$ does not satisfy $\EHI(r \leq 1)$.
\end{example}

We also prove the following general negative result.

\begin{theorem}\label{EHI:t:mainNegative}
    Let $X=(X_t)_{t \geq 0}$ be an isotropic unimodal L\'{e}vy jujmp process on $\R^d$ with jump kernel $j(r)$. Suppose there exists a sequence $(r_n) \subseteq (0, \infty)$ such that $r_n \to 0^+$ and
    \begin{equation} \label{EHI:e:m2intermedCondition}
        r_n^2 \lambda(r_n) \lesssim m_2(r_n) \ll (r_n^{d+2} j(r_n))^{-\frac{2}{d}} (r_n^2 \lambda(r_n))^{\frac{d+2}{2}}.
    \end{equation}
    Then $X$ does not satisfy $\EHI(r\leq 1)$. Similarly, if there exists a sequence $(r_n) \subseteq (0, \infty)$ such that $r_n \to \infty$ and \eqref{EHI:e:m2intermedCondition} holds, then $X$ does not satisfy $\EHI(r \geq 1)$.
\end{theorem}

Note that Theorem \ref{EHI:t:mainNegative} does not require \eqref{EHI:e:regularJumps}.

\section{\texorpdfstring{Green's function and Poisson kernel}{TEXT}}
 \label{EHI:s:GreenPoisson}

In this section, we introduce the Green's function and Poisson kernel of $X$, and then give an equivalent formulation of the condition $\EHI$ in terms of the Poisson kernel, which we will use to prove our main results.

Let $p_t(x, y)$ be the heat kernel of $X$.
For all $t>0$, $p_t(\cdot, \cdot)$ is a symmetric function such that $p_t(x, \cdot)$ is the density of $X_t$ given an initial value of $X_0=x$. For all open sets $D \subseteq \mathbb{R}^d$, let $p^D_t(x, y)$ be the heat kernel of the process killed upon exiting $D$. Then
\begin{align} \label{EHI:e:HeatKernelDef}
\begin{split}
    \mathbb{P}_x(X_t \in A)=\int_A p_t(x, y) \, dy;\\
    \mathbb{P}_x(X_t \in A, \tau_D>t)=\int_A p^D_t(x, y) \, dy
\end{split}
\end{align}
for all $x \in \mathbb{R}^d$, all $t>0$, and all measurable $A$. The killed heat kernel $p^D_t$ is also symmetric. Because $X$ is isotropic unimodal, it can be shown that $p_t$ and $p^D_t$ are jointly continuous and isotropic unimodal.

Let us define the Green's functions $G$ and $G_D$, as
\begin{equation} \label{EHI:e:GreenDef}
    G(x, y) := \int_0^\infty p_t(x, y) \, dt, \qquad G_D(x, y) := \int_0^\infty p^D_t(x, y) \, dt.
\end{equation}
The Green's function has a probabilistic interpretation. For all measurable $A \subseteq D$, if the process starts at $X_0=x$, then $\int_A G_D(x, y) \, dy$ is the expected amount of time the process spends in $A$ before time $\tau_D$.

By a result of Ikeda and Watanabe \cite[Theorem 1]{IW}, if $D$ is a measurable set, then for any $x \in D$ and any function $h$ that is harmonic on $D$,
\begin{equation*}
    h(x) = \mathbb{E}_x[h(X_{\tau_D})] = \int_{z \in D} \int_{w \in D^c} G_D(x, z) j(|w-z|) h(w) \, dw \, dz.
\end{equation*}
By Tonelli's theorem, we can switch the order of integration. The result is
\begin{equation} \label{EHI:e:h(x)IntermsofKd}
    h(x) = \int_{w \in D^c} K_D(x, w) h(w) \, dw,
\end{equation}
where $K_D$ is the Poisson kernel, defined by
\begin{equation} \label{EHI:e:DefKd}
    K_D(x, w) := \int_{z \in D} G_D(x, z) j(|w-z|) \, dz.
\end{equation}

\begin{definition}[Characterization of $\EHI$ in terms of Poisson kernel]\label{EHI:d:ehiPoisson}
    Let $S \subseteq (0, \infty)$.
    We say that $X$ satisfies $\EHIPoisson(r \in S)$ if there exist constants $C \geq 1$ and $\kappa \in (0, 1)$ such that for all $r \in S$, $x \in B(0, \kappa r)$, and $w \in B(0, r)^c$, we have
    \begin{equation*}
        C^{-1} K_{B(0, r)}(-x, w) \leq K_{B(0, r)}(x, w) \leq C K_{B(0, r)}(-x, w).
    \end{equation*}
\end{definition}

\begin{prop} \label{EHI:p:EHIpoisson=EHI}
    For all $S \subseteq (0, \infty)$, the conditions $\EHI(r \in S)$ and $\EHIPoisson(r \in S)$ are equivalent.
\end{prop}

In order to prove Proposition \ref{EHI:p:EHIpoisson=EHI}, we define a list of related conditions, all of which we show are equivalent, which ``bridge the gap" between $\EHI(r \in S)$ and $\EHIPoisson(r \in S)$. To make the proof more readable, we split this proof into several lemmas, with a separate lemma for each implication.

\begin{definition} \label{EHI:d:EquivalentEhis}
    Let $S \subseteq (0, \infty)$.
    \begin{itemize}
        \item We say that $X$ satisfies $\EHIp(r \in S)$ if there exist constants $C \geq 1$ and $\kappa \in (0, 1)$ such that for all $x_0 \in \mathbb{R}^d$ and $r \in S$, if a function $h$ is non-negative everywhere and harmonic with respect to $X$ on $B(x_0, r)$, then
    \begin{equation*}
        C^{-1} h(y) \leq h(x) \leq C h(y)
    \end{equation*}
    for all $x, y \in B(x_0, \kappa r)$ such that $x_0$ is the midpoint between $x$ and $y$.
    \item We say that $X$ satisfies $\EHIpp(r \in S)$ if there exist constants $C \geq 1$ and $\kappa \in (0, 1)$ such that for all $r \in S$, if a function $h$ is non-negative everywhere and harmonic with respect to $X$ on $B(0, r)$, then
    \begin{equation*}
        C^{-1} h(-x) \leq h(x) \leq C h(-x)
    \end{equation*}
    for all $x \in B(0, \kappa r)$.
    \end{itemize}
\end{definition}

\begin{remark}\label{EHI:r:easyPartsOfEHI=EHIpoisson}
    Clearly, $\EHI(r \in S) \Longrightarrow \EHIp(r \in S) \Longrightarrow \EHIpp(r \in S)$, since each condition along this chain requires \eqref{EHI:e:ehiDef} to hold on a smaller set of values $(x_0, x, y)$ than the condition preceding it. Since $X$ is an isotropic unimodal L\'{e}vy process, it is also easy to show that $\EHIpp(r \in S) \Longrightarrow \EHIp(r \in S)$ by shifting the argument of the function $h$. Therefore, we already have
    \begin{equation*}
        \EHI(r \in S) \Longrightarrow \EHIp(r \in S) \Longleftrightarrow \EHIpp(r \in S).
    \end{equation*}
\end{remark}

\begin{lemma} \label{EHI:l:EHI'impliesEHI}
    For all $S \subseteq (0, \infty)$, $\EHIp(r \in S) \Longrightarrow \EHI(r \in S)$.
\end{lemma}

\begin{proof}
    Suppose $X$ satisfies $\EHIp(r \in S)$. Let $C$ and $\kappa$ be the constants for which $\EHIp(r \in S)$ holds.
    Fix $x_0 \in \R^d$ and $r \in S$, and let $h$ be non-negative everywhere and harmonic on $B(x_0, r)$. Let $x, y \in B(x_0, \frac{\kappa}{1+\kappa}r)$. Let $x_1 := \frac{x+y}{2} \in B(x_0, \frac{\kappa}{1+\kappa}r)$. Then $x$ and $y$ belong to $B(x_1, \frac{\kappa}{1+\kappa}r)$, since $$|x-x_1|=|y-x_1| = \frac12|x-y| \leq \frac12(|x-x_0|+|y-x_0|) < \frac{\kappa}{1+\kappa}r.$$ Also, $B(x_1, \frac{1}{1+\kappa}r) \subseteq B(x_0, r)$, because for all $z \in B(x_1, \frac{1}{1+\kappa}r)$ the triangle inequality tells us that $$|z-x_0| \leq |z-x_1|+|x_1-x_0| < \frac{1}{1+\kappa}r + \frac{\kappa}{1+\kappa}r = r.$$ Therefore, $h$ is non-negative and harmonic on $B(x_1, \frac{1}{1+\kappa}r)$.
    Since $x$ and $y$ belong to $B(x_1, \frac{\kappa}{1+\kappa}r)$, $x_1$ is their midpoint, and $h$ is non-negative and harmonic on $B(x_1, \frac{1}{1+\kappa}r)$, by $\EHI'(r \in S)$, we have
    \begin{equation*}
        C^{-1} h(y) \leq h(x) \leq C h(y).
    \end{equation*}
    Therefore, $X$ satisfies $\EHI(r \in S)$ with $\frac{\kappa}{1+\kappa}$ in place of $\kappa$.
\end{proof}

We have shown that $\EHI(r \in S)$, $\EHIp(r \in S)$, and $\EHIpp(r \in S)$ are equivalent. In order to complete the proof of Proposition \ref{EHI:p:EHIpoisson=EHI}, we must show that $\EHIPoisson(r \in S)$ is equivalent to these three conditions.

\begin{lemma}\label{EHI:l:EHIpoissonImpliesEHI''}
    For all $S \subseteq (0, \infty)$, $\EHIPoisson(r \in S) \Longrightarrow \EHIpp(r \in S)$.
\end{lemma}

\begin{proof}
    Suppose $X$ satisfies $\EHIPoisson(r \in S)$, and let $C$ and $\kappa$ be the constants witnessing $\EHIPoisson(r \in S)$. Fix $r \in S$, let $h$ be non-negative everywhere and harmonic on $B(0, r)$, and let $x \in B(0, \kappa r)$. By \eqref{EHI:e:h(x)IntermsofKd} and $\EHIPoisson(r \in S)$,
    \begin{align*}
        C^{-1} h(-x) &= C^{-1}\int_{w \in B(0, r)^c} K_{B(0, r)}(-x, w) h(w) \dee{w} \\
        &\leq \int_{w \in B(0, r)} K_{B(0, r)}(x, w) h(w) \dee{w} = h(x) \\
        &\leq C\int_{w \in B(0, r)^c} K_{B(0, r)}(-x, w) h(w) \dee{w} = Ch(-x).
    \end{align*}
    Thus, $X$ satisfies $\EHIpp(r \in S)$.
\end{proof}

The final step in the proof of Proposition \ref{EHI:p:EHIpoisson=EHI} is to show that $\EHIpp(r \in S) \Longrightarrow \EHIPoisson(r \in S)$. Note that the function $K_{B(0, r)}(\cdot, w)$ is not harmonic on $B(0, r)$, like it would be if $X$ were a diffusion, so we can not simply apply $\EHI'(r \in S)$ to the function $K_{B(0, r)}(\cdot, w)$. Instead, we will take limits as $s \to 0^+$ of the function $x \mapsto s^{-1}\int_{B(w, s)} K_{B(0, r)}(x, w) \dee{s} = s^{-1} \P_x(X_{\tau_{B(0, r)}} \in B(w, s))$, which truly is harmonic.

\begin{lemma} \label{EHI:l:EHI''impliesEHIpoisson}
    For all $S \subseteq (0, \infty)$, $\EHIpp(r \in S) \Longrightarrow \EHIPoisson(r \in S)$.
\end{lemma}

\begin{proof}
    Suppose $X$ satisfies $\EHIpp(r \in S)$, and let $C$ and $\kappa$ be the constants witnessing this. Fix $r \in S$, $x \in B(0, \kappa r)$, and $w \in \overline{B(0, r)}^c$. Let $s>0$ be small enough so that $B(w, s) \subseteq B(0, r)^c$. Let $h_s$ be the function
    \begin{equation*}
        h_s(z) := \frac{1}{|B(w, s)|} \int_{B(w, s)} K_{B(0, r)}(z, v) \dee{v} = \frac{1}{|B(w, s)|} \P_z \left( X_{\tau_{B(0, r)}} \in B(w, s) \right).
    \end{equation*}
    Because $X$ is isotropic unimodal, the Green's functions and jump kernels are all continuous, so $K_{B(0, r)}$ is continuous in each argument, and therefore $\lim_{s \to 0^+} h_s(z) = K_{B(0, r)}(z, w)$.
    Since $h_s$ is non-negative everywhere and harmonic on $B(0, r)$, by applying $\EHIpp(r \in S)$ to $h_s$, we obtain
    \begin{align*}
        C^{-1} h_s(-x) \leq h_s(x) \leq Ch_s(-x, w).
    \end{align*}
    By taking the limit as $s \to 0^+$,
    \begin{equation*}
        C^{-1} K_{B(0, r)}(-x, w) \leq K_{B(0, r)}(x, w) \leq C K_{B(0, r)}(-x, w).
    \end{equation*}
    This holds for all $x \in B(0, \kappa r)$ and $w \in \overline{B(0, r)}^c$. By the continuity of $K_{B(0, r)}$ in each argument, the result extends to all $x \in B(0, \kappa r)$ and $w \in B(0, r)^c$, so $X$ satisfies $\EHIPoisson(r \in S)$.
\end{proof}

\begin{proof}[Proof of Proposition \ref{EHI:p:EHIpoisson=EHI}]
    By Remark \ref{EHI:r:easyPartsOfEHI=EHIpoisson} and Lemma \ref{EHI:l:EHI'impliesEHI}, $\EHI(r \in S) \Longleftrightarrow \EHIp(r \in S) \Longleftrightarrow \EHIpp(r \in S)$. By Lemmas \ref{EHI:l:EHIpoissonImpliesEHI''} and \ref{EHI:l:EHI''impliesEHIpoisson}, $\EHIpp(r \in S) \Longleftrightarrow \EHIPoisson(r \in S)$. Thus, all four conditions are equivalent.
\end{proof}

\begin{prop} \label{EHI:p:EHI[a,b]forfree}
    Let $X$ be an isotropic unimodal L\'{e}vy jump process on $\R^d$. If \eqref{EHI:e:regularJumps} holds for all $r>0$, then $X$ satisfies $\EHI(r \in [a, b])$ for all $0<a<b<\infty$.
\end{prop}

\begin{proof}
    Fix $b>a>0$.
    Consider the function
    \begin{equation*}
        f(r, x, w) := \frac{K_{B(0, r)}(-x, w)}{K_{B(0, r)}(x, w)}
    \end{equation*}
    with domain $\{(r, x, w) : r>0, x \in \overline{B(0, r/2)}, w \in B(0, r)^c \}$. Note that $f$ is continuous (since $X$ is an isotropic unimodal L\'{e}vy process), and only takes positive values. Therefore, on any compact subset of the domain, $f$ has a finite upper bound. In particular, there exists a $C>0$ such that
    \begin{equation}\label{EHI:e:compactnessArgument}
        f(r, x, w) \leq C \qquad\mbox{for all $r \in [a, b]$, $|x| \leq r/2$, and $r \leq |w| \leq 3r$.}
    \end{equation}
    For all $w \in B(0, r)^c$, by \eqref{EHI:e:DefKd}, relabeling $z$ as $-z$, and symmetry (since $X$ is isotropic unimodal), we have
    \begin{align} \label{EHI:e:WFar}
        K_{B(0, r)}(-x, w) &= \int_{z \in B(0, r)} G_{B(0, r)}(-x, z) j(|w-z|) \dee{z} \notag\\
        &= \int_{z \in B(0, r)} G_{B(0, r)}(-x, -z) j(|w+z|) \dee{z} \notag\\
        &= \int_{z \in B(0, r)} G_{B(0, r)}(x, z) j(|-w-z|) \dee{z}.
    \end{align}
    If $|w| \geq 3r$, then $\frac12 \leq \frac{|-w-z|}{|w-z|} \leq 2$ for all $z \in B(0, r)$, so \eqref{EHI:e:regularJumps} gives us
    \begin{equation*}
        c_j \leq \frac{j(|-w-z|)}{j(|w-z|)} \leq c_j^{-1}.
    \end{equation*}
    Thus, by comparing the $j(|-w-z|)$ term in \eqref{EHI:e:WFar} with $j(|w-z|)$,
    \begin{equation} \label{EHI:e:WFar2}
        K_{B(0, r)}(-x, w) \leq c_j^{-1} \int_{z \in B(0, r)} G_{B(0, r)}(x, z) j(|w-z|) \dee{z} = c_j^{-1} K_{B(0, r)}(x, w).
    \end{equation}
    By \eqref{EHI:e:compactnessArgument} and \eqref{EHI:e:WFar2},
    \begin{equation*}
        K_{B(0, r)}(-x, w) \leq \max\{C, c_j^{-1} \} K_{B(0, r)}(x, w).
    \end{equation*}
    By the same argument,
    \begin{equation*}
        K_{B(0, r)}(x, w) \leq \max\{C, c_j^{-1} \} K_{B(0, r)}(-x, w).
    \end{equation*}
    This holds for all $r \in [a, b]$, $x \in \overline{B(0, r/2)}$, and $w \in B(0, r)^c$. Thus, $X$ satisfies $\EHIPoisson(r \in [a, b])$. By Proposition \ref{EHI:p:EHIpoisson=EHI}, this is equivalent to $X$ satisfying $\EHI(r \in [a, b])$.
\end{proof}

\section{Meyer decompositions}\label{EHI:s:meyer}

One of the key ideas of this chapter (for both the positive results and the negative results) is to take a jump process $X$ with jump kernel $j(r)$, and to write $j(r)$ as the sum of two other jump kernels, $j'(r)$ and $\hat{j}'(r)$.
(The apostrophes will be replaced with a superscript indicating which particular decomposition we are considering.)
Usually this will be done in a way so that $j'(r)$ tends to admit smaller jumps, while $\hat{j}'(r)$ tends to admit larger jumps, and $\int_{\mathbb{R}^d} \hat{j}'(|x|) \dee{x} < \infty$ so $\hat{j}'(r)$ only admits finitely many jumps per time-interval. Then $X$ can be thought of as the sum of two independent jump processes $X'$ and $\hat{X}'$, with jump processes $j'(r)$ and $\hat{j}'(r)$ respectively. Such a decomposition is known as a \emph{Meyer decomposition} (see \cite{meyer}, \cite{ms3}). Let $T'$ be the first time the process $\hat{X}'$ with jump kernel $\hat{j}'(r)$ takes a jump. Then $T'$ is an Exponential\footnote{A random variable is said to be Exponential($\lambda$) (or have Exponential($\lambda$) distribution) if it takes non-negative values and has a density of $\lambda e^{-\lambda x}$ for $x \geq 0$. The mean of such a random variable is $\lambda^{-1}$.}($\int_{\R^d} \hat{j}'(|x|) \dee{x}$) stopping time, independent of $X'$. Because of this independence, we can think of the process $X$ as behaving like $X'$ (which might be easier to analyze because it takes shorter jumps) until time $T'$, which often has either a ``flattening" or ``scattering" property: either the jump at time $T'$ takes the process to a point that is near uniform on some large ball (flattening), or the jump at time $T'$ takes the process somewhere very far away from the scale of our analysis (scattering).

In this section, we give definitions for the various Meyer decompositions used throughout this chapter, and prove some of their properties.

\subsection{The zoo of Meyer decompositions}\label{ss:zooOfMeyer}

Let us define some of the decompositions we will use, and introduce some notation related to them.

\begin{definition}[Small/large decomposition]\label{EHI:d:smalllarge}
    Let $X=(X_t)_{t \geq 0}$ be an isotropic unimodal L\'{e}vy jujmp process on $\R^d$ with jump kernel $j(r)$. For each $r_0>0$, we define the \emph{small/large decomposition} of $X$ at $r_0$ as follows: For all $r>0$, let
    \begin{equation*}
        j^{(r_0)}(r) := j(r) \indicatorWithSetBrackets{r \leq r_0}, \qquad \hat{j}^{(r_0)}(r) := j(r) \indicatorWithSetBrackets{r>r_0}.
    \end{equation*}
    Let $X = X^{(r_0)} + \hat{X}^{(r_0)}$ be a Meyer decomposition of $X$, such that $X^{(r_0)}_0 = X_0$, $\hat{X}^{(r_0)}_0=0$, $X^{(r_0)}$ has jump kernel $j^{(r_0)}(r)$, $\hat{X}^{(r_0)}$ has jump kernel $\hat{j}^{(r_0)}(r)$, and $X^{(r_0)}$ and $\hat{X}^{(r_0)}$ are independent.
    Let
    \begin{equation}\label{EHI:e:m2lambda()}
        m_2(r_0) := \int_{\R^d} |x|^2 j^{(r_0)}(|x|) \dee{x} \qquad \mbox{and} \qquad \lambda(r_0) := \int_{\R^d} \hat{j}^{(r_0)}(|x|) \dee{x}.
    \end{equation}
    By the L\'{e}vy-Khintchine condition, $\lambda(r_0)<\infty$.
    Let $T^{(r_0)} := \inf\{ t>0 : \hat{X}^{(r_0)}_t \neq 0 \}$ be the first time that the process $\hat{X}^{(r_0)}$ takes a jump, and note that $T^{(r_0)}$ is Exponential($\lambda(r_0)$) and independent of the process $X^{(r_0)}$.

    Denote exit times for the process $X^{(r)}$ by $\tau^{(r)}_U := \inf \{t \geq 0 : X^{(r)}_t \notin U \}$.
\end{definition}

\begin{figure}[ht] 
    \begin{framed}
        \captionof{figure}{An illustration of the small/large and small/flat decompositions}
        \label{EHI:f:decompositions}
        \centering
        \begin{minipage}{0.45\textwidth}
        \centering
        \begin{tikzpicture}
            \begin{axis}[
                axis lines = middle,
                domain = 0.05:2, 
                ymin = 0, ymax = 2, 
                xtick = \empty,
                ytick = \empty,
                extra x ticks = {1},
                extra x tick labels = {\,}, 
                extra y ticks = {1},
                extra y tick labels = {$j(r_0)$},
                width = 6cm, height = 6cm, 
                title = {\textbf{Small/large decomposition}}, 
                samples = 100, 
                restrict y to domain = 0:2, 
                enlargelimits,
                clip=false,
                every axis label/.append style={font=\small}, 
                every axis tick label/.append style={font=\small}, 
                axis line style={-}, 
                only marks=false, 
                unbounded coords=jump, 
                xlabel = \( r \), 
                ylabel = \( j(r) \), 
                xlabel style={at={(axis cs:2,0)}, anchor=north west}, 
                ylabel style={at={(axis cs:0.3,2)}, anchor=west},
            ]
                \addplot[domain=0.05:2, smooth] {1/x^2}; 
                
                \draw[dotted, thick] (axis cs:1,0) -- (axis cs:1,1);
                \node at (axis cs:1,0) [anchor=north] {$r_0$}; 
                \node at (axis cs:0.8, 0.5) [font=\small] {small};
                \node at (axis cs:1.3, 0.3) [font=\small] {large};
            \end{axis}
        \end{tikzpicture}
    \end{minipage}%
    \hspace{0.5cm} 
    \begin{minipage}{0.45\textwidth}
        \centering
        \begin{tikzpicture}
            \begin{axis}[
                axis lines = middle,
                domain = 0.05:2, 
                ymin = 0, ymax = 2, 
                xtick = \empty,
                ytick = \empty,
                extra x ticks = {1},
                extra x tick labels = {\,}, 
                extra y ticks = {1, 0.5},
                extra y tick labels = {$j(r_0)$, $\frac12 j(r_0)$},
                width = 6cm, height = 6cm, 
                title = {\textbf{Small/flat decomposition}}, 
                samples = 100, 
                restrict y to domain = 0:2, 
                enlargelimits,
                clip=false,
                every axis label/.append style={font=\small}, 
                every axis tick label/.append style={font=\small}, 
                axis line style={-}, 
                only marks=false, 
                unbounded coords=jump, 
                xlabel = \( r \), 
                ylabel = \( j(r) \), 
                xlabel style={at={(axis cs:2,0)}, anchor=north west}, 
                ylabel style={at={(axis cs:0.3,2)}, anchor=west},
            ]
                \addplot[domain=0.05:2, smooth] {1/x^2}; 
                
                \draw[dotted, thick] (axis cs:1,0.5) -- (axis cs:1,1);
                \draw[dotted, thick] (axis cs:0.63, 0.5) -- (axis cs:1, 0.5);
                \node at (axis cs:1,0) [anchor=north] {$r_0$}; 
                \node at (axis cs:0.8, 0.8) [font=\small] {small};
                \node at (axis cs:1.3, 0.3) [font=\small] {flat};
            \end{axis}
        \end{tikzpicture}
    \end{minipage}
    \end{framed}
\end{figure}
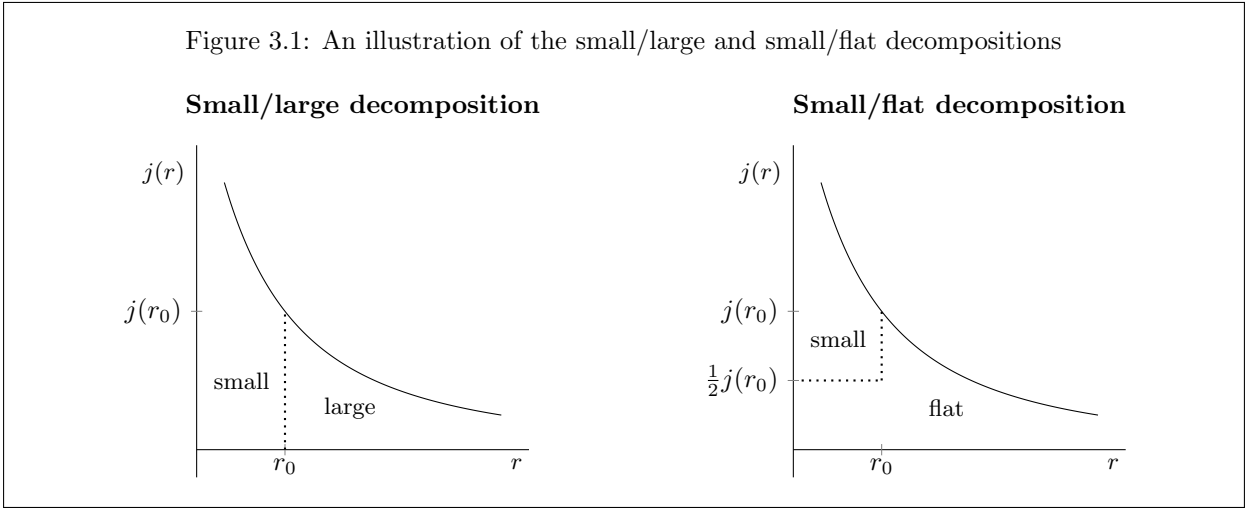

\begin{definition}[Small/flat decomposition] \label{EHI:d:smallflat}
    Let $X=(X_t)_{t \geq 0}$ be an isotropic unimodal L\'{e}vy jujmp process on $\R^d$ with jump kernel $j(r)$. For each $r_0>0$, we define the \emph{small/flat decomposition} of $X$ at $r_0$ as follows: For all $r>0$, let
    \begin{equation*}
        j^{\inner{r_0}}(r) := \left\{\begin{matrix}
            j(r) - \frac12 j(r_0) &:& \mbox{if $r \leq r_0$} \\
            0 &:& \mbox{if $r>r_0$}
        \end{matrix}
        \right.
        , \qquad
        \hat{j}^{\inner{r_0}}(r) := \left\{
        \begin{matrix}
            \frac12 j(r_0) &:& \mbox{if $r \leq r_0$} \\
            j(r) &:& \mbox{if $r>r_0$}
        \end{matrix}
        \right..
    \end{equation*}
    Let $X = X^{\inner{r_0}} + \hat{X}^{\inner{r_0}}$ be a Meyer decomposition of $X$, such that $X^{\inner{r_0}}_0 = X_0$, $\hat{X}^{\inner{r_0}}_0=0$, $X^{\inner{r_0}}$ has jump kernel $j^{\inner{r_0}}(r)$, $\hat{X}^{\inner{r_0}}$ has jump kernel $\hat{j}^{\inner{r_0}}(r)$, and $X^{\inner{r_0}}$ and $\hat{X}^{\inner{r_0}}$ are independent.
    Let
    \begin{equation}\label{EHI:e:m2lambda<>}
        m_2\inner{r_0} := \int_{\R^d} j^{\inner{r_0}} |x|^2 j(|x|) \dee{x}
        \qquad\mbox{and}\qquad
        \lambda\inner{r_0} := \int_{\R^d} \hat{j}^{\inner{r_0}}(|x|) \dee{x}.
    \end{equation}
    By the L\'{e}vy-Khintchine condition, $\lambda\inner{r_0}<\infty$.
    Let $T^{\inner{r_0}} := \inf\{ t>0 : \hat{X}^{\inner{r_0}}_t \neq 0 \}$ be the first time that the process $\hat{X}^{\inner{r_0}}$ takes a jump, and note that $T^{\inner{r_0}}$ is Exponential($\lambda\inner{r_0}$) and independent of the process $X^{\inner{r_0}}$.

    Denote exit times for the process $X^{\inner{r}}$ by $\tau^{\inner{r}}_U := \inf \{t \geq 0 : X^{\inner{r}}_t \notin U \}$.
\end{definition}

Let us briefly give an intuition behind the names we have given these decompositions. In the small/large decomposition, $X^{(r_0)}$ only takes jump of magnitude $r_0$ or smaller, while $\hat{X}^{(r_0)}$ only takes jumps of magnitude larger than $r_0$. In the small/flat decomposition, $X^{\inner{r_0}}$ only takes jumps of magnitude $r_0$ or smaller, and $\hat{j}^{\inner{r_0}}$ is ``flat" in the sense that $\hat{j}^{\inner{r_0}}(r_1) \asymp \hat{j}^{\inner{r_0}}(r_2)$ for all $r_1, r_2 \lesssim r_0$. These decompositions each have their respective advantages: the small/large decomposition is more straightforward and intuitive, while the flattening property of the small/flat decomposition is an essential ingredient in our proof of Theorem \ref{EHI:t:main}.

Note that the definitions of $m_2(r_0)$ and $\lambda(r_0)$ given in \eqref{EHI:e:m2lambda()} are consistent with those from \eqref{EHI:e:m2def}-\eqref{EHI:e:lambdadef}. We only express them in this additional form to emphasize the analogy to $m_2\inner{r_0}$ and $\lambda\inner{r_0}$ in \eqref{EHI:e:m2lambda<>}.
The reader may wonder why the factor of $1/2$ appears in the definition of the small/flat decomposition. As we prove later in Lemma \ref{EHI:l:m2lambdaComparable}, defining the small/flat decomposition this way gives us the convenient property that $m_2\inner{r} \asymp m_2(r)$ and $\lambda\inner{r} \asymp \lambda(r)$ when $X$ satisfies \eqref{EHI:e:regularJumps}.

\begin{definition}[SBM decomposition]\label{EHI:d:sbmDecomposition}
    Let $X=(X_t)_{t \geq 0} = (B_{S_t})_{t \geq 0}$ be a subordinated Brownian motion, and let $\mu$ be the L\'{e}vy measure of its subordinator $S$. For all $s>0$, define the \emph{SBM decomposition} of $X$ at $s$ as follows: Let $\mu^{\brackets{s}}$ and $\hat{\mu}^{\brackets{s}}$ be the measures given by
    \begin{equation*}
        \mu^{\brackets{s}}(A) := \mu(A \cap (0, s]), \quad \hat{\mu}^{\brackets{s}}(A) := \mu(A \cap (s, \infty)) \qquad\mbox{for all Lebesgue-measurable $A \subseteq (0, \infty)$}.
    \end{equation*}
    Let $S=S^{\brackets{s}}+\hat{S}^{\brackets{s}}$ be a Meyer decomposition of $S$ such that $S^{\brackets{s}}_0 = \hat{S}^{\brackets{s}}_0 = 0$, $S^{\brackets{s}}$ has L\'{e}vy measure $\mu^{\brackets{s}}$, $\hat{S}^{\brackets{s}}$ has L\'{e}vy measure $\hat{\mu}^{\brackets{s}}$, and $S^{\brackets{s}}$ and $\hat{S}^{\brackets{s}}$ are independent. Let $X^{\brackets{s}} = (X^{\brackets{s}}_t)_{t \geq 0} = B(S^{\brackets{s}}_t)_{t \geq 0}$ and $\hat{X}^{\brackets{s}} = (\hat{X}^{\brackets{s}}_t)_{t \geq 0} = \left(B(\hat{S}^{\brackets{s}}_t) - B_0 \right)_{t \geq 0}$.
    Then $X=X^{\brackets{s}}+\hat{X}^{\brackets{s}}$ is a Meyer decompposition of $X$, $X^{\brackets{s}}_0=X_0$, $\hat{X}^{\brackets{s}}_0=0$, and $X^{\brackets{s}}$ and $\hat{X}^{\brackets{s}}$ are independent.

    Let $j^{\brackets{s}}$ and $\hat{j}^{\brackets{s}}$ denote the jump kernels of $X^{\brackets{s}}$ and $\hat{X}^{\brackets{s}}$ respectively. Let
    \begin{equation*}
        \lambda\brackets{s} := \hat{\mu}^{\brackets{s}}((s, \infty)).
    \end{equation*}
    By the L\'{e}vy-Khintchine condition, $\lambda\brackets{s} < \infty$. Let $T^{\brackets{s}} := \inf\{ t\geq 0 : \hat{X}^{\brackets{s}} \neq 0 \}$ be the first time that the process $\hat{X}^{\brackets{s}}$ takes a jump, and note that $T^{\brackets{s}}$ is Exponential($\lambda\brackets{s}$) and independent of the process $X^{\brackets{s}}$.

    Denote exit times for the process $X^{\brackets{s}}$ by $\tau^{\brackets{s}}_U := \inf\{ t \geq 0 : X^{\brackets{s}}_t \notin U \}$.

    \begin{definition}[Meyer-decomposition arising exponential stopping time]\label{EHI:d:meyerStoppingTimes}
        If $X=X'+\hat{X}'$ is a Meyer decomposition of $X$, where the processes $X'$ and $\hat{X}'$ have jump kernels $j'$ and $\hat{j}'$ respectively, with $\int_{\R^d} \hat{j}'(|x|) \dee{x} < \infty$, and $T'$ is defined as the first time that the process $\hat{X}'$ takes a jump, then we call $T'$ a \emph{Meyer-decomposition arising exponential stopping time}.
    \end{definition}
\end{definition}
For example, $T^{(r_0)}$, $T^{\inner{r_0}}$, and $T^{\brackets{s}}$ (if $X$ is a subordinate Brownian motion) are all Meyer-decomposition arising exponential stopping times.

\subsection{Decompositions of the Green's function and Poisson kernel}

For the small/flat decomposition, the sequence of ``flattening jumps" (in other words, the sequence of times that $\hat{X}^{\inner{r_0}}$ takes a jump) will be of interest for the proof of Theorem \ref{EHI:t:main}. The following definition gives us notation with which to refer to each of these jumps, as opposed to only the first such jump.

\begin{definition}
    Let $X=(X_t)_{t \geq 0}$ be an isotropic unimodal L\'{e}vy jump process on $\R^d$, and let $X=X^{\inner{r_0}}+\hat{X}^{\inner{r_0}}$ be the small/flat decomposition of $X$ at $r_0$, as defined in Definition \ref{EHI:d:smallflat}. Let $\left(T^{\inner{r_0}}_k \right)_{k=0}^\infty$ be the sequence of stopping times defined recursively by
    \begin{equation*}
        T_0 := 0, \qquad T_{k+1} = \inf \left\{ t>T_k : \hat{X}^{\inner{r_0}}_t \neq \hat{X}^{\inner{r_0}}_{t-} \right\}.
    \end{equation*}
\end{definition}
In other words, for all $k \geq 1$, $T^{\inner{r_0}}_k$ is the $k$th time that $\hat{X}^{\inner{r_0}}$ takes a jump. Note that $T^{\inner{r_0}}_1$ is equal to $T^{\inner{r_0}}$ (as defined in Definition \ref{EHI:d:smallflat}).

Recall how the heat kernels $p_t$ and $p^D_t$, and the Green's functions $G$ and $G_D$, are defined in \eqref{EHI:e:HeatKernelDef} and \eqref{EHI:e:GreenDef}. We would like to decompose the Green's function $G_D(x, y)$ into an infinite sum $G_D(x, y) = \sum_{k=0}^\infty G_D^{\inner{r_0}, k}(x, y)$, where for each $k$, for all measurable $A \subseteq \R^d$,
\begin{equation} \label{EHI:e:GkPurpose}
    \int_A G_D^{\inner{r_0}, k}(x, y) \dee{y} = \int_0^\infty \P_x \left( X_t \in A, \tau_D>t, T^{\inner{r_0}}_k \leq t < T^{\inner{r_0}}_{k+1} \right) \dee{t}.
\end{equation}
Let us start by defining the functions $G_D^{\inner{r_0}, k}: D \times D \to [0, \infty)$ recursively. Then we will show that these functions satisfy \eqref{EHI:e:GkPurpose}.

The following definition makes extensive use of the notation ($\lambda\inner{r_0}$, $X^{\inner{r_0}}$, $T^{\inner{r_0}}$, etc.) from Definition \ref{EHI:d:smallflat}. From now on, we will use this notation without commenting on it.

\begin{definition}\label{EHI:d:GkDef}
    Let $X=(X_t)_{t \geq 0}$ be an isotropic unimodal L\'{e}vy jump process on $\R^d$. Fix an open set $D \subseteq \R^d$.
    Let $p^{\inner{r_0}, D}_t(\cdot, \cdot)$ be the heat kernel of the process $X^{\inner{r_0}}$ killed upon exiting $D$.
    Define the function $G_D^{\inner{r_0}, 0} : D \times D \to [0, \infty)$ by
    \begin{equation*} \label{EHI:e:G0def}
        G_D^{\inner{r_0}, 0}(x, y) := \int_0^\infty e^{-\lambda\inner{r_0} t} p^{\inner{r_0}, D}_t(x, y) \dee{t}.
    \end{equation*}
    For all $k$ such that the function $G_D^{\inner{r_0}, k}(\cdot, \cdot)$ has been defined, define $G_D^{\inner{r_0}, k+1}(\cdot, \cdot)$ by
    \begin{equation*}\label{EHI:e:Gk+1def}
        G_D^{\inner{r_0}, k+1}(x, y) := \E_x \left[ \indicatorWithSetBrackets{T^{\inner{r_0}} < \tau_D} G_D^{\inner{r_0}, k}(X_{T^{\inner{r_0}}}, y) \right].
    \end{equation*}
    Furthermore, let
    \begin{equation*}
        G_D^{\inner{r_0}, >0}(x, y) := \sum_{k=1}^\infty G_D^{\inner{r_0}, k}(x, y).
    \end{equation*}
\end{definition}

\begin{lemma}\label{EHI:l:GkPurpose}
    Let $X=(X_t)_{t \geq 0}$ be an isotropic unimodal L\'{e}vy jump process on $\R^d$, and let the functions $G_D^{\inner{r_0}, k}(\cdot, \cdot)$ be as defined in Definition \ref{EHI:d:GkDef}. Then $G_D^{\inner{r_0}, 0}(\cdot, \cdot)$ is symmetric, and each $G_D^{\inner{r_0}, k}(\cdot, \cdot)$ satisfies \eqref{EHI:e:GkPurpose}.
\end{lemma}

\begin{proof}
    The symmetry of $G_D^{\inner{r_0}, 0}(\cdot, \cdot)$ is clear from the definition, since the heat kernel $p^{\inner{r_0}, D}_D(x, y)$ of $X^{\inner{r_0}}$ (killed upon exiting $D$) is symmetric. Let us prove \eqref{EHI:e:GkPurpose} by induction on $k$. For the base case, because $T^{\inner{r_0}}$ is Exponential($\lambda\inner{r_0}$) and independent of the process $X^{\inner{r_0}}$, we have
    \begin{align*}
        \int_0^\infty \P_x \left( X_t \in A, \tau_D>t, T^{\inner{r_0}}_0 \leq t < T^{\inner{r_0}}_1 \right) \dee{t} &= \int_0^\infty \P_x \left( X^{\inner{r_0}}_t \in A, \tau^{\inner{r_0}}_D > t, T^{\inner{r_0}}>t \right) \dee{t} \\
        &= \int_0^\infty e^{-\lambda\inner{r_0}t} \P_x \left( X^{\inner{r_0}}_t \in A, \tau^{\inner{r_0}}_D > t \right) \dee{t} \\
        &= \int_0^\infty e^{-\lambda\inner{r_0}t} \int_A p^{\inner{r_0}, D}_t(x, y) \dee{y} \dee{t}.
    \end{align*}
    Tonelli's theorem allows us to switch the order of integration, which yields
    \begin{align*}
        \int_0^\infty \P_x \left( X_t \in A, \tau_D>t, T^{\inner{r_0}}_0 \leq t < T^{\inner{r_0}}_1 \right) \dee{t} &= \int_A \int_0^\infty e^{-\lambda\inner{r_0}t} p^{\inner{r_0}, D}_t(x, y) \dee{t} \dee{y} \\
        &=: \int_A G_D^{\inner{r_0}, 0}(x, y) \dee{y},
    \end{align*}
    thus confirming \eqref{EHI:e:GkPurpose} for $k=0$. Assume \eqref{EHI:e:GkPurpose} holds for $k$. By the Strong Markov property, the inductive hypothesis, and multiple uses of Tonelli's theorem,
    \begin{align*}
        \int_0^\infty \P_x \left( X_t \in A, \tau_D>t, T^{\inner{r_0}}_{k+1} \leq t < T^{\inner{r_0}}_{k+2} \right) \dee{t} &= \E_x \left[ \int_0^\infty \indicatorWithSetBrackets{X_t \in A, \tau_D>t, T^{\inner{r_0}}_{k+1} \leq t < T^{\inner{r_0}}_{k+2}} \dee{t} \right] \\
        &= \E_x \left[ \indicatorWithSetBrackets{T^{\inner{r_0}}_1 < \tau_D} \int_A G_D^{\inner{r_0}, k}(X_{T^{\inner{r_0}}_1}, y) \dee{y} \right] \\
        &= \int_A \E_x \left[ \indicatorWithSetBrackets{T^{\inner{r_0}}_1 < \tau_D} G_D^{\inner{r_0}, k}(X_{T^{\inner{r_0}}_1}, y) \right] \dee{y} \\
        &=: \int_A G_D^{\inner{r_0}, k+1}(x, y) \dee{y}.
    \end{align*}
\end{proof}

Let us similarly decompose the Poisson kernel.
Recall that in the small/flat decomposition, we decompose the jump kernel $j(r)$ into $j^{\inner{r_0}}(r) + \hat{j}^{\inner{r_0}}(r)$. The Poisson kernel can be decomposed into the sum $K_D = K^{\inner{r_0}, 0}_D + K^{\inner{r_0}, >0}_D$, where
\begin{equation}\label{EHI:e:K0formula}
    K_D^{\inner{r_0}, 0}(x, w) := \int_{z \in D} G_D^{\inner{r_0}, 0}(x, z) j^{\inner{r_0}}(|w-z|) \dee{z}
\end{equation}
and
\begin{equation} \label{EHI:e:K>0formula}
    K_D^{\inner{r_0}, >0}(x, w) := \int_{z \in D} \left( G_D^{\inner{r_0}, 0}(x, z) \hat{j}^{\inner{r_0}}(|w-z|) + G_D^{\inner{r_0}, >0}(x, z) j(|w-z|) \right) \dee{z}.
\end{equation}
Probabilistically, $K_D^{\inner{r_0}, 0}(x, z)$ accounts for $X$ traveling from $x$ to $z$ without $\hat{X}^{\inner{r_0}}$ taking any jumps, and $K_D^{\inner{r_0}, >0}(x, z)$ accounts for $X$ traveling from $x$ to $z$ involving at least one jump from $\hat{X}^{\inner{r_0}}$. 
Note that
\begin{equation}\label{EHI:e:(K>0)>K1}
    K_D^{\inner{r_0}, >0}(x, w) \geq \int_{z \in D} G_D^{\inner{r_0}, 1}(x, z) j^{\inner{r_0}}(|w-z|) \dee{z}.
\end{equation}

Note that all the constructions in this subsection generalize to the other Meyer decompositions. We only explicitly present them for the small/flat decomposition, since that is what we will use in the proof of our main results.

\subsection{Properties of the small/flat decomposition}

When $X$ satisfies \eqref{EHI:e:regularJumps}, the small/flat decomposition has two important properties which are essential to our proofs: the first of these properties is that $m_2\inner{r_0} \asymp m_2(r_0)$ and $\lambda\inner{r_0} \asymp \lambda(r_0)$, and the second is that $\hat{j}^{\inner{r_0}}(r_1) \asymp \hat{j}^{\inner{r_0}}(r_2)$ whenever $r_1, r_2 \lesssim r_0$ (even if one of them is very close to $0$). We prove these properties in the following two lemmas. Then in Proposition \ref{EHI:p:h1(x)=h1(-x)} we use these properties to prove that $h^{\inner{r_0}, >0}_{B(0, r)}(x, w) \asymp h^{\inner{r_0}, >0}_{B(0, r)}(-x, w)$ for $x \in B(0, r)$ and non-negative functions $h$ that are harmonic on some $D \Supset B(0, r)$.

\begin{lemma}\label{EHI:l:m2lambdaComparable}
    Let $X$ be an isotropic unimodal L\'{e}vy jump process. For all $r_0>0$, we have
    \begin{equation*}
        \frac12 m_2(r_0) \leq m_2\inner{r_0} \leq m_2(r_0) \qquad\mbox{and}\qquad \lambda\inner{r_0} \geq \lambda(r_0).
    \end{equation*}
    If $X$ satisfies \eqref{EHI:e:regularJumps} for all $r>0$, then we also have
    \begin{equation*}
        \lambda\inner{r_0} \lesssim \lambda(r_0),
    \end{equation*}
    where the constant implicit in $\lesssim$ depends only on $c_j$ (from \eqref{EHI:e:regularJumps}) and $d$.
\end{lemma}

\begin{proof}
    For all $r \leq r_0$, it follows from the definition of the small/flat decomposition that $\frac12 j(r) \leq j(r) - \frac12 j(r_0) =: j^{\inner{r_0}}(r) \leq j(r)$. Thus, $\frac12 m_2(r_0) \leq m_2\inner{r_0} \leq m_2(r_0)$. It also follows easily from the definition that $\lambda\inner{r_0} \geq \lambda(r_0)$. If \eqref{EHI:e:regularJumps} holds, then
    \begin{align*}
        \frac{\lambda\inner{r_0} - \lambda(r_0)}{\lambda(r_0)} &= \frac{|B(0, r_0)| \cdot \frac12 j(r_0)}{\int_{B(0, r_0)^c} j(|x|) \dee{x}} \leq \frac{|B(0, r_0)| \cdot \frac12 j(r_0)}{\int_{B(0, 2r_0) \setminus B(0, r_0)} j(|x|) \dee{x}} \\
        &\leq \frac{|B(0, r_0)| \cdot \frac12 j(r_0)}{|B(0, 2r_0) \setminus B(0, r_0)| \cdot c_j j(r_0)} = \frac{1}{2(2^d-1) c_j}.
    \end{align*}
\end{proof}

\begin{lemma} \label{EHI:l:flatness}
    Let $X$ be an isotropic unimodal L\'{e}vy jump process on $\R^d$. If \eqref{EHI:e:regularJumps} holds for all $r>0$, then for all $L>1$, there exists a constant $C_{\ref{EHI:l:flatness}}(L, c_j)$ (depending on $L$ and $c_j$) such that for all $r_0>0$, for all $r_1, r_2 \in (0, Lr_0]$, we have
    \begin{equation*}
        \hat{j}^{\inner{r_0}}(r_1) \leq C_{\ref{EHI:l:flatness}}(L, c_j) \hat{j}^{\inner{r_0}}(r_2).
    \end{equation*}
\end{lemma}

\begin{proof}
    For all $r \in (0, r_0]$, we have $\hat{j}^{\inner{r_0}}(r) = \frac12 j(r_0)$. For all $r \in (r_0, Lr_0]$, we have $j(Lr_0) \leq \hat{j}^{\inner{r_0}}(r) \leq j(r_0)$. Therefore,
    \begin{equation} \label{EHI:e:flatness1}
        \min \left\{ \frac12 j(r_0), j(Lr_0) \right\} \leq \hat{j}^{\inner{r_0}}(r_1), \hat{j}^{\inner{r_0}}(r_2) \leq j(r_0).
    \end{equation}
    By \eqref{EHI:e:flatness1} and \eqref{EHI:e:regularJumps}, 
    \begin{equation*}
        \frac{\hat{j}^{\inner{r_0}}(r_1)}{\hat{j}^{\inner{r_0}}(r_2)} \leq \max \left\{ 2, \frac{j(r_0)}{j(Lr_0)} \right\} \leq \max \left\{ 2, c_j^{-\ceil{\log_2 L}} \right\}.
    \end{equation*}
\end{proof}

Now that we have defined the small/flat decomposition and shown the basic properties above, we can explain how it will be used in the proof of our main results. Suppose a function $h$ is non-negative everywhere and harmonic on a set $D$ such that the ball $B(0, r)$ is compactly supported in $D$. The process $(h(X_{t \wedge \tau_{B(0, r)}}))_{t \geq 0}$ is a martingale with respect to the filtration $\mathcal{F}^X_t := \sigma\left(\left\{ X_s : 0 \leq s \leq t \right\} \right)$.
Fix $x \in B(0, r)$.
We would like to show that $h(x)$ and $h(-x)$ are comparable.

Consider the small/flat decomposition of $X$ at some $r_0$ smaller than but within a large multiplicative constant of $r$.
Let $(\mathcal{F}_t)_{t \geq 0}$ be the filtration that takes into account not only $X$, but also $X^{\inner{r_0}}$ and $\hat{X}^{\inner{r_0}}$ from the small/flat decomposition at $r_0$, so that $T^{\inner{r_0}}$ is a stopping time with respect to $(\mathcal{F}_t)$. The process $(h(X_{t \wedge \tau_{B(0, r)}}))_{t \geq 0}$ is still a martingale with respect to $(\mathcal{F}_t)$, because for all $0 \leq s \leq t$,
\begin{align*}
    \mathbb{E}_x [h(X_{t \wedge \tau_{B(0, r)}})|\mathcal{F}_s] &= \mathbb{E}_x \left[ \mathbb{E}_x [h(X_{t \wedge \tau_{B(0, r)}}) | \mathcal{F}^X_s] \Big| \mathcal{F}_s \right]\\
    &= \mathbb{E}_x [h(X_{s \wedge \tau_{B(0, r)}}) | \mathcal{F}_s] = h(X_{s \wedge \tau_{B(0, r)}}).
\end{align*}
Since $\tau_{B(0, r)}$ and $T^{\inner{r_0}}$ are both stopping times with respect to this new filtration, for all $x \in {B(0, r)}$ we have both
\begin{equation*}
    h(x) = \mathbb{E}_x \left[h(X_{\tau_{B(0, r)}}) \right] = \mathbb{E}_x \left[h(X_{\tau_{B(0, r)}}) \indicatorWithSetBrackets{\tau_{B(0, r)}<T^{\inner{r_0}}} \right] + \mathbb{E}_x \left[h(X_{\tau_{B(0, r)}}) \indicatorWithSetBrackets{\tau_{B(0, r)} \geq T^{\inner{r_0}}} \right]
\end{equation*}
and
\begin{equation*}
    h(x) = \mathbb{E}_x \left[h(X_{T \wedge \tau_{B(0, r)}}) \right] = \mathbb{E}_x \left[h(X_{\tau_{B(0, r)}}) \indicatorWithSetBrackets{\tau_{B(0, r)}<T^{\inner{r_0}}} \right] + \mathbb{E}_x \left[h(X_T) \indicatorWithSetBrackets{\tau_{B(0, r)} \geq T^{\inner{r_0}}} \right].
\end{equation*}
Therefore, we can write $h(x)$ as the sum of $h^{\inner{r_0}, 0}_{B(0, r)}(x)$ and $h^{\inner{r_0}, >0}_{B(0, r)}(x)$, where
\begin{equation}\label{EHI:e:h0def}
    h^{\inner{r_0}, 0}_{B(0, r)}(x) := \mathbb{E}_x \left[h(X_{\tau_{B(0, r)}}) \indicatorWithSetBrackets{\tau_{B(0, r)}<T^{\inner{r_0}}} \right]
\end{equation}
and $h^{\inner{r_0}, >0}_{B(0, r)}(x)$ is the common value of
\begin{equation}\label{EHI:e:h1def}
    \mathbb{E}_x \left[h(X_{\tau_{B(0, r)}}) \indicatorWithSetBrackets{\tau_{B(0, r)} \geq T^{\inner{r_0}}} \right] =: h^{\inner{r_0}, >0}_{B(0, r)}(x) := \mathbb{E}_x \left[h(X_{T^{\inner{r_0}}}) \indicatorWithSetBrackets{\tau_{B(0, r)} \geq T^{\inner{r_0}}} \right].
\end{equation}
We will use these two different expressions for $h^{\inner{r_0}, >0}_{B(0, r)}(x)$ for different purposes. Using right-hand side of \eqref{EHI:e:h1def} and the flatness of $\hat{j}^{\inner{r_0}}$, we show that $h^{\inner{r_0}, >0}_{B(0, r)}(x) \asymp h^{\inner{r_0}, >0}_{B(0, r)}(-x)$ for all $x \in B(0, r)$. Later, in Section \ref{EHI:s:positiveResults}, we use the left-hand side of \eqref{EHI:e:h1def} to compare $h^{\inner{r_0}, >0}_{B(0, r)}(x)$ with $h^{\inner{r_0}, 0}_{B(0, r)}(x)$.

Recall how the Poisson kernel $K_D$ is decomposed into $K^{\inner{r_0}, 0}_D$ and $K^{\inner{r_0}, >0}_D$, as given by \eqref{EHI:e:K0formula} and \eqref{EHI:e:K>0formula}.
It follows from the Ikeda-Watanabe formula (cf. \cite{IW}) that
\begin{equation}\label{EHI:e:h0h1relatedToK0K1}
    h^{\inner{r_0}, 0}_{B(0, r)}(x) = \int_{w \in B(0, r)^c} K^{\inner{r_0}, 0}_{B(0, r)}(x, w) h(w) \dee{w}, \qquad h^{\inner{r_0}, >0}_{B(0, r)}(x) = \int_{w \in B(0, r)^c} K^{\inner{r_0}, >0}_{B(0, r)}(x, w) h(w) \dee{w}.
\end{equation}

In the following proposition, we show that $h^{\inner{r_0}, >0}_{B(0, r)}(x) \asymp h^{\inner{r_0}, >0}_{B(0, r)}(-x)$.

\begin{prop} \label{EHI:p:h1(x)=h1(-x)}
    Let $X$ be an isotropic unimodal L\'{e}vy jump process on $\R^d$ satisfying \eqref{EHI:e:regularJumps}. Suppose $B(0, r) \Subset D$, and let $h$ be a function that is non-negative everywhere and harmonic on $D$. Suppose $r_0 \geq r/L$ for some $L>1$. Then there exists a universal constant $C_{\ref{EHI:p:h1(x)=h1(-x)}}(L, c_j)$ (depending on $L$ and $c_j$) such that for all $x \in B(0, r)$,
    \begin{equation*}
        C_{\ref{EHI:p:h1(x)=h1(-x)}}(L, c_j)^{-1} h^{\inner{r_0}, >0}_{B(0, r)}(-x) \leq h^{\inner{r_0}, >0}_{B(0, r)}(x) \leq C_{\ref{EHI:p:h1(x)=h1(-x)}}(L, c_j) h^{\inner{r_0}, >0}_{B(0, r)}(-x). 
    \end{equation*}
\end{prop}

\begin{proof}
    Let $\eta_x$ be the measure such that
    \begin{equation*}
        \eta_x(E) := \mathbb{P}_x \left(\tau^{\inner{r_0}}_{B(0, r)}>T^{\inner{r_0}}, X^{\inner{r_0}}_{T^{\inner{r_0}}} \in E \right) \qquad\mbox{ for all measurable $E \subseteq B(0, r)$}.
    \end{equation*}
    Consider $\Delta X(T^{\inner{r_0}})$, the displacement vector of the first flattening jump. Note that $X_{T^{\inner{r_0}}} = X^{\inner{r_0}}_{T^{\inner{r_0}}} + \Delta X (T^{\inner{r_0}})$ and that $\Delta X (T^{\inner{r_0}})$ has density $\frac{\hat{j}^{\inner{r_0}}(|\cdot|)}{\lambda\inner{r_0}}$. Thus,
    \begin{align}
        h^{\inner{r_0}, >0}_{B(0, r)}(x) &= \mathbb{E}_x \left[ h(X_{T^{\inner{r_0}}}) \indicatorWithSetBrackets{\tau_{B(0, r)} \geq T^{\inner{r_0}}} \right] \notag \\
        &= \mathbb{E}_x \left[ h \left(X^{\inner{r_0}}_{T^{\inner{r_0}}} + \Delta X(T^{\inner{r_0}}) \right) \indicatorWithSetBrackets{\tau^{\inner{r_0}}_{B(0, r)}>T^{\inner{r_0}}} \right] \notag \\
        &= \int_{w \in \mathbb{R}^d} \int_{z \in B(0, r)} \frac{\hat{j}^{\inner{r_0}}(|w-z|)}{\lambda\inner{r_0}} h(w) \, \eta_x(\mathrm{d}z) \dee{w}. \label{EHI:e:h1comp1}
    \end{align}
    Note that the process $-X^{\inner{r_0}}+\hat{X}^{\inner{r_0}}$ has the same law as $X^{\inner{r_0}}+\hat{X}^{\inner{r_0}}=X$, because $-X^{\inner{r_0}}$ has the same law as $X^{\inner{r_0}}$, and both are independent of $\hat{X}^{\inner{r_0}}$.
    By symmetry, if $T^{\inner{r_0}}$ occurs before $X^{\inner{r_0}}$ exits $B(0, r)$, then $T$ also occurs before $-X^{\inner{r_0}}_t$ exits $B(0, r)$. Thus,
    \begin{align}\label{EHI:e:h1comp2}
        h^{\inner{r_0}, >0}_{B(0, r)}(-x) &= \E_{-x} \left[ h\left( X^{\inner{r_0}}_{T^{\inner{r_0}}} + \Delta X (T^{\inner{r_0}})\right) \indicatorWithSetBrackets{\tau^{\inner{r_0}}_{B(0, r)} > T^{\inner{r_0}}} \right] \notag\\
        &= \E_x \left[ h\left( -X^{\inner{r_0}}_{T^{\inner{r_0}}} + \Delta X (T^{\inner{r_0}})\right) \indicatorWithSetBrackets{\tau^{\inner{r_0}}_{B(0, r)} > T^{\inner{r_0}}} \right] \notag\\
        &= \int_{w \in \R^d} \int_{z \in B(0, r)} \frac{\hat{j}^{\inner{r_0}}(|w+z|)}{\lambda\inner{r_0}} h(w) \, \eta_x(\mathrm{d}z) \dee{w}.
    \end{align}
    
    The only difference between \eqref{EHI:e:h1comp1} and \eqref{EHI:e:h1comp2} is that $\hat{j}^{\inner{r_0}}(|w-z|)$ is replaced with $\hat{j}^{\inner{r_0}}(|w+z|)$. Therefore, in order to show that the quantites in \eqref{EHI:e:h1comp1} and \eqref{EHI:e:h1comp2} are comparable, it is enough to show that $\hat{j}^{\inner{r_0}}(|w-z|)$ and $\hat{j}^{\inner{r_0}}(|w+z|)$ are comparable for all $w \in \mathbb{R}^d$, $z \in B(0, r)$.
    We prove this by two cases. If $|w| \geq 2r$, then $r \leq \frac12 |w| \leq |w-z|, |w+z| \leq 2|w|$, so by the definition of the small/flat decomposition, and by \eqref{EHI:e:regularJumps},
    \begin{equation*}
        \frac{\hat{j}^{\inner{r_0}}(|w-z|)}{\hat{j}^{\inner{r_0}}(|w+z|)} = \frac{j(|w-z|)}{j(|w+z|)} \leq \frac{j(\frac12|w|)}{j(2|w|)} \leq c_j^{-2}.
    \end{equation*}
    On the other hand, if $|w| \leq 2r$, then both $|w-z|$ and $|w+z|$ are at most $3r$, so by Lemma \ref{EHI:l:flatness},
    \begin{equation*}
        \hat{j}^{\inner{r_0}}(|w-z|) \leq C_{\ref{EHI:l:flatness}} \left( 3L, c_j \right) \hat{j}^{\inner{r_0}}(|w+z|).
    \end{equation*}
    In either case, if we let $C_{\ref{EHI:p:h1(x)=h1(-x)}}(L, c_j) := \max\{ c_j^{-2}, C_{\ref{EHI:l:flatness}}(3L, c_j) \} $, then 
    \begin{equation*}
        \hat{j}^{\inner{r_0}}(|w-z|) \leq C_{\ref{EHI:p:h1(x)=h1(-x)}}(L, c_j) \hat{j}^{\inner{r_0}}(|w+z|).
    \end{equation*}
    By the same argument,
    \begin{equation*}
        \hat{j}^{\inner{r_0}}(|w+z|) \leq C_{\ref{EHI:p:h1(x)=h1(-x)}}(L, c_j) \hat{j}^{\inner{r_0}}(|w-z|).
    \end{equation*}
    This holds for all $w \in \R^d$ and $z \in B(0, r)$. Thus, by \eqref{EHI:e:h1comp1} and \eqref{EHI:e:h1comp2},
    \begin{equation*}
        C_{\ref{EHI:p:h1(x)=h1(-x)}}(L, c_j)^{-1} h^{\inner{r_0}, >0}_{B(0, r)}(-x) \leq h^{\inner{r_0}, >0}_{B(0, r)}(x) \leq C_{\ref{EHI:p:h1(x)=h1(-x)}}(L, c_j) h^{\inner{r_0}, >0}_{B(0, r)}(-x). 
    \end{equation*}
\end{proof}

\begin{corollary}\label{EHI:c:K1(x)=K1(-x)}
    Let $X$ be an isotropic unimodal L\'{e}vy jump process on $\R^d$ satisfying \eqref{EHI:e:regularJumps}. Suppose $B(0, r) \Subset D$, and let $h$ be a function that is non-negative everywhere and harmonic on $D$. Suppose $r_0 \geq r/L$ for some $L>1$. For all $x \in B(0, r)$ and $w \in B(0, r)^c$,
    \begin{equation}\label{EHI:e:K1(x)=K1(-x)}
        C_{\ref{EHI:p:h1(x)=h1(-x)}}(L, c_j)^{-1} K^{\inner{r_0}, >0}_{B(0, r)}(-x, w) \leq K^{\inner{r_0}, >0}_{B(0, r)}(x, w) \leq C_{\ref{EHI:p:h1(x)=h1(-x)}}(L, c_j) K^{\inner{r_0}, >0}_{B(0, r)}(-x, w). 
    \end{equation}
\end{corollary}

\begin{proof}
    Fix $x \in B(0, r)$. Let
    \begin{equation*}
        A := \left\{ w \in B(0, r)^c : K^{\inner{r_0}, >0}_{B(0, r)}(x, w) > C_{\ref{EHI:p:h1(x)=h1(-x)}}(L, c_j) K^{\inner{r_0}, >0}_{B(0, r)}(-x, w)\right\}.
    \end{equation*}
    We would like to prove that $A$ is empty.
    Let $h$ be the function $h(x) := \P_x(X_{\tau_{B(0, r)}} \in A)$. By Proposition \ref{EHI:p:h1(x)=h1(-x)},
    \begin{equation}\label{EHI:e:K1(x)=K1(-x)1}
        h^{\inner{r_0}, >0}_{B(0, r)}(x) \leq C_{\ref{EHI:p:h1(x)=h1(-x)}}(L, c_j) h^{\inner{r_0}, >0}_{B(0, r)}(-x).
    \end{equation}
    By \eqref{EHI:e:h0h1relatedToK0K1},
    \begin{equation}\label{EHI:e:K1(x)=K1(-x)2}
        h^{\inner{r_0}, >0}_{B(0, r)}(x) - C_{\ref{EHI:p:h1(x)=h1(-x)}}(L, c_j) h^{\inner{r_0}, >0}_{B(0, r)}(-x) =\int_{w \in A} \left( K^{\inner{r_0}, >0}_{B(0, r)}(x, w) - C_{\ref{EHI:p:h1(x)=h1(-x)}}(L, c_j) K^{\inner{r_0}, >0}_{B(0, r)}(-x, w) \right) \dee{w}.
    \end{equation}
    By the definition of $A$, the integrand of \eqref{EHI:e:K1(x)=K1(-x)2} is positive for all $w \in A$, so if $A$ has positive Lebesgue measure, then the quantity in \eqref{EHI:e:K1(x)=K1(-x)2} is positive. However, \eqref{EHI:e:K1(x)=K1(-x)1} tells us that this quantity is non-positive. Therefore, $A$ has Lebesgue measure $0$. By the continuity of $K^{\inner{r_0}, >0}$, $A$ is empty, which gives us the second inequality in \eqref{EHI:e:K1(x)=K1(-x)}. Repeating the same argument with $x$ and $-x$ switched gives us the first inequality.
\end{proof}

\section{Preliminary facts about isotropic unimodal L\'{e}vy jump processes with bounded jumps} \label{EHI:s:Prelim}

In this section, we build up a collection of facts to be used in the proof of our main results. Most of these facts relate to processes with bounded jumps, such as $X^{(r_0)}$ and $X^{\inner{r_0}}$ (from the small/large and small/flat decompositions, respectively). To keep the results general, we refer to the process in question as $X$ in the statement of the results of this section, but the reader should keep in mind that when we apply these results, it will usually be to $X^{(r_0)}$ or $X^{\inner{r_0}}$, where $X$ is the process that our main results refer to.

The most important results in this Section are Lemma \ref{EHI:l:m2isvariance}, which tells us that for a process $X$ with bounded jumps, the second moment of the jump kernel of $X$ is the speed at which $|X_t|^2$ grows, and the results of Section \ref{EHI:ss:BerryEsseen}, which use a multivariate form the Berry-Esseen theorem to show that the behavior of such a process at large scales is comparable to a Brownian motion in which $|X_t|^2$ grows at the same speed. Other results in this section are mostly calculations that will simplify the work in the proof of our main results.

\subsection{Basic observations}\label{EHI:ss:PrelimBasic}

\FloatBarrier

\begin{lemma}\label{EHI:l:FactorOf2ForExits}
    Let $X=(X_t)_{t \geq 0}$ be an isotropic unimodal L\'{e}vy jump process on $\mathbb{R}^d$ starting at $X_0=0$, let $R>0$, and let $T$ be a stopping time such that the conditional distribution of $X_T$ given $\tau_{B(0, R)} \leq T$ is isotropic unimodal centered at $\tau_{B(0, R)}$. Then
    \begin{equation*}
        \mathbb{P}_0 \left( |X_T| \geq R \right) \leq \mathbb{P}_0 \left( \tau_{B(0, R)} \leq T \right) \leq 2 \mathbb{P}_0 \left( |X_T| \geq R \right).
    \end{equation*}
\end{lemma}

\begin{figure}[ht]
    \begin{framed}
        \captionof{figure}[A visual aide for the proof of Lemma \ref{EHI:l:FactorOf2ForExits}]{A visual aide for the proof of Lemma \ref{EHI:l:FactorOf2ForExits}. The dotted ball is the reflection of $B(0, R)$ across $X_{\tau_{B(0, R)}}$. The conditional distribution of $X_T$ given $\tau_{B(0, R)} \leq T$ is isotropic unimodal, centered at $X_{\tau_{B(0, R)}}$. Therefore, it assigns equal mass to each of the disjoint balls in the figure. In particular, it assigns mass at most $1/2$ to $B(0, R)$.}
        \centering
        \begin{tikzpicture}[scale=1]
            \node[shape=circle, fill=black, scale=0.1, label=$0$] (0) at (0, 0) {$0$};
            \node[shape=circle, fill=black, scale=0.1, label=$X_{\tau_{B(0, R)}}$] (x) at (2, 0.5) {$X_{\tau_{B(0, R)}}$};
            \draw[fill=none] (0, 0) circle (1);
            \draw[dotted, thick, fill=none] (4, 1) circle (1);
            \draw (0, 0) -- (1, 0);
            \node at (0.5, -0.2) {$R$};
        \end{tikzpicture}
    \end{framed}
\end{figure}
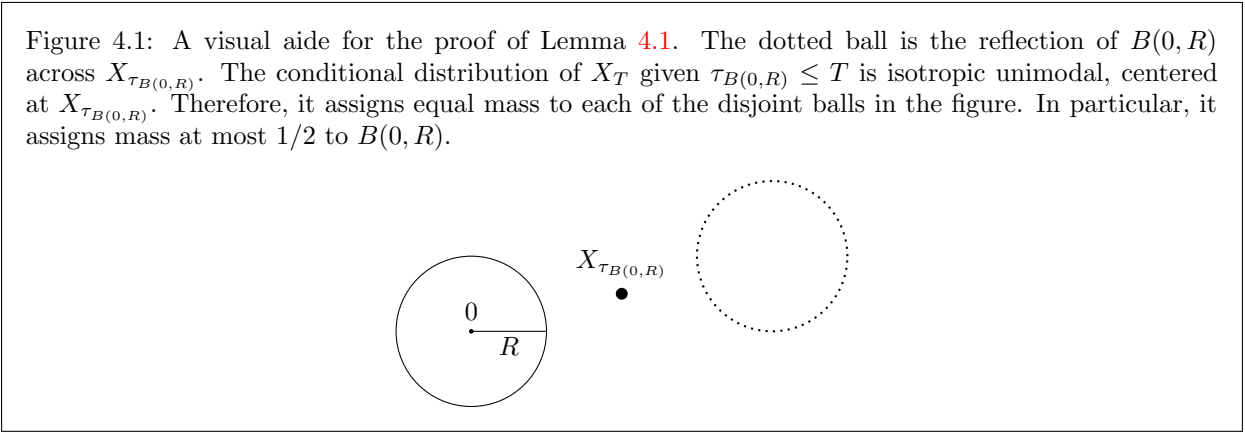

\begin{proof}
    If $|X_T| \geq R$, then we clearly have $\tau_{B(0, R)} \leq T$, by the definition of $\tau_{B(0, R)}$. Therefore, the desired lower bound on $\mathbb{P}_0 \left( \tau_{B(0, R)} \leq T \right)$ holds.
    
    For the upper bound, we will use the Strong Markov property and symmetry to argue that if the process exits $B(0, R)$ before time $t$, then it is more likely than not to still be outside $B(0, R)$ at time $T$. Let us condition on the event $\{\tau_{B(0, R)} \leq T \}$ occurring. The conditional distribution of $X_T$ is isotropic unimodal, centered at $X_{\tau_{B(0, R)}}$, which is outside of $B(0, R)$. Therefore, it assigns mass at most $1/2$ to $B(0, R)$. Thus,
    \begin{equation*}
        \mathbb{P}_0 \left( |X_T| \geq R \right) \geq \mathbb{P}_0 \left( \tau_{B(0, R)} \leq T \right) \mathbb{P}_0 \left( |X_T| \geq R \Big| \tau_{B(0, R)} \leq T \right) \geq \frac12 \mathbb{P}_0 \left( \tau_{B(0, R)} \leq T \right),
    \end{equation*}
    which is equivalent to the upper bound on $\mathbb{P}_0 \left( \tau_{B(0, R)} < T \right)$ that we desired.
\end{proof}

\FloatBarrier

Many of the remaining results in this section will be for processes $X$ that only take bounded up (ie, no jumps of magnitude greater than $r_0$ for some $r_0$). In practice, when we use these results later in this chapter, we will apply them not to the original process $X$, but to $X'$ from a Meyer decomposition $X=X'+\hat{X}'$ in which $X'$ only takes small jumps.

\begin{definition}
    For all $r_0>0$, let $\mathcal{I}(d, r_0)$ be the class of isotropic unimodal L\'{e}vy jump processes that do not take any jumps of magnitude greater than $r_0$ (in other words, whose jump kernel is supported on $[0, r_0]$.)
\end{definition}

\begin{lemma}\label{EHI:l:CloseToSubmultiplicative}
    Let $X \in \mathcal{I}(d, r_0)$, and let $T$ be a random variable with Exponential($\lambda$) distribution, independent of $X$. For all $r > 0$ and $m \in \mathbb{N}$,
    \begin{equation*}
        \mathbb{P}_0 \left( \tau_{B(0, mr+(m-1)r_0)} < T \right) \leq \left( \mathbb{P}_0 \left( \tau_{B(0, r)}<T \right)\right)^m.
    \end{equation*}
\end{lemma}

\begin{proof}
    For all $m$, let $B_m := B(0, mr+(m-1)r_0)$ and $\tilde{B}_m := B(0, mr+mr_0)$.
    We will show by induction on $m$ that
    \begin{equation} \label{EHI:e:SubmultiplicativityWithBm}
        \mathbb{P}_0 \left( \tau_{B_m} < T \right) \leq \left( \mathbb{P}_0 \left( \tau_{B(0, r)}<T \right)\right)^m
    \end{equation}
    For the base case ($m=1$), the two sides of \eqref{EHI:e:SubmultiplicativityWithBm} are identical. Assume \eqref{EHI:e:SubmultiplicativityWithBm} holds for $m$, and let us show that it also holds for $m+1$. Consider what must happen for the process starting at $0$ to exit $B_{m+1}$ before time $T$. First, it must exit $B_m$, which will bring it to some location in $\tilde{B}_m \setminus B_m$ (since its last jump had a magnitude less than $r_0$, and the radius of $\tilde{B}_m$ is equal to $r_0$ plus the radius of $B_m$).
    From this location, it must travel an additional distance of at least $r$ before time $T$ in order to exit $B_{m+1}$ (since the radius of $B_{m+1}$ is equal to $r$ plus the radius of $\tilde{B}_m$). Therefore, by the Strong Markov property, the memoryless property of $T$, and the inductive hypothesis,
    \begin{align*}
         \mathbb{P}_0 \left( \tau_{B_{m+1}} < T \right) &=  \mathbb{P}_0 \left( \tau_{B_m} < T \right)  \mathbb{P}_0 \left( \tau_{B_{m+1}} < T \Big| \tau_{B_m} < T \right) \\
         &\leq \mathbb{P}_0 \left( \tau_{B_m} < T \right) \sup_{z \in \tilde{B}_m \setminus B_m} \mathbb{P}_z \left( \tau_{B_{m+1}} < T \right) \\
         &\leq \mathbb{P}_0 \left( \tau_{B_m} < T \right) \sup_{z \in \tilde{B}_m \setminus B_m} \mathbb{P}_z \left( \tau_{B(z, r))} < T\right)\\
         &= \mathbb{P}_0 \left( \tau_{B_m} < T \right) \mathbb{P}_0 \left( \tau_{B(0, r))} < T \right) \\
         &\leq \left( \mathbb{P}_0 \left( \tau_{B(0, r))} < T \right) \right)^{m+1},
    \end{align*}
    which completes the inductive proof.
\end{proof}

\subsection{\texorpdfstring{The second moment is the rate of speed of growth of $|X_t|^2$}{TEXT}}

Given a process $X \in \mathcal{I}(d, r_0)$, let
\begin{equation}\label{EHI:e:m2(X)def}
    m_2(X) := \int_{\R^d} |x|^2 j(|x|).
\end{equation}
Because $j(|x|)=0$ for $|x|>r_0$, the L\'{e}vy-Khintchine condition guarantees that $m_2(X) < \infty$ for $X \in \mathcal{I}(d, r_0)$.

\begin{lemma} \label{EHI:l:m2isvariance}
    Given $X \in \mathcal{I}(d, r_0)$, we have $\mathbb{E}_x |X_t-x|^2 = m_2(X) \cdot t$ for all $x \in \mathbb{R}^d$ and $t>0$.
\end{lemma}

\begin{proof}
    For all $y \in \mathbb{R}^d$ and $j \in \{1, \dots, d\}$, let $y^{(j)}$ denote the $j$th coordinate of the vector $y$.
    Let $(\mathcal{G}_t)_{t \geq 0}$ be the filtration generated by only the magnitudes of the jumps of $X$: that is, $\mathcal{G}_t := \sigma \left( \left\{ |X_s - X_{s-}| : 0<s \leq t \right\} \right)$. 
    Condition on $\mathcal{G}_t$ and consider an $s \in (0, t]$ at which a jump of size $X_s-X_{s-}$ occurs. By the definition of $|\cdot|$ on $\mathbb{R}^d$,
    \begin{equation*}
        \left|X_s-X_{s-} \right|^2 = \sum_{j=1}^d \left|X_s^{(j)}-X_{s-}^{(j)} \right|^2.
    \end{equation*}
    By symmetry, for each $j \in \{1, \dots, d\}$,
    \begin{equation*}
        \Var_x \left(  \left. X_s^{(j)}-X_{s-}^{(j)} \right| \mathcal{G}_t \right) = \mathbb{E}_x \left[ \left| \left. X_s^{(j)}-X_{s-}^{(j)} \right|^2 \right| \mathcal{G}_t \right] = \frac{1}{d} |X_s-X_{s-}|^2.
    \end{equation*}
    By summing over all $s \in (0, t]$, since variances of independent random variables are additive,
    \begin{align*}
        \mathbb{E}_x \left[ \left. \left| X_t^{(j)}-x^{(j)} \right|^2 \right| \mathcal{G}_t \right] &= \Var_x \left( \left.  X_t^{(j)}-x^{(j)} \right| \mathcal{G}_t \right) = \sum_{0<s \leq t} \Var_x \left(  \left. X_s^{(j)}-X_{s-}^{(j)} \right| \mathcal{G}_t \right) \\
        &= \frac{1}{d} \sum_{0<s \leq t} |X_s-X_{s-}|^2.
    \end{align*}
    Summing over all $d$,
    \begin{equation*}
        \mathbb{E}_x \left[ \left. \left| X_t-x \right|^2 \right| \mathcal{G}_t \right] = \sum_{j=1}^d \mathbb{E}_x \left[ \left. \left| X_t^{(j)}-x^{(j)} \right|^2 \right| \mathcal{G}_t \right] = \sum_{0<s \leq t} |X_s-X_{s-}|^2.
    \end{equation*}
    Taking expectations,
    \begin{equation}\label{EHI:m2tsecondmoment1}
        \mathbb{E}_x |X_t-x|^2 = \mathbb{E}_x \left(  \mathbb{E}_x \left[ \left. \left| X_t-x \right|^2 \right| \mathcal{G}_t \right]\right) = \mathbb{E}_x \left[ \sum_{0<s \leq t} |X_s-X_{s-}|^2 \right].
    \end{equation}

    Fix $\ell>1$. For all $k \in \mathbb{N}$, let $E_k$ denote the annulus $E_k := \{ x \in \mathbb{R}^d : \ell^{-k} r_0 < |x| \leq \ell^{-(k-1)} r_0 \}$. Then $\{E_k\}_{k=1}^\infty$ is a partition of $B(0, r_0)$, so every jump that $X$ takes has its displacement in $E_k$ for some unique $k$. Fix $t>0$. For all $k$, let $N^{(k)}_t$ be the number of jumps $X$ takes with displacement in $E_k$ by time $t$. (In other words, $N^{(k)}_t := \#\{s \in (0, t] : X_s-X_{s-} \in E_k\}$.) The total sum of the squares of the jumps taken by time $t$ can be approximated as follows:
    \begin{equation*}
        \sum_{k=1}^\infty \ell^{-2k} r_0^2 N^{(k)}_t \leq \sum_{0<s \leq t} |X_s-X_{s-}|^2 \leq \sum_{k=1}^\infty \ell^{-2(k-1)} r_0^2 N^{(k)}_t.
    \end{equation*}
    Taking expectations,
    \begin{align}
        \mathbb{E}_x \left[ \sum_{0<s \leq t} |X_s-X_{s-}|^2 \right] &\leq \sum_{k=1}^\infty \ell^{-2(k-1)} r_0^2 \, \mathbb{E}_x \left[ N^{(k)}_t \right] = \sum_{k=1}^\infty \ell^{-2(k-1)} r_0^2 \int_{E_k} j(|x|) \, dx \cdot t \notag\\
        &\leq \sum_{k=1}^\infty \int_{E_k} \ell^2 |x|^2 j(|x|) \, dx \cdot t = \ell^2 m_2(X) t \label{EHI:m2tsecondmoment2}
    \end{align}
    and
    \begin{equation}\label{EHI:m2tsecondmoment3}
        \mathbb{E}_x \left[ \sum_{0<s \leq t} |X_s-X_{s-}|^2 \right] \geq \ell^{-2} m_2(X) t.
    \end{equation}
    Since \eqref{EHI:m2tsecondmoment2} and \eqref{EHI:m2tsecondmoment3} hold for all $\ell>1$,
    \begin{equation}\label{EHI:m2tsecondmoment4}
        \mathbb{E}_x \left[ \sum_{0<s \leq t} |X_s-X_{s-}|^2 \right] = m_2(X) t.
    \end{equation}
    By \eqref{EHI:m2tsecondmoment1} and \eqref{EHI:m2tsecondmoment4}, the proof is complete.
\end{proof}

\begin{lemma} \label{EHI:l:PE}
    Let $X \in \mathcal{I}(d, r_0)$, and let $T$ be a random variable with Exponential($\lambda$) distribution, independent of $X$. For all $x \in \mathbb{R}^d$ and $r>0$, we have each of the following:
    \begin{equation} \label{EHI:e:PEP}
        \mathbb{P}_x(\tau_{B(x, r)}>T) \geq (1-e^{-1}) \left(1 - \frac{2m_2(X)}{r^2 \lambda} \right)
    \end{equation}
    and
    \begin{equation} \label{EHI:e:PEE}
        \mathbb{E}_x \left[ \tau_{B(x, r)} \wedge T \right] \geq e^{-1} \left(1 - \frac{2m_2(X)}{r^2 \lambda} \right) \frac{1}{\lambda} = e^{-1} \left(1 - \frac{2m_2(X)}{r^2 \lambda} \right) \mathbb{E}[T].
    \end{equation}
\end{lemma}

\begin{proof}
    For all $t>0$, by Lemma \ref{EHI:l:FactorOf2ForExits}, Markov's inequality, and Lemma \ref{EHI:l:m2isvariance},
    \begin{equation}\label{EHI:e:escapeUnlikely}
        \P_x(\tau_{B(x, r)} \leq t) \leq 2\P_x (|X_t-x| \geq r) \leq 2r^{-2} \E_x \left( |X_t-x|^2 \right) = \frac{2m_2(X)t}{r^2}.
    \end{equation}
    By the independence of $X$ and $T$, and by \eqref{EHI:e:escapeUnlikely},
    \begin{align*}
        \mathbb{P}_x(\tau_{B(x, r)}>T) &\geq \mathbb{P}_x \left(\tau_{B(x, r)} > \frac{1}{\lambda} > T \right) = \mathbb{P}_x \left(\tau_{B(x, r)} > \frac{1}{\lambda} \right) \mathbb{P}\left( T \leq \frac{1}{\lambda} \right) \\
        &\geq \left(1-\frac{2m_2(X)}{r^2 \lambda} \right) (1-e^{-1}).
    \end{align*}
    By Markov's inequality and \eqref{EHI:e:escapeUnlikely},
    \begin{align*}
        \mathbb{E}_x \left[ \tau_{B(x, r)} \wedge T \right] &\geq \frac{1}{\lambda} \mathbb{P}_x \left(\tau_{B(x, r)} \wedge T > \frac{1}{\lambda} \right) = \frac{1}{\lambda} \mathbb{P}_x \left(\tau_{B(x, r)} > \frac{1}{\lambda} \right) \mathbb{P}\left( T > \frac{1}{\lambda} \right) \\
        &\geq \frac{1}{\lambda}\left(1-\frac{2m_2(X)}{r^2 \lambda} \right) e^{-1} .
    \end{align*}
\end{proof}

\subsection{Quadratic variation and Berry-Esseen}\label{EHI:ss:BerryEsseen}

Lemma \ref{EHI:l:m2isvariance} is useful in analyzing the distance a process $X \in \mathcal{I}(d, r_0)$ travels over a certain time. In this subsection, we develop another tool to allow for an even more refined analysis.
Given an isotropic unimodal L\'{e}vy jump process $X$, consider its quadratic variation
\begin{equation} \label{EHI:quadraticVariationDef}
    [X]_t := \sum_{0<s \leq t} |\Delta X(s)|^2,
\end{equation}
The process $[X]=([X]_t)_{t \geq 0}$ is a non-decreasing L\'{e}vy process on $[0, \infty)$. If $X \in \mathcal{I}(d, r_0)$, then $[X]$ only takes jumps of size $r_0^2$ or less, and growths with a ``speed" of $m_2(X)$. The relation between $X$ and $[X]$ is somewhat analogous to the relationship between a subordinate Brownian motion and its subordinator, in that $[X]$ is a one-dimensional, non-decreasing L\'{e}vy process and $[X]_t$ is usually on the order of $|X_t-X_0|^2$. We show in this subsection that the analogy goes further: conditional on $[X]_t$, $X_t$ can be approximated as a Gaussian with variance\footnote{The variance of a random vector refers to its covariance matrix.} proportional to $[X]_t I_d$, using a multidimensional form of the Berry-Esseen theorem (itself a quantitative version of the Central limit theorem), and conditioning on the $\sigma$-field generated by $[X]$.

We begin with the following lemma, which establishes something like a Laplace exponent for $[X]$.

\begin{lemma}\label{EHI:l:ChernoffExpectation}
    If $X$ is an isotropic unimodal L\'{e}vy jump process on $\R^d$, then
    \begin{equation*}
        \mathbb{E}\left[e^{-\lambda[X]_t} \right] = \exp(-t \int_{\mathbb{R}^d} (1-e^{-\lambda|x|^2}) j(|x|) \, dx)
    \end{equation*}
    for all $\lambda>0, t>0$.
\end{lemma}

Before we prove Lemma \ref{EHI:l:ChernoffExpectation}, we need the following simple fact.

\begin{lemma} \label{EHI:l:poisson}
    Fix $\lambda_1, \lambda_2 > 0$. If a random variable $N$ has Poisson($\lambda_1$) distribution, then
    \begin{equation*}
        \mathbb{E}[e^{-\lambda_2 N}] = \exp(-\lambda_1(1-e^{-\lambda_2})).
    \end{equation*}
\end{lemma}

\begin{proof}
    The proof is a simple calculation:
    \begin{align*}
        \mathbb{E}[e^{-\lambda_2 N}] &= \sum_{n=0}^\infty e^{-\lambda_1} \frac{\lambda_1^n}{n!} e^{-\lambda_2 n} = e^{-\lambda_1} \sum_{n=0}^\infty \frac{(\lambda_1 e^{-\lambda_2})^n}{n!} \\
        &= \exp(-\lambda_1) \exp(\lambda_1 e^{-\lambda_2}) = \exp(-\lambda_1 (1-e^{-\lambda_2})) .
    \end{align*}
\end{proof}

\begin{proof}[Proof of Lemma \ref{EHI:l:ChernoffExpectation}]
    Fix $\ell>1$. For all $k \in \mathbb{Z}$, let $E_k$ be the annulus $E_k := \{x \in \mathbb{R}^d : \ell^k \leq |x| < \ell^{k+1}\}$. Then $\{E_k\}_{k \in \mathbb{Z}}$ is a partition of $\mathbb{R}^d \setminus \{0\}$, so each jump that $X$ takes has a displacement in $E_k$ for some unique $k$. For all $k$, let
    \begin{equation*}
        [X]^{(k)}_t := \sum_{0<s\leq t} \Delta |X(s)|^2 \indicatorWithSetBrackets{\Delta X(s) \in E_k}
    \end{equation*}
    and
    \begin{equation*}
        N^{(k)}_t := \# \{ 0<s\leq t : \Delta X(s) \in E_k \}.
    \end{equation*}
    Let $\lambda_k := \int_{E_k} j(|x|) \, dx$, so that $N^{(k)}_t$ is Poisson($\lambda_k t$).

    Whenever $X$ takes a jump whose displacement vector is in $E_k$, $[X]^{(k)}$ increases by a value between $\ell^{2k}$ and $\ell^{2(k+1)}$. Thus,
    \begin{equation} \label{EHI:VtAndNt}
        \ell^{2k} N^{(k)}_t \leq [X]^{(k)}_t \leq \ell^{2(k+1)} N^{(k)}_t.
    \end{equation}
    The collection $\left\{ [X]^{(k)}_t \right\}_{k=1}^\infty$ is independent and its sum is $[X]_t$. Thus, for all $\lambda>0$,
    \begin{equation} \label{EHI:ProductForChernoff}
        \mathbb{E}\left[ e^{-\lambda [X]_t} \right] = \mathbb{E}\left[ \prod_{k=1}^\infty e^{-\lambda [X]^{(k)}_t} \right] = \prod_{k=1}^\infty \mathbb{E}\left[ e^{-\lambda [X]^{(k)}_t} \right].
    \end{equation}
    For all $k$, by \eqref{EHI:VtAndNt} and Lemma \ref{EHI:l:poisson} (a lemma in the appendix that tells us how to compute $\mathbb{E}[e^{-\lambda_2 N}]$ when $N$ is Poisson($\lambda_1$)), and by the fact that $|x| \geq \ell^k$ for all $x \in E_k$,
    \begin{align} \label{EHI:e:ApplyPoissonLemma}
        \mathbb{E}\left[ e^{-\lambda[X]^{(k)}_t} \right] &\leq \mathbb{E} \left[ \exp(-\lambda \ell^{2(k+1)} N^{(k)}_t) \right] = \exp(-\lambda_k t \left(1-e^{-\lambda \ell^{2(k+1)}} \right)) \notag\\
        &= \exp(-t\int_{E_k} \left(1-e^{-\lambda \ell^{2(k+1)}} \right) j(|x|) \, dx ) \notag\\
        &\leq \exp(-t\int_{E_k} \left(1-e^{-\lambda \ell^2 |x|^2} \right) j(|x|) \, dx ).
    \end{align}
    By \eqref{EHI:ProductForChernoff}, taking the product of \eqref{EHI:e:ApplyPoissonLemma} over all $k$ gives
    \begin{equation*}
        \mathbb{E}\left[ e^{-\lambda [X]_t} \right] \leq \exp(-t\int_{\mathbb{R}^d} \left(1-e^{-\lambda \ell^2 |x|^2} \right) j(|x|) \, dx ).
    \end{equation*}
    Since this holds for all $\ell>1$, taking the limit as $\ell\to 1^+$ (using the Dominated convergence theorem) gives
    \begin{equation}\label{EHI:e:ChernoffExpectationIneq}
        \mathbb{E}\left[ e^{-\lambda [X]_t} \right] \leq \exp(-t\int_{\mathbb{R}^d} \left(1-e^{-\lambda |x|^2} \right) j(|x|) \, dx ).
    \end{equation}
    By repeating the same argument but using the lower bound from \eqref{EHI:VtAndNt} instead of the upper bound, and again taking the limit as $\ell \to 1^+$, we find that the inequality in \eqref{EHI:e:ChernoffExpectationIneq} is in fact an equality.
\end{proof}

The Berry-Esseen theorem gives a quantitative version of the Central limit theorem, with an error term that depends on the third moments of the independent random variables being summed. The following multivariate version of the theorem, due to Rai\v{c} \cite[Theorem 1.1]{Rai}, does not require the random variables to be identically-distributed.

\begin{theorem}[Multivariate Berry-Esseen, \cite{Rai}] \label{EHI:t:BerryEsseen}
    Let $(\xi_i)_{i \in I}$ be an independent collection of random vectors in $\mathbb{R}^d$, each with mean $0$. Assume that $\Var(\sum_{i \in I} \xi_i) = \sigma^2 I_d$, where $\sigma>0$. Let $Z$ be a standard multivariate Gaussian (with mean $0$ and variance $I_d$). Then for all measurable convex sets $A \subseteq \mathbb{R}^d$,
    \begin{equation*}
        \left| \mathbb{P}\left( \sum_{i \in I} \xi_i \in A \right) - \mathbb{P}\left( \sigma Z \in A \right) \right| \leq C_{BE} \cdot \frac{\sum_{i \in I} \mathbb{E}|\xi_i|^3}{\sigma^3},
    \end{equation*}
    where $C_{BE} = 42d^{1/4}+16$ (which depends only on $d$).
\end{theorem}

In \cite{Rai}, the theorem is expressed slightly differently, with $\Var(\sum_{i \in I} \xi_i)=I_d$. Our formulation, which allows $\Var(\sum_{i \in I} \xi_i)=\sigma^2 I_d$ for any $\sigma>0$, can be obtained by a simple renormalization.

We would like to put a lower bound on $\P_x(|X_t-x|>r)$. It is not enough to use Lemma \ref{EHI:l:m2isvariance} and Markov's or Chebyshev's inequality, since that would only give us an upper bound. Instead, we use Berry-Esseen.
In the following lemma, we condition on the values of $|\Delta X(s)|$ for all $0<s \leq t$, and apply Theorem \ref{EHI:t:BerryEsseen} to the conditional distribution of $X_t = X_0+\sum_{0<s \leq t} \Delta X(s)$. The result is that $\P_x(|X_t-x|>r)$ can be bounded below by the probability of $[X]_t$ being sufficiently large.

\begin{lemma} \label{EHI:l:LeaveBallConstantHelper}
    For all $\eps \in (0, 1)$, there exists a universal constant $C_{\ref{EHI:l:LeaveBallConstantHelper}}(\eps)=C_{\ref{EHI:l:LeaveBallConstantHelper}}(\eps, d)>0$ such that if $X \in \mathcal{I}(d, r_0)$, then for all $r>0$ and $x \in \R^d$,
    \begin{equation*}
        \P_x \left( |X_t-x| > r \right) \geq (1-\eps) \P \left( [X]_t \geq C_{\ref{EHI:l:LeaveBallConstantHelper}}(\eps) (r \vee r_0)^2\right).
    \end{equation*}
\end{lemma}

\begin{proof}
    Let $(\mathcal{G}_t)_{t\geq0}$ be the filtration of $\sigma$-fields that takes into account only the magnitude of the jumps $X$ takes at each time:
    \begin{equation*}
        \mathcal{G}_t := \sigma\left(\left\{ |\Delta X(s)| : 0<s\leq t \right\} \right).
    \end{equation*}
    The process $[X] = ([X]_t)_{t \geq 0}$ is adapted to $(\mathcal{G}_t)_{t \geq 0}$. Let us condition on $\mathcal{G}_t$ for some fixed $t>0$. Suppose the magnitudes of the jumps, and when they occur, are determined. Given this information, let us analyze the conditional distribution of $X_t$.
    For simplicity, let us assume without loss of generality the process begins at $X_0=0$, so $X_t = \sum_{0<s \leq t} \Delta X(s)$.
    By summing over all $s$, the conditional variance of $X_t$ is
    \begin{equation*}
        \Var(X_t | \mathcal{G}_t) = \sum_{0<s \leq t} d^{-1} |\Delta X(s)|^2 I_d = d^{-1} [X]_t I_d.
    \end{equation*}
    (The variance of a uniform point on the sphere $\{x:|x|=r\}$ is $d^{-1} r^2 I_d$.) Since each value of $|\Delta X(s)|^3$ is $\mathcal{G}_t$-measurable, by Theorem \ref{EHI:t:BerryEsseen}, for all convex sets $A \subseteq \mathbb{R}^d$, we have
    \begin{equation*}
        \left| \mathbb{P}_0(X_t \in A | \mathcal{G}_t) - \mathbb{P} \left( \sqrt{d^{-1} [X]_t} Z \in A \right) \right| \leq C_{BE} \frac{\sum_{0<s \leq t} |\Delta X(s)|^3}{(d^{-1}[X]_t)^{3/2}} = C_1 \frac{\sum_{0<s \leq t} |\Delta X(s)|^3}{([X]_t)^{3/2}}
    \end{equation*}
    where $C_1 := d^{3/2} C_{BE}$.
    The same result holds for sets whose complements are convex, by considering the complements of the probabilities.
    Therefore, we can take $A$ to be $B(0, r)^c$, and obtain
    \begin{equation} \label{EHI:BeforePhi}
        \mathbb{P}_0 \left(|X_t|>r \Big| \mathcal{G}_t \right) \geq \mathbb{P}\left( |Z| > r\sqrt{\frac{d}{[X]_t}} \right) - C_1 \frac{\sum_{0<s \leq t} |\Delta X(s)|^3}{([X]_t)^{3/2}}.
    \end{equation}
    Let $F_{|Z|}:[0, \infty) \to [0, 1)$ be the cumulative distribution function $F_{|Z|}(z)=\mathbb{P}(|Z| \leq z)$. (Recall that $Z$ is a standard $d$-dimensional Gaussian, so $F_{|Z|}$ depends on $d$.)
    Then \eqref{EHI:BeforePhi} can be written as
    \begin{equation} \label{EHI:AfterPhi}
        \mathbb{P}_0 \left(|X_t|>r \Big| \mathcal{G}_t \right) \geq 1- F_{|Z|} \left( r\sqrt{\frac{d}{[X]_t}}\right) - C_1 \frac{\sum_{0<s \leq t} |\Delta X(s)|^3}{([X]_t)^{3/2}}.
    \end{equation}
    Note that $\sum_{0<s \leq t} |\Delta X(s)|^3 \leq r_0 [X]_t$, since $|\Delta X(s)|^3 \leq r_0 |\Delta X(s)|^2$ for all $s$. Also, $\mathbb{P}_0 \left(|X_t|>r \Big| \mathcal{G}_t \right) \geq 0$ almost surely. Thus, \eqref{EHI:AfterPhi} implies
    \begin{equation*}
        \mathbb{P}_0 \left(|X_t|>r \Big| \mathcal{G}_t \right) \geq \left( 1-F_{|Z|} \left( r\sqrt{\frac{d}{[X]_t}}\right) - C_1 \frac{r_0}{\sqrt{[X]_t}} \right)_+.
    \end{equation*}
    Taking the expectation,
    \begin{equation} \label{EHI:CondOnGt}
        \mathbb{P}_0 \left(|X_t|>r \right) \geq \mathbb{E}_0 \left[ \left( 1-F_{|Z|} \left( r\sqrt{\frac{d}{[X]_t}}\right) - C_1 \frac{r_0}{\sqrt{[X]_t}} \right)_+ \right].
    \end{equation}

    We will show that when $[X]_t / (r \vee r_0)^2$ is sufficiently large, $F_{|Z|}(r\sqrt{d/[X]_t})$ and $C_1 r_0/\sqrt{[X]_t}$ are each at most $\eps/2$, so the quantity whose expectation is being taken in \eqref{EHI:CondOnGt} is at least $1-\eps$.
    Let $c_2(\eps)$ be the value such that $F_{|Z|}(c_2(\eps))=\eps/2$, and note that
    \begin{equation*}
        \left\{ F_{|Z|} \left( r \sqrt{\frac{d}{[X]_t}} \right) \leq \frac{\eps}{2} \right\} \quad=\quad \left\{ r \sqrt{\frac{d}{[X]_t}} \leq c_2(\eps) \right\} \quad=\quad \left\{ [X]_t \geq d (c_2(\eps))^{-2} r^2 \right\}
    \end{equation*}
    and
    \begin{equation*}
        \quad \left\{ C_1 \frac{r_0}{\sqrt{[X]_t}} \leq \frac{\eps}{2}\right\} \quad=\quad \left\{ [X]_t \geq 4 C_1^2 \eps^{-2} r_0^2 \right\}.
    \end{equation*}
    Let $C_{\ref{EHI:l:LeaveBallConstantHelper}}(\eps) := \max\{ d (c_2(\eps))^{-2}, 4 C_1^2 \eps^{-2} \}$. Then whenever $[X]_t \geq C_{\ref{EHI:l:LeaveBallConstantHelper}}(\eps)(r \vee r_0)^2$, we have both $F_{|Z|}(r \sqrt{d/[X]_t}) \leq \eps/2$ and $C_1 \frac{r_0}{\sqrt{[X]_t}} \leq \eps/2$.
    By this fact and \eqref{EHI:CondOnGt},
    \begin{equation*}
        \P_0 (|X_t|>r) \geq (1-\eps ) \P \left( [X]_t \geq C_{\ref{EHI:l:LeaveBallConstantHelper}}(\eps) (r \vee r_0)^2 \right).
    \end{equation*}
\end{proof}

Next, we use Lemma \ref{EHI:l:ChernoffExpectation} to obtain a lower bound on the probability that $[X]_t \geq C_{\ref{EHI:l:LeaveBallConstantHelper}}(\eps) (r \vee r_0)^2$.

\begin{lemma} \label{EHI:l:LeaveBallConstant}
    For all $\eps \in (0, 1)$, there exists a universal constant $C_{\ref{EHI:l:LeaveBallConstant}}(\eps)=C_{\ref{EHI:l:LeaveBallConstant}}(\eps, d)>0$ such that if $X \in \mathcal{I}(d, r_0)$, then for all $r>0$ and $x \in \R^d$,
    \begin{equation*}
        \P_x \left( \tau_{B(x, r)} > C_{\ref{EHI:l:LeaveBallConstant}}(\eps) \frac{(r \vee r_0)^2}{m_2(X)} \right) \leq \eps. 
    \end{equation*}
\end{lemma}

\begin{proof}
    Let $\eps \in (0, 1)$, $X \in \mathcal{I}(d, r_0)$, $r>0$, and $x \in \R^d$. Fix $t>0$, $\mu>0$, and $\delta>0$, all of which will be specified later. By Lemma \ref{EHI:l:LeaveBallConstantHelper},
    \begin{equation}\label{EHI:e:LeaveBallConstant1}
        \P_x(|X_t-x| > r) \geq (1-\delta) \P\left([X]_t \geq C_{\ref{EHI:l:LeaveBallConstantHelper}}(\delta)(r \vee r_0)^2 \right).
    \end{equation}
    By Markov's inequality and Lemma \ref{EHI:l:ChernoffExpectation},
    \begin{align}\label{EHI:e:LeaveBallConstant2}
        \P\left([X]_t < C_{\ref{EHI:l:LeaveBallConstantHelper}}(r \vee r_0)^2 \right) &= \P\left(e^{-\lambda [X]_t} > \exp(-C_{\ref{EHI:l:LeaveBallConstantHelper}}(\delta) \mu (r \vee r_0)^2) \right) \notag\\
        &\leq \exp(C_{\ref{EHI:l:LeaveBallConstantHelper}}(\delta) \mu (r \vee r_0)^2) \E \left[ e^{-\mu [X]_t} \right] \notag\\
        &= \exp( C_{\ref{EHI:l:LeaveBallConstantHelper}}(\delta) \mu (r \vee r_0)^2 - t \int_{B(0, r_0)} (1-e^{-\mu|x|^2}) j(|x|) \dee{x} ).
    \end{align}
    The integral in \eqref{EHI:e:LeaveBallConstant2} is over $B(0, r_0)$ instead of $\R^d$ because $X \in \mathcal{I}(d, r_0)$. We would like to replace the integrand $1-e^{-\mu|x|^2}$ with some multiple of $\mu|x|^2$, using the approximation $1-e^{-u} \approx u$.
    To make this precise, note that the function $u \mapsto \frac{1-e^{-u}}{u}$ is decreasing, so $1-e^{-\mu|x|^2} \geq \frac{1-e^{-\mu r_0^2}}{\mu r_0^2} \mu |x|^2$ for all $x \in B(0, r_0)$. Thus,
    \begin{equation}\label{EHI:e:LeaveBallConstant3}
        \int_{B(0, r_0)} (1-e^{-\mu|x|^2}) j(|x|) \dee{x} \geq \frac{1-e^{-\mu r_0^2}}{\mu r_0^2} \int_{B(0, r_0)} \mu|x|^2 j(|x|) \dee{x} = \frac{1-e^{-\mu r_0^2}}{\mu r_0^2} \mu m_2(X).
    \end{equation}
    Combining \eqref{EHI:e:LeaveBallConstant2} and \eqref{EHI:e:LeaveBallConstant3},
    \begin{equation}\label{EHI:e:LeaveBallConstant4}
        \P\left([X]_t < C_{\ref{EHI:l:LeaveBallConstantHelper}}(r \vee r_0)^2 \right) \leq \exp( C_{\ref{EHI:l:LeaveBallConstantHelper}}(\delta) \mu (r \vee r_0)^2 - \frac{1-e^{-\mu r_0^2}}{\mu r_0^2} \mu m_2(X) t ).
    \end{equation}
    Recall that $\mu$, $t$, and $\delta$ have been unspecified up to now. We now specify them one-by-one.
    Let $\mu:=r_0^{-2}$. Then \eqref{EHI:e:LeaveBallConstant4} becomes
    \begin{equation}\label{EHI:e:LeaveBallConstant5}
        \P\left([X]_t < C_{\ref{EHI:l:LeaveBallConstantHelper}}(r \vee r_0)^2 \right) \leq \exp( C_{\ref{EHI:l:LeaveBallConstantHelper}}(\delta) \cdot \left( \frac{r \vee r_0}{r_0} \right)^2 - (1-e^{-1}) \frac{m_2(X)}{r_0^2} t ).
    \end{equation}
    Let
    \begin{equation*}
        t:= \frac{C_{\ref{EHI:l:LeaveBallConstantHelper}}(\delta)+\log(\delta^{-1})}{1-e^{-1}} \cdot \frac{(r \vee r_0)^2}{m_2(X)}.
    \end{equation*}
    Then \eqref{EHI:e:LeaveBallConstant5} becomes
    \begin{align}\label{EHI:e:LeaveBallConstant6}
        \P\left([X]_t < C_{\ref{EHI:l:LeaveBallConstantHelper}}(r \vee r_0)^2 \right) \leq \exp(-\log(\delta^{-1}) \cdot \left( \frac{r \vee r_0}{r_0} \right)^2) \leq \delta.
    \end{align}
    By \eqref{EHI:e:LeaveBallConstant1} and \eqref{EHI:e:LeaveBallConstant6},
    \begin{equation*}
        \P_x(|X_t-x| > r) \geq (1-\delta)^2.
    \end{equation*}
    Finally, let $\delta:=\eps/2$, so that $\P_x(|X_t-x| > r) \geq (1-\eps/2)^2 \geq 1-\eps$. Note that if $|X_t-x|>r$, then $\tau_{B(x, r)} \leq t$. In conclusion,
    \begin{equation*}
        \P_x \left(\tau_{B(x, r)} \leq \frac{C_{\ref{EHI:l:LeaveBallConstantHelper}}(\eps/2)+\log(\frac{1}{2\eps})}{1-e^{-1}} \cdot \frac{(r \vee r_0)^2}{m_2(X)} \right) = \P_x \left(\tau_{B(x, r)} \leq t \right) \geq \P_x(|X_t-x|>r) \geq 1-\eps.
    \end{equation*}
    This completes the proof, with
    \begin{equation*}
        C_{\ref{EHI:l:LeaveBallConstant}}(\eps) := \frac{C_{\ref{EHI:l:LeaveBallConstantHelper}}(\eps/2)+\log(\frac{1}{2\eps})}{1-e^{-1}}.
    \end{equation*}
\end{proof}

By a chaining argument, we can improve Lemma \ref{EHI:l:LeaveBallConstant} to get that $\P_x(\tau_{B(x, r)}>t)$ decays exponentially in $t$, with the rate of decay proportional to $m_2(X)/(r \vee r_0)^2$.

\begin{lemma}\label{EHI:l:LeaveBallExponential}
    There exists a universal constant $c_{\ref{EHI:l:LeaveBallExponential}}=c_{\ref{EHI:l:LeaveBallExponential}}(d)>0$ such that if $X \in \mathcal{I}(d, r_0)$, then for all $r>0$, $x \in \R^d$, and $t>0$,
    \begin{equation*}
        \P_0(\tau_{B(x, r)}>t) \leq \exp(-\floor{c_{\ref{EHI:l:LeaveBallExponential}} \frac{m_2(X)}{(r \vee r_0)^2} t}).
    \end{equation*}
\end{lemma}

\begin{proof}
    Let $t_0 := C_{\ref{EHI:l:LeaveBallConstant}}(e^{-1})(2r \vee r_0)/m_2(X)$. Consider the times $0, t_0, 2t_0, \dots, nt_0$, where $n := \floor{t/t_0}$. In order for the process to remain in the ball $B(x, r)$ up until time $t$, we must have $X_{jt_0} \in B(x, r)$ for all $0 \leq j \leq n$. This means that $\left|X_{(j+1)t_0} - X_{j t_0} \right| \leq 2r$ for all $0 \leq j<n$. Thus, by the Strong Markov property and Lemma \ref{EHI:l:LeaveBallConstant},
    \begin{align*}
        \P_x(\tau_{B(x, r)} > t) & \leq \P_x \left( \mbox{$\left|X_{(j+1)t_0} - X_{j t_0} \right| \leq 2r$ for all $0 \leq j<n$} \right) = \left( \P_0(\tau_{B(0, r)} > t_0) \right)^n \\
        &\leq e^{-n} =e^{-\floor{t/t_0}} = \exp(-\floor{\frac{1}{C_{\ref{EHI:l:LeaveBallConstant}}(e^{-1})} \cdot \frac{m_x(X)}{(2r \vee r_0)^2} t}) \\
        &\leq \exp(-\floor{\frac{1}{4C_{\ref{EHI:l:LeaveBallConstant}}(e^{-1})} \cdot \frac{m_x(X)}{(r \vee r_0)^2} t}).
    \end{align*}
\end{proof}

\section{Proof of main results and examples}\label{EHI:s:positiveResults}

In this section, we prove Theorem \ref{EHI:t:main} (our main result) and the corollaries and examples that follow from it in the introduction.

\subsection{Proof of Theorem \ref{EHI:t:main}}

\begin{lemma} \label{EHI:l:G0upperbound}
    Let $X$ be an isotropic unimodal L\'{e}vy jump process on $\R^d$. For $r, r_0 > 0$ and all $x \in B(0, r)$, we have
    \begin{equation*}
        G^{\inner{r_0}, 0}_{B(0, r)}(0, x) \lesssim \frac{(r \vee r_0)^2 |x|^{-d}}{m_2 \inner{r_0}}.
    \end{equation*}
    The constant implicit in $\lesssim$ depend only on the dimension $d$. 
\end{lemma}

\begin{proof}
    By Lemma \ref{EHI:l:GkPurpose}, the independence of $X^{\inner{r_0}}$ from $\hat{X}^{\inner{r_0}}$, and Lemma \ref{EHI:l:LeaveBallExponential} applied to $X^{\inner{r_0}}$,
    \begin{align} \label{EHI:e:G0upperbound1}
        \int_{B(0, r)} G^{\inner{r_0}, 0}_{B(0, r)}(0, y) \dee{y} &= \int_0^\infty \P_0 \left( \tau_{B(0, r)}>t , T^{\inner{r_0}}>t \right) \dee{t} \notag\\
        &= \int_0^\infty \P_0 \left( \tau^{\inner{r_0}}_{B(0, r)} > t , T^{\inner{r_0}}>t \right) \dee{t} \notag\\
        &\leq \int_0^\infty \P_0 \left( \tau^{\inner{r_0}}_{B(0, r)} > t \right) \dee{t} \notag\\
        &\leq \int_0^\infty \exp( -\floor{c_{\ref{EHI:l:LeaveBallExponential}} \frac{m_2\inner{r_0}}{(r \vee r_0)^2} t } ) \dee{t}
    \end{align}
    Let $t_0 := (r \vee r_0)^2 / (c_{\ref{EHI:l:LeaveBallExponential}} m_2\inner{r_0})$, so that \eqref{EHI:e:G0upperbound1} is equivalent to
    \begin{equation*}\label{EHI:e:G0upperbound2}
        \int_{B(0, r)} G^{\inner{r_0}, 0}_{B(0, r)}(0, y) \dee{y} \leq \int_0^\infty e^{-\floor{t/t_0}} \dee{t} = \sum_{k=0}^\infty t_0 e^{-k} = (1-e^{-1})^{-1} t_0.
    \end{equation*}
    Now fix $x \in B(0, r)$. Note that $G^{\inner{r_0}, 0}_{B(0, r)}(0, \cdot)$ is non-increasing in $|\cdot|$, so the value of $G^{\inner{r_0}, 0}_{B(0, r)}(0, x)$ is at most equal to the average of $G^{\inner{r_0}, 0}_{B(0, r)}(0, y)$ for $y \in B(0, |x|)$. Thus,
    \begin{align*}
        G^{\inner{r_0}, 0}_{B(0, r)}(0, x) &\leq \frac{\int_{B(0, |x|)} G^{\inner{r_0}, 0}_{B(0, r)}(0, y) \dee{y}}{|B(0, r)|} \leq \frac{\int_{B(0, r)} G^{\inner{r_0}, 0}_{B(0, r)}(0, y) \dee{y}}{|x|^d|B(0, 1)|} \\
        &\leq \frac{(1-e^{-1})^{-1}}{|B(0, 1)|} |x|^{-d} t_0 = \frac{(1-e^{-1})^{-1}}{c_{\ref{EHI:l:LeaveBallExponential}}|B(0, 1)|} \cdot \frac{(r \vee r_0)^2 |x|^{-d}}{m_2\inner{r_0}}.
    \end{align*}
\end{proof}

\begin{lemma} \label{EHI:l:G0lowerbound}
    Let $X$ be an isotropic unimodal L\'{e}vy jump process on $\R^d$.
    There exist universal constants $L_{\ref{EHI:l:G0lowerbound}} >1$ and $c_{\ref{EHI:l:G0lowerbound}}>0$ such that if $0<r_0 \leq r/L_{\ref{EHI:l:G0lowerbound}} \leq r$, then
    \begin{equation*}
        G^{\inner{r_0}, 0}_{B(0, r)}(0, x) \geq c_{\ref{EHI:l:G0lowerbound}} \frac{r^{2-d}}{m_2\inner{r_0}} \exp(-\frac{1}{10} \cdot \frac{r^2 \lambda\inner{r_0}}{m_2\inner{r_0}}) \qquad\mbox{for all $x \in B(0, r/L_{\ref{EHI:l:G0lowerbound}})$}.
    \end{equation*}
\end{lemma}

\begin{proof}
    Let
    \begin{equation*}
        L_{\ref{EHI:l:G0lowerbound}} := \sqrt{ 20 C_{\ref{EHI:l:LeaveBallConstant}}\left(\frac{1}{10} \right) }.
    \end{equation*}
    Suppose $0<r_0 \leq r/L_{\ref{EHI:l:G0lowerbound}} \leq r$ and fix $x \in B(0, r/L_{\ref{EHI:l:G0lowerbound}})$.
    Let $s:=|x|$ and $A:=B(0, r) \setminus \overline{B(0, s)}$. We will put a lower bound on $\int_A G^{\inner{r_0}, 0}_{B(0, r)}(0, y) \dee{y}$, and then conclude that $G^{\inner{r_0}, 0}_{B(0, r)}(0, x)$ is at least the average value of $G^{\inner{r_0}, 0}_{B(0, r)}(0, y)$ for $y \in A$, much like in the proof of Lemma \ref{EHI:l:G0upperbound}, because $G^{\inner{r_0}, 0}_{B(0, r)}(0, \cdot)$ is non-increasing in $|\cdot|$.

    By Lemma \ref{EHI:l:GkPurpose} and the independence of $X^{\inner{r_0}}$ from $\hat{X}^{\inner{r_0}}$,
    \begin{align}\label{EHI:e:G0lowerbound1}
        \int_A G^{\inner{r_0}, 0}_{B(0, r)}(0, y) \dee{y} &= \int_0^\infty \P_0 \left( X_t \in A, \tau_{B(0, r)}>t, T^{\inner{r_0}}>t \right) \dee{t} \notag\\
        &= \int_0^\infty \P_0 \left( X^{\inner{r_0}}_t \in A, \tau^{\inner{r_0}}_{B(0, r)}>t, T^{\inner{r_0}}>t \right) \dee{t} \notag\\
        &= \int_0^\infty e^{-\lambda\inner{r_0}t} \P_0 \left( X^{\inner{r_0}}_t \in A, \tau^{\inner{r_0}}_{B(0, r)}>t \right) \dee{t} \notag\\
        &= \int_0^\infty e^{-\lambda\inner{r_0}t} \left( \P_0 \left( X^{\inner{r_0}}_t \in A \right) - \P_0 \left( X^{\inner{r_0}}_t \in A, \tau^{\inner{r_0}}_{B(0, r)} \leq t \right) \right) \dee{t}.
    \end{align}
    For all $t>0$,
    \begin{equation}\label{EHI:e:G0lowerbound2}
        \P_0 \left( X^{\inner{r_0}}_t \in A \right) = \P_0 \left( |X^{\inner{r_0}}_t| \geq s \right) - \P_0 \left( |X^{\inner{r_0}}_t| \geq r \right)
    \end{equation}
    and by Lemma \ref{EHI:l:FactorOf2ForExits},
    \begin{equation}\label{EHI:e:G0lowerbound3}
        \P_0 \left( X^{\inner{r_0}}_t \in A, \tau^{\inner{r_0}}_{B(0, r)} \leq t \right) \leq \P_0 \left( \tau^{\inner{r_0}}_{B(0, r)} \leq t \right) \leq 2\P_0 \left( |X^{\inner{r_0}}_t| > r \right).
    \end{equation}
    By substituting \eqref{EHI:e:G0lowerbound2} and \eqref{EHI:e:G0lowerbound3} into \eqref{EHI:e:G0lowerbound1},
    \begin{equation}\label{EHI:e:G0lowerbound4}
        \int_A G^{\inner{r_0}, 0}_{B(0, r)}(0, y) \dee{y} \geq \int_0^\infty e^{-\lambda\inner{r_0}t} \left( \P_0 \left( |X^{\inner{r_0}}_t| \geq s \right) - 3\P_0 \left( |X^{\inner{r_0}}_t| \geq r \right) \right) \dee{t}.
    \end{equation}
    
    Let us put a lower bound on the integrand of \eqref{EHI:e:G0lowerbound4} for all $t \in \left[\frac{r^2}{20 m_2\inner{r_0}} , \frac{r^2}{10 m_2\inner{r_0}} \right]$. Fix $t$ in this set. We then have
    \begin{equation}\label{EHI:e:G0lowerbound5}
        e^{-\lambda\inner{r_0}t} \geq \exp(-\frac{1}{10} \cdot \frac{r^2 \lambda\inner{r_0}}{m_2\inner{r_0}}).
    \end{equation}
    By Markov's inequality,
    \begin{equation}\label{EHI:e:G0lowerbound6}
        \P_0 \left( |X^{\inner{r_0}}_t| \geq r \right) \leq r^{-2} \E_0 \left[ |X^{\inner{r_0}}_t|^{-2} \right] = \frac{m_2\inner{r_0} t}{r^2} \leq \frac{1}{10}.
    \end{equation}
    We still need a lower bound on $\P_0 \left( |X^{\inner{r_0}}_t| \geq s \right)$.
    Recall that $(s \vee r_0)=(|x| \vee r_0) \leq r/L_{\ref{EHI:l:G0lowerbound}}$ and $t \geq \frac{r^2}{20 m_2\inner{r_0}}$. Thus,
    \begin{equation*}
        t \geq \frac{L_{\ref{EHI:l:G0lowerbound}}^{2} (s \vee r_0)^2}{20 m_2\inner{r_0}} = C_{\ref{EHI:l:LeaveBallConstant}}\left(\frac{1}{10}\right) \frac{(s \vee r_0)^2}{m_2\inner{r_0}}.
    \end{equation*}
    Therefore, by Lemmas \ref{EHI:l:FactorOf2ForExits} and \ref{EHI:l:LeaveBallConstant} (applied to $X^{\inner{r_0}}$),
    \begin{align}\label{EHI:e:G0lowerbound7}
        \P_0 \left( |X^{\inner{r_0}}_t| \geq s \right) &\geq \frac12 \P_0 \left( \tau^{\inner{r_0}}_{B(0, s)} \leq t \right) \notag\\
        &\geq \frac12 \P_0 \left( \tau^{\inner{r_0}}_{B(0, s)} \leq C_{\ref{EHI:l:LeaveBallConstant}}\left(\frac{1}{10}\right) \frac{(s \vee r_0)^2}{m_2\inner{r_0}} \right) \notag\\
        &\geq \frac{1}{2}\left(1-\frac{1}{10}\right) = \frac{9}{20}.
    \end{align}
    Substituting \eqref{EHI:e:G0lowerbound5}-\eqref{EHI:e:G0lowerbound7} into the integral from \eqref{EHI:e:G0lowerbound4} for $t \in \left[\frac{r^2}{20 m_2\inner{r_0}} , \frac{r^2}{10 m_2\inner{r_0}} \right]$, we obtain
    \begin{equation*}
        \int_A G^{\inner{r_0}, 0}_{B(0, r)}(0, y) \dee{y} \geq \frac{r^2}{20 m_2\inner{r_0}} \cdot \exp(-\frac{1}{10} \cdot \frac{r^2 \lambda\inner{r_0}}{m_2\inner{r_0}}) \cdot \frac{3}{20}.
    \end{equation*}
    Since $|A| \leq |B(0, r)| = r^d |B(0, 1)|$ and $G^{\inner{r_0}, 0}_{B(0, r)}(0, x)$ is at least the average value of $G^{\inner{r_0}, 0}_{B(0, r)}(0, y)$ for $y \in A$,
    \begin{equation*}
        G^{\inner{r_0}, 0}_{B(0, r)}(0, x) \geq \frac{3}{400 |B(0, 1)|} \cdot \frac{r^{2-d}}{m_2\inner{r_0}} \exp(-\frac{1}{10} \cdot \frac{r^2 \lambda\inner{r_0}}{m_2\inner{r_0}}).
    \end{equation*}
\end{proof}

\begin{prop}\label{EHI:p:EHIbigm2}
    Let $X$ be an isotropic unimodal L\'{e}vy jump process on $\R^d$, satisfying \eqref{EHI:e:regularJumps}. Suppose there exist $c>0$, and a set $E \subseteq (0, \infty)$ such that
    \begin{equation} \label{EHI:e:bigm2Condition}
        m_2(r) \geq c r^2 \lambda(r) \qquad\mbox{for all $r \in E$}.
    \end{equation}
    Let $L_{\ref{EHI:l:G0lowerbound}}$ be the constant from Lemma \ref{EHI:l:G0lowerbound}. For all $L \geq L_{\ref{EHI:l:G0lowerbound}}$ and $M>1$, if $S :=  \bigcup_{r \in E} [100Lr, 100MLr]$, then $X$ satisfies $\EHI(r \in S)$.
\end{prop}

\begin{proof}
    Throughout this proof, we will use the notation ``$\lesssim$," ``$\asymp$," and ``$\gtrsim$." The implied constant may depend on $M$, $c$, and universal constants, but nothing else.
    By Proposition \ref{EHI:p:EHIpoisson=EHI}, it is enough to prove $\EHIPoisson(r \in S)$.
    We must show that there exists a $\kappa \in (0, 1)$ (which may depend on $X$) such that for all $r \in S$, $x \in B(0, \kappa r)$, and $w \in B(0, r)^c$, we have $K_{B(0, r)}(x, w) \asymp K_{B(0, r)}(-x, w)$. Let
    \begin{equation*}
        \kappa := \frac{1}{100L}.
    \end{equation*}
    Fix $r \in S$, $x \in B(0, \kappa r)$, and $w \in B(0, r)^c$. By the definition of $S$, $r$ satisfies $100L r_0 \leq r \leq 100M L r_0$ for some $r_0 \in E$. We will consider the small/flat decomposition at $r_0$.

    Choose $\eps>0$ small enough that $x, -x \in B(0, (1-\eps)r)$. Since the ratio $r/r_0$ is bounded above and below (with the bounds depending only on $M$), by applying Corollary \ref{EHI:c:K1(x)=K1(-x)} to $B(0, (1-\eps)r) \Subset B(0, r)$ and taking the limit as $\eps \to 0^+$, we already have
    \begin{equation*}
        K^{\inner{r_0}, >0}_{B(0, r)}(x, w) \asymp K^{\inner{r_0}, >0}_{B(0, r)}(-x, w)
    \end{equation*}
    so all that remains is to show that
    \begin{equation*}
        K^{\inner{r_0}, 0}_{B(0, r)}(x, w) \asymp K^{\inner{r_0}, 0}_{B(0, r)}(-x, w),
    \end{equation*}
    or equivalently,
    \begin{equation*}
        \int_{z \in B(0, r)} G^{\inner{r_0}, 0}_{B(0, r)}(x, z) j^{\inner{r_0}}(|w-z|) \dee{z} \asymp \int_{z \in B(0, r)} G^{\inner{r_0}, 0}_{B(0, r)}(-x, z) j^{\inner{r_0}}(|w-z|) \dee{z}.
    \end{equation*}
    Recall that $j^{\inner{r_0}}$ only admits jumps of size $r_0$ or smaller, so $j^{\inner{r_0}}(|w-z|) = 0$ for all $z$ at a distance greater than $r_0$ from the boundary of $B(0, r)$. Therefore, it is enough to show that
    \begin{equation}\label{EHI:e:EHIbigm2-1}
        G^{\inner{r_0}, 0}_{B(0, r)}(x, z) \asymp G^{\inner{r_0}, 0}_{B(0, r)}(-x, z) \qquad\mbox{for all $z \in B(0, r) \setminus B(0, r-r_0)$}.
    \end{equation}
    
    Let
    \begin{align*}
        r_1 &:= \kappa r = r/(100L), \\
        r_2 &:= r/(10L), \\
        \mbox{and} \quad r_3 &:= r/(2L).
    \end{align*}
    Recall that we already have $\frac{r}{100ML} \leq r_0 \leq \frac{r}{100L}$, so there is a hierarchy $r_0 \leq r_1 \leq r_2 \leq r_3 \leq r$. We also have $r_0 \asymp r_1 \asymp r_2 \asymp r_3 \asymp r$, with the ratio between any two of them being at most $100ML$. Observe also that the pairwise difference between any two of these radii (except possibly between $r_0$ and $r_1$) are on the order of $r$, and that $r_3 \leq r-r_0$. These facts will be used in the proof. Let $A$ denote the annulus $B(0, r_3) \setminus B(0, r_2)$.
    
    Recall that $x \in B(0, \kappa r)$. In terms of our newly introduced radii, this is equivalent to $x \in B(0, r_1)$. We will first show that $G^{\inner{r_0}, 0}_{B(0, r)}(x, y) \asymp G^{\inner{r_0}, 0}_{B(0, r)}(-x, y)$ for all $y \in A$, and then we will extend this result from $y \in A$ to all $z \in B(0, r) \setminus B(0, r_2)$.
    
    Since $B(0, r) \subseteq B(x, r+r_1)$, by Lemma \ref{EHI:l:G0upperbound}, for all $y \in A$ we have
    \begin{align}\label{EHI:e:EHIbigm2-2}
        G^{\inner{r_0}, 0}_{B(0, r)}(x, y) &\leq G^{\inner{r_0}, 0}_{B(x, r+r_1)}(x, y) = G^{\inner{r_0}, 0}_{B(0, r+r_1)}(0, y-x) \notag\\
        &\lesssim \frac{(r+r_1)^2 |y-x|^{-d}}{m_2\inner{r_0}} \leq \frac{(r+r_1)^2 (r_2-r_1)^{-d}}{m_2\inner{r_0}} \asymp \frac{r^{2-d}}{m_2\inner{r_0}}.
    \end{align}
    Since $B(0, r) \supseteq B(x, r-r_1)$, by Lemma \ref{EHI:l:G0lowerbound}, for all $y \in A$ we have
    \begin{align}\label{EHI:e:EHIbigm2-3}
        G^{\inner{r_0}, 0}_{B(0, r)}(x, y) &\geq G^{\inner{r_0}, 0}_{B(x, r-r_1)}(x, y) = G^{\inner{r_0}, 0}_{B(0, r-r_1)}(0, y-x) \notag\\
        &\gtrsim \frac{(r-r_1)^{2-d}}{m_2\inner{r_0}} \exp(-\frac{1}{10} \frac{(r-r_1)^2 \lambda\inner{r_0}}{m_2\inner{r_0}}) \notag\\
        &\asymp \frac{r^{2-d}}{m_2\inner{r_0}} \exp(-\frac{1}{10} \frac{(r-r_1)^2 \lambda\inner{r_0}}{m_2\inner{r_0}}).
    \end{align}
    By \eqref{EHI:e:bigm2Condition} and Lemma \ref{EHI:l:m2lambdaComparable}, $m_2\inner{r_0} \asymp m_2(r_0) \gtrsim r_0^2 \lambda(r_0) \asymp r^2 \lambda\inner{r_0}$. Therefore, the exponential term in \eqref{EHI:e:EHIbigm2-3} is on the order of $1$. Thus, by \eqref{EHI:e:EHIbigm2-2} and \eqref{EHI:e:EHIbigm2-3},
    \begin{equation}\label{EHI:e:EHIbigm2-4}
        G^{\inner{r_0}, 0}_{B(0, r)}(x, y) \asymp \frac{r^{2-d}}{m_2\inner{r_0}} \qquad\mbox{for all $y \in A := B(0, r_3) \setminus B(0, r_2)$}.
    \end{equation}

    It still remains to extend \eqref{EHI:e:EHIbigm2-4} from $y \in A$ to $z \in B(0, r) \setminus B(0, r_3)$. Fix $z \in B(0, r) \setminus B(0, r_3)$. Since $r_3-r_2 \geq r_0$, the only way the process $X$ can get near $z$ before the first flattening jump occurs is by first entering $A$ at some point. Thus, by the symmetry of $G^{\inner{r_0}, 0}_{B(0, r)}$ and the probabilistic interpretation of $G^{\inner{r_0}, 0}_{B(0, r)}$,
    \begin{align}\label{EHI:e:EHIbigm2-5}
        G^{\inner{r_0}, 0}_{B(0, r)}(x, z) &= G^{\inner{r_0}, 0}_{B(0, r)}(z, x) = \E_z \left[ \indicatorWithSetBrackets{T_A < \tau_{B(0, r)} \wedge T^{\inner{r_0}}} G^{\inner{r_0}, 0}_{B(0, r)}(X_{T_A}, x) \right] \notag\\
        &= \E_z \left[ \indicatorWithSetBrackets{T_A < \tau_{B(0, r)} \wedge T^{\inner{r_0}}} G^{\inner{r_0}, 0}_{B(0, r)}(x, X_{T_A}) \right].
    \end{align}
    By \eqref{EHI:e:EHIbigm2-4} and \eqref{EHI:e:EHIbigm2-5},
    \begin{equation*}
        G^{\inner{r_0}, 0}_{B(0, r)}(x, z) \asymp \P_z \left( T_A < \tau_{B(0, r)} \wedge T^{\inner{r_0}} \right) \frac{r^{2-d}}{m_2\inner{r_0}} \qquad\mbox{for all $z \in B(0, r) \setminus B(0, r_3)$}.
    \end{equation*}
    The same argument holds if we replace $x$ with $-x$. Therefore,
    \begin{equation*}
        G^{\inner{r_0}, 0}_{B(0, r)}(x, z) \asymp \P_z \left( T_A < \tau_{B(0, r)} \wedge T^{\inner{r_0}} \right) \frac{r^{2-d}}{m_2\inner{r_0}} \asymp G^{\inner{r_0}, 0}_{B(0, r)}(-x, z)  \qquad\mbox{for all $z \in B(0, r) \setminus B(0, r_3)$}.
    \end{equation*}
    Since $r_3 \leq r-r_0$, this completes the proof of \eqref{EHI:e:EHIbigm2-1}, which was all we needed to show.
\end{proof}

\begin{lemma}\label{EHI:l:G1lowerbound}
    Let $X$ be an isotropic unimodal L\'{e}vy jump process on $\R^d$. For all $r \geq r_0 >0$ and all $x, z \in B(0, r)$,
    \begin{equation*}
        G^{\inner{r_0}, 1}_{B(0, r)}(x, z) \geq \frac12 j(2r) \E_x \left[ \tau_{B(0, r)} \wedge T^{\inner{r_0}} \right] \cdot \E_z \left[ \tau_{B(0, r)} \wedge T^{\inner{r_0}} \right].
    \end{equation*}
\end{lemma}

\begin{proof}
    By the definition of $G^{\inner{r_0}, 1}_{B(0, r)}$ (that is, Definition \eqref{EHI:d:GkDef}),
    \begin{equation}\label{EHI:e:G1lowerbound1}
        G^{\inner{r_0}, 1}_{B(0, r)}(x, z) := \E_x \left[ \indicatorWithSetBrackets{T^{\inner{r_0}} < \tau_{B(0, r)}} G^{\inner{r_0}, 0}_{B(0, r)}(X_{T^{\inner{r_0}}}, z) \right].
    \end{equation}
    Let $\eta_x$ denote the measure
    \begin{equation*}
        \eta_x(E) := \P_x \left( T^{\inner{r_0}} < \tau_{B(0, r)}, X_{T^{\inner{r_0}}} \in E \right) \qquad\mbox{for all Lebesgue-measurable $E \subseteq B(0, r)$}.
    \end{equation*}
    Then \eqref{EHI:e:G1lowerbound1} can be written as
    \begin{equation}\label{EHI:e:G1lowerbound2}
        G^{\inner{r_0}, 1}_{B(0, r)}(x, z) = \int_{B(0, r)} G^{\inner{r_0}, 0}_{B(0, r)}(y, z) \, \eta_x(\mathrm{d}y).
    \end{equation}
    We will use the flatness of $\hat{j}^{\inner{r_0}}$ to show that the measure $\eta_x$ is bounded below by $\frac12 j(2r) \E_x \left[ \tau_{B(0, r)} \wedge T^{\inner{r_0}} \right]$ times the Lebesgue measure. Therefore, $\eta_x(\mathrm{d}y)$ in \eqref{EHI:e:G1lowerbound2} can be replaced with $\frac12 j(2r) \E_x \left[ \tau_{B(0, r)} \wedge T^{\inner{r_0}} \right] \, dy$.
    Then, in the resulting integral, we can switch $y$ and $z$ since $G^{\inner{r_0}, 0}_{B(0, r)}$, and the result will have a probabilistic interpretation, given by Lemma \ref{EHI:l:GkPurpose}.

    For all $y, y' \in B(0, r)$, we have $|y-y'|<2r$, so by the definition of the small/flat decomposition, $\hat{j}^{\inner{r_0}}(|y-y'|) \geq \min\{ \frac12 j(r_0), j(2r) \} \geq \frac12 j(2r)$. By the L\'{e}vy system formula (cf. \cite{BL}, \cite{ck1}),
    \begin{align*}
        \eta_x(E) &= \E_x \left[ \int_0^{\tau_{B(0, r)} \wedge T^{\inner{r_0}}} \int_{y \in E} \hat{j}^{\inner{r_0}}(|y-X_t|) \dee{y} \dee{t} \right] \\
        &\geq |E| \cdot \frac12 j(2r) \E_x \left[ \tau_{B(0, r)} \wedge T^{\inner{r_0}} \right].
    \end{align*}
    Since this holds for all Lebesgue-measurable $E \subseteq \R^d$, $\eta_x(\mathrm{d}y) \leq \frac12 j(2r) \E_x \left[ \tau_{B(0, r)} \wedge T^{\inner{r_0}} \right] \dee{y}$. Thus, \eqref{EHI:e:G1lowerbound2} implies
    \begin{equation} \label{EHI:e:G1lowerbound3}
        G^{\inner{r_0}, 1}_{B(0, r)}(x, z) \geq \frac12 j(2r) \E_x \left[ \tau_{B(0, r)} \wedge T^{\inner{r_0}} \right]  \int_{B(0, r)} G^{\inner{r_0}, 0}_{B(0, r)}(y, z) \dee{y}.
    \end{equation}
    Since $G^{\inner{r_0}, 0}$ is symmetric, and by Lemma \ref{EHI:l:GkPurpose},
    \begin{equation}\label{EHI:e:G1lowerbound4}
        \int_{B(0, r)} G^{\inner{r_0}, 0}_{B(0, r)}(y, z) \dee{y} = \int_{B(0, r)} G^{\inner{r_0}, 0}_{B(0, r)}(z, y)  \dee{y} = \E_z \left[ \tau_{B(0, r)} \wedge T^{\inner{r_0}} \right].
    \end{equation}
    By substituting \eqref{EHI:e:G1lowerbound4} into \eqref{EHI:e:G1lowerbound3}, the proof is complete.
\end{proof}

\begin{prop} \label{EHI:p:EHIsmallm2}
    Let $X$ be an isotropic unimodal L\'{e}vy jump process on $\R^d$, satisfying \eqref{EHI:e:regularJumps}. Suppose there exist $\eps \in (0, 1)$, $C>0$, and a set $E \subseteq (0, \infty)$ such that
    \begin{equation} \label{EHI:e:EHIsmallm2Condition}
        m_2(r) \leq C (r^{d+2} j(r))^\eps (r^2 \lambda(r))^{1-\eps} \qquad\mbox{for all $r \in E$}.
    \end{equation}
    Then there exists an $L_{\ref{EHI:p:EHIsmallm2}}=L_{\ref{EHI:p:EHIsmallm2}}(C, c_j, d, \eps)>1$ such that for all $L \geq L_{\ref{EHI:p:EHIsmallm2}}$ and all $M>1$, if $S := \bigcup_{r \in E} [Lr, MLr]$, then $X$ satisfies $\EHI(r \in S)$.
\end{prop}

The proof of Proposition \ref{EHI:p:EHIsmallm2} will follow a very similar pattern to that of Proposition \ref{EHI:p:EHIbigm2}, but with some key differences. In the proof of Proposition \ref{EHI:p:EHIbigm2} we showed that
\begin{equation*}
    K^{\inner{r_0}, >0}_{B(0, r)}(x, w) \asymp K^{\inner{r_0}, >0}_{B(0, r)}(-x, w) \qquad\mbox{and}\qquad K^{\inner{r_0}, 0}_{B(0, r)}(x, w) \asymp K^{\inner{r_0}, 0}_{B(0, r)}(-x, w)
\end{equation*}
and by taking the sum concluded that $K_{B(0, r)}(x, w) \asymp K_{B(0, r)}(-x, w)$. In the proof of Proposition \ref{EHI:p:EHIsmallm2}, we instead show that
\begin{equation*}
    K^{\inner{r_0}, >0}_{B(0, r)}(x, w) \asymp K^{\inner{r_0}, >0}_{B(0, r)}(-x, w) \qquad\mbox{and}\qquad \left\{\begin{matrix}
        K^{\inner{r_0}, 0}_{B(0, r)}(x, w) \lesssim K^{\inner{r_0}, >0}_{B(0, r)}(x, w), \\
        \\
        K^{\inner{r_0}, 0}_{B(0, r)}(-x, w) \lesssim K^{\inner{r_0}, >0}_{B(0, r)}(-x, w).
    \end{matrix}\right.
\end{equation*}
In other words, rather than show that both $K^{\inner{r_0}, >0}_{B(0, r)}(\cdot, w)$ and $K^{\inner{r_0}, 0}_{B(0, r)}(\cdot, w)$ are comparable between $x$ and $-x$, we show that $K^{\inner{r_0}, >0}_{B(0, r)}(\cdot, w)$ is the leading term of both these decompositions and it is therefore enough to simply check that $K^{\inner{r_0}, >0}_{B(0, r)}(x, w) \asymp K^{\inner{r_0}, >0}_{B(0, r)}(-x, w)$.

Note also that in Proposition \ref{EHI:p:EHIsmallm2}, $L_{\ref{EHI:p:EHIsmallm2}}$ depends on $C$, $c_j$, $d$, and $\eps$, whereas in Proposition \ref{EHI:p:EHIbigm2}, $L_{\ref{EHI:l:G0lowerbound}}$ is just a universal constant. This will require us to be just a bit more careful at times in tracking the implicit constants in inequalities involving ``$\lesssim$," ``$\asymp$," and ``$\gtrsim$."

\begin{proof}[Proof of Proposition \ref{EHI:p:EHIsmallm2}]
    Throughout this proof, we will use the notation ``$\lesssim$," ``$\asymp$," and ``$\gtrsim$." The implied constant may depend on $M$, $\eps$, $C$, and universal constants, but nothing else.

    Let $k:=\ceil{1+\eps^{-1}}$. This way, $k \in \mathbb{N}$ and $\frac{1}{k-1} \leq \eps$. For all $r \in E$, by Lemma \ref{EHI:l:m2lambdaComparable}, \eqref{EHI:e:EHIsmallm2Condition}, and Lemma \ref{EHI:l:SmallestOfThreeQuantities},
    \begin{align*}
        m_2\inner{r} &\asymp m_2(r) \lesssim (r^{d+2} j(r))^\eps (r^2 \lambda(r))^{1-\eps} \notag\\
        &\lesssim (r^{d+2} j(r))^{\frac{1}{k-1}} (r^2 \lambda(r))^{\frac{k-2}{k-1}} \notag \\
        &\asymp (r^{d+2} j(r))^{\frac{1}{k-1}} (r^2 \lambda\inner{r})^{\frac{k-2}{k-1}} \label{EHI:e:EHIsmallm2-0} \\
        &\lesssim r^2 \lambda\inner{r}. \notag
    \end{align*}
    It will help to make reference to these implied constants, so let $C_1$ and $C_2$ be constants such that
    \begin{equation} \label{EHI:e:EHIsmallm2-1}
        m_2\inner{r} \leq C_1 \left( \frac{m_2\inner{r_0}}{r_0^{d+2} j(r_0)} \right)^{\frac{1}{k-1}} \left( \frac{m_2\inner{r_0}}{r_0^2 \lambda\inner{r_0}}\right)^{\frac{k-2}{k-1}} \leq C_2 r^2 \lambda\inner{r} \qquad\mbox{for all $r \in E$}.
    \end{equation}
    Let $L_{\ref{EHI:p:EHIsmallm2}}:=\max\{\sqrt{4C_2}, 4k, 100\}$, suppose $L \geq L_{\ref{EHI:p:EHIsmallm2}}$, and let $S := \bigcup_{r \in E} [100Lr, 100MLr]$. We would like to prove $\EHI(r \in S)$.
    By Proposition \ref{EHI:p:EHIpoisson=EHI}, it is enough to prove $\EHIPoisson(r \in S)$.
    We must show that there exists a $\kappa \in (0, 1)$ (which may depend on $X$) such that for all $r \in S$, $x \in B(0, \kappa r)$, and $w \in B(0, r)^c$, we have $K_{B(0, r)}(x, w) \asymp K_{B(0, r)}(-x, w)$. Let
    \begin{equation*}
        \kappa := \frac{1}{32}.
    \end{equation*}
    Fix $r \in S$, $x \in B(0, \kappa r)$, and $w \in B(0, r)^c$. By the definition of $S$, $r$ satisfies $100L r_0 \leq r \leq 100M L r_0$ for some $r_0 \in E$. We will consider the small/flat decomposition at $r_0$.

    Choose $\delta>0$ small enough that $x, -x \in B(0, (1-\delta)r)$. Since the ratio $r/r_0$ is bounded above and below (with the bounds depending only on $M$ and $C$), by applying Corollary \ref{EHI:c:K1(x)=K1(-x)} to $B(0, (1-\delta)r) \Subset B(0, r)$ and taking the limit as $\delta \to 0^+$, we already have
    \begin{equation*}
        K^{\inner{r_0}, >0}_{B(0, r)}(x, w) \asymp K^{\inner{r_0}, >0}_{B(0, r)}(-x, w).
    \end{equation*}
    We claim that it is now sufficient to simply prove that
    \begin{equation}\label{EHI:e:EHIsmallm2-2}
        K^{\inner{r_0}, 0}_{B(0, r)}(x, w) \lesssim K^{\inner{r_0}, >0}_{B(0, r)}(x, w) \qquad\mbox{and}\qquad K^{\inner{r_0}, 0}_{B(0, r)}(-x, w) \lesssim K^{\inner{r_0}, >0}_{B(0, r)}(-x, w).
    \end{equation}
    Indeed, if we had \eqref{EHI:e:EHIsmallm2-2}, then we would have
    \begin{align*}
        K_{B(0, r)}(x, w) &= K^{\inner{r_0}, 0}_{B(0, r)}(x, w) + K^{\inner{r_0}, >0}_{B(0, r)}(x, w) \\
        &\asymp K^{\inner{r_0}, >0}_{B(0, r)}(x, w) \\
        &\asymp K^{\inner{r_0}, >0}_{B(0, r)}(-x, w) \\
        &\asymp K^{\inner{r_0}, 0}_{B(0, r)}(-x, w) + K^{\inner{r_0}, >0}_{B(0, r)}(-x, w)\\
        &= K_{B(0, r)}(-x, w).
    \end{align*}
    Therefore, all that remains is to prove \eqref{EHI:e:EHIsmallm2-2}.
    
    By definition,
    \begin{equation*}
        K^{\inner{r_0}, 0}_{B(0, r)}(x, w) := \int_{z \in B(0, r)} G^{\inner{r_0}, 0}_{B(0, r)}(x, z) j^{\inner{r_0}}(|w-z|) \dee{z}
    \end{equation*}
    and by \eqref{EHI:e:(K>0)>K1},
    \begin{equation*}
        K^{\inner{r_0}, >0}_{B(0, r)}(x, w) \geq \int_{z \in B(0, r)} G^{\inner{r_0}, 1}_{B(0, r)}(x, z) \hat{j}^{\inner{r_0}}(|w-z|) \dee{z}.
    \end{equation*}
    Therefore, it is enough to show that
    \begin{align}\label{EHI:e:EHIsmallm2NewGoal1}
    \begin{split}
        G^{\inner{r_0}, 0}_{B(0, r)}(x, z) j^{\inner{r_0}}(|w-z|) &\lesssim G^{\inner{r_0}, 1}_{B(0, r)}(x, z) j^{\inner{r_0}}(|w-z|), \\
        G^{\inner{r_0}, 0}_{B(0, r)}(-x, z) j^{\inner{r_0}}(|w-z|) &\lesssim G^{\inner{r_0}, 1}_{B(0, r)}(-x, z) j^{\inner{r_0}}(|w-z|)
    \end{split}
    \end{align}
    for all $z \in B(0, r)$. If $z \in B(0, r-r_0)$, then $|z-w| \geq r_0$, so $j^{\inner{r_0}}(|w-z|)=0$, and subsequently sides of \eqref{EHI:e:EHIsmallm2NewGoal1} are $0$. Thus, it is enough to show that
    \begin{equation} \label{EHI:e:EHIsmallm2NewGoal2}
        \left.\begin{matrix}
            G^{\inner{r_0}, 0}_{B(0, r)}(x, z) \lesssim G^{\inner{r_0}, 1}_{B(0, r)}(x, z) \\
            \\
            G^{\inner{r_0}, 0}_{B(0, r)}(-x, z) \lesssim G^{\inner{r_0}, 1}_{B(0, r)}(-x, z)
        \end{matrix} \right\}
        \qquad\mbox{for all $z \in B(0, r) \setminus B(0, r-r_0)$}.
    \end{equation}
    
    Fix $z \in B(0, r) \setminus B(0, r-r_0)$.
    Let
    \begin{align*}
        r_1 &:= \kappa r = r/32, \\
        r_2 &:= r/16, \\
        r_3 &= r/8, \\
        r_4 &= r/4, \\
        \mbox{and} \quad r_5 &:= r/2.
    \end{align*}
    Let $A_1$ and $A_2$ denote the annuli
    \begin{equation*}
        A_1 := B(0, r_3) \setminus B(0, r_2), \qquad\mbox{and}\qquad A_2 := B(0, r_5) \setminus B(0, r_4).
    \end{equation*}
    Then $0<r_0 < r_1 < r_2 < r_3 < r_4 < r_5 < r-r_0 < r$ and $r_0 \asymp r_1 \asymp r_2 \asymp r_3 \asymp r_4 \asymp r_5 \asymp r$, with the ratio between any two of them bounded above by $100ML$. The pairwise difference between any two of these radii is also on the order of $r$.

    Recall that $x \in B(0, \kappa r) = B(0, r_1)$ and $z \in B(0, r) \setminus B(0, r-r_0) \subseteq B(0, r) \setminus B(0, r_5)$. We would like to show that $G^{\inner{r_0}, 0}_{B(0, r)}(x, z) \lesssim G^{\inner{r_0}, 1}_{B(0, r)}(x, z)$. Let us put a lower bound on $ G^{\inner{r_0}, 1}_{B(0, r)}(x, z)$ and an upper bound on $ G^{\inner{r_0}, 0}_{B(0, r)}(x, z)$, and then compare these two bounds.

    By Lemma \ref{EHI:l:G1lowerbound},
    \begin{equation}\label{EHI:e:EHIsmallm2-3}
        G^{\inner{r_0}, 1}_{B(0, r)}(x, z) \geq \frac12 j(2r) \E_x \left[ \tau_{B(0, r)} \wedge T^{\inner{r_0}} \right] \cdot \E_z \left[ \tau_{B(0, r)} \wedge T^{\inner{r_0}} \right].
    \end{equation}
    We need lower bounds for both of the expectations that appear in \eqref{EHI:e:EHIsmallm2-3}.
    Since $B(0, r) \supseteq B(x, r-r_1)$, by Lemma \ref{EHI:l:PE},
    \begin{equation}\label{EHI:e:EHIsmallm2-4}
        \E_x \left[ \tau_{B(0, r)} \wedge T^{\inner{r_0}} \right] \geq \E_x \left[ \tau_{B(x, r-r_1)} \wedge T^{\inner{r_0}} \right] \geq e^{-1} \left(1-\frac{2m_2\inner{r_0}}{(r-r_1)^2 \lambda\inner{r_0}} \right) \frac{1}{\lambda\inner{r_0}}.
    \end{equation}
    Unfortunately, there is no radius $s$ such that $B(z, s)$ is guaranteed to be contained in $B(0, r)$, since $z$ could be very close to the boundary of $B(0, r)$. Thus, we can not directly use the same trick to get a lower bound on $\E_z \left[ \tau_{B(0, r)} \wedge T^{\inner{r_0}} \right]$. Instead, let us condition on reaching $A_2$ from $z$ before $T{\inner{r_0}}$. For all $y \in A_2$, the distance from $y$ to the boundary of $B(0, r)$ is at least $r-r_5$. Therefore, by the Strong Markov property, the memoryless property of $T^{\inner{r_0}}$, and Lemma \ref{EHI:l:PE},
    \begin{align}\label{EHI:e:EHIsmallm2-5}
        \E_z \left[ \tau_{B(0, r)} \wedge T^{\inner{r_0}} \right] &\geq \P_z(T_{A_2} < T^{\inner{r_0}}) \cdot \left(\inf_{y \in A_2} \E_y \left[ \tau_{B(0, r)} \wedge T^{\inner{r_0}} \right] \right) \notag\\
        &\geq \P_z(T_{A_2} < T^{\inner{r_0}}) \cdot \left(\inf_{y \in A_2} \E_y \left[ \tau_{B(y, r/6)} \wedge T^{\inner{r_0}} \right] \right) \notag\\
        &\geq \P_z(T_{A_2} < T^{\inner{r_0}}) \cdot e^{-1} \left( 1- \frac{2m_2\inner{r_0}}{(r-r_5)^2 \lambda\inner{r_0}} \right) \frac{1}{\lambda\inner{r_0}}.
    \end{align}
    By substituting \eqref{EHI:e:EHIsmallm2-4} and \eqref{EHI:e:EHIsmallm2-5} into \eqref{EHI:e:EHIsmallm2-3},
    \begin{equation}\label{EHI:e:EHIsmallm2-6}
        G^{\inner{r_0}, 1}_{B(0, r)}(x, z) \geq \frac12 e^{-2} \left(1-\frac{2m_2\inner{r_0}}{(r-r_1)^2 \lambda\inner{r_0}} \right) \left( 1- \frac{2m_2\inner{r_0}}{(r-r_5)^2 \lambda\inner{r_0}} \right) \P_z(T_{A_2} < T^{\inner{r_0}}) \frac{j(2r)}{(\lambda\inner{r_0})^2}.
    \end{equation}
    We claim that both of the terms in parentheses in \eqref{EHI:e:EHIsmallm2-6} are at least $1/2$.
    Indeed, since $r-r_1 \geq r-r_5 \geq r/2$, $r \geq Lr_0 \geq 4\sqrt{C_2} r_0$, and $r_0 \in E$, by \eqref{EHI:e:EHIsmallm2-1},
    \begin{align*}
        \left( 1-\frac{2m_2\inner{r_0}}{(r-r_1)^2 \lambda\inner{r_0}} \right) , \left(1-\frac{2m_2\inner{r_0}}{(r-r_5)^2 \lambda\inner{r_0}} \right) &\geq 1 - \frac{8m_2\inner{r_0}}{r^2 \lambda\inner{r_0}} \\
        &\geq 1-\frac{m_2\inner{r_0}}{2C_2 r_0^2 \lambda\inner{r_0}} \\
        &\geq 1-\frac{C_2}{2C_2}=\frac12.
    \end{align*}
    Therefore,
    \begin{equation}\label{EHI:e:EHIsmallm2-7}
        G^{\inner{r_0}, 1}_{B(0, r)}(x, z) \geq \frac18 e^{-2} \P_z(T_{A_2} < T^{\inner{r_0}}) \frac{j(2r)}{(\lambda\inner{r_0})^2}.
    \end{equation}

    Now let us establish an upper bound on $G^{\inner{r_0}, 0}_{B(0, r)}(x, z)$. Since $B(0, r) \subseteq B(x, r+r_1)$, by Lemma \ref{EHI:l:G0upperbound}, for all $y \in A_1$ we have
    \begin{align}\label{EHI:e:EHIsmallm2-8}
        G^{\inner{r_0}, 0}_{B(0, r)}(x, y) &\leq G^{\inner{r_0}, 0}_{B(x, r+r_1)}(x, y) = G^{\inner{r_0}, 0}_{B(0, r+r_1)}(0, y-x) \notag\\
        &\lesssim \frac{(r+r_1)^2 |y-x|^{-d}}{m_2\inner{r_0}} \leq \frac{(r+r_1)^2 (r_2-r_1)^{-d}}{m_2\inner{r_0}} \asymp \frac{r^{2-d}}{m_2\inner{r_0}}.
    \end{align}
    In order for the process to travel from $x$ to $z$ before time $T^{\inner{r_0}}$, it must pass through $A_1$ first, because $A_1$ has a thickness greater than $r_0$ and the process only takes jumps of magnitude $r_0$ or smaller before time $T^{\inner{r_0}}$. Thus, by the probabilistic interpretation of $G^{\inner{r_0}, 0}_{B(0, r)}$, the symmetry of $G^{\inner{r_0}, 0}_{B(0, r)}$, and \eqref{EHI:e:EHIsmallm2-8},
    \begin{align}\label{EHI:e:EHIsmallm2-9}
        G^{\inner{r_0}, 0}_{B(0, r)}(x, z) &= G^{\inner{r_0}, 0}_{B(0, r)}(x, z) = \E_z \left[ \indicatorWithSetBrackets{T_{A_1} < T^{\inner{r_0}}} G^{\inner{r_0}, 0}_{B(0, r)} \left(X_{T_{A_1}}, x\right) \right] \notag\\
        &= \E_z \left[ \indicatorWithSetBrackets{T_{A_1} < T^{\inner{r_0}}} G^{\inner{r_0}, 0}_{B(0, r)} \left(x, X_{T_{A_1}} \right) \right] \notag\\
        &\lesssim \P_z \left( T_{A_1} < T^{\inner{r_0}} \right) \frac{r^{2-d}}{m_2\inner{r_0}}.
    \end{align}
    In order for the process to travel from $z$ to $A_1$ before time $T^{\inner{r_0}}$, two things must happen: first it must reach $A_2$ (since $A_2$ also has a thickness greater than $r_0$), and then it must travel from $A_2$ to $A_1$. In order to travel from $A_2$ to $A_1$, it must cover a distance of at least $r_4-r_3$ before $T^{\inner{r_0}}$. Since $L \geq 4k$, we have $r_4-r_3 = r/8 \geq Lr_0/8 \geq 2k r_0 \geq (2k-1)r_0$. Therefore, by the Strong Markov property and Lemma \ref{EHI:l:CloseToSubmultiplicative} (applied to $X^{\inner{r_0}}$ and $T^{\inner{r_0}}$),
    \begin{align}\label{EHI:e:EHIsmallm2-10}
        \P_z \left( T_{A_1} < T^{\inner{r_0}} \right) &\leq \P_z \left( T_{A_2} < T^{\inner{r_0}} \right) \cdot \left( \sup_{y \in A_2} \P_y \left( T_{A_1} < T^{\inner{r_0}} \right) \right) \notag\\
        &\leq \P_z \left( T_{A_2} < T^{\inner{r_0}} \right) \cdot \left( \sup_{y \in A_2} \P_y \left( \tau_{B(y, 2k r_0)} < T^{\inner{r_0}} \right) \right) \notag\\
        &= \P_z \left( T_{A_2} < T^{\inner{r_0}} \right) \cdot \P_0 \left( \tau_{B(0, (2k-1) r_0)} < T^{\inner{r_0}} \right) \notag\\
        &\leq \P_z \left( T_{A_2} < T^{\inner{r_0}} \right) \left( \P_0 \left( \tau_{B(0, r_0)} < T^{\inner{r_0}} \right)\right)^k.
    \end{align}
    Recall that $T^{\inner{r_0}}$ has Exponential($\lambda\inner{r_0}$) distribution, and therefore its mean is $1/\lambda\inner{r_0}$.
    By Lemma \ref{EHI:l:FactorOf2ForExits}, Markov's inequality, and Lemma \ref{EHI:l:m2isvariance} (applied to $X^{\inner{r_0}}$),
    \begin{align}\label{EHI:e:EHIsmallm2-11}
        \P_0 \left( \tau_{B(0, r_0)} < T^{\inner{r_0}} \right) &\leq 2\P_0 \left( |X_{T^{\inner{r_0}}}| \geq r_0 \right) \leq 2r_0^{-2} \E_0 \left[ \left| X_{T^{\inner{r_0}}} \right|^2 \right] = \frac{2m_2\inner{r_0}}{r_0^2 \lambda\inner{r_0}}.
    \end{align}
    By plugging \eqref{EHI:e:EHIsmallm2-10} and then \eqref{EHI:e:EHIsmallm2-11} into \eqref{EHI:e:EHIsmallm2-9},
    \begin{equation}\label{EHI:e:EHIsmallm2-12}
        G^{\inner{r_0}, 0}_{B(0, r)}(x, z) \lesssim \P_z \left( T_{A_2} < T^{\inner{r_0}} \right) \frac{r^{2-d}}{m_2\inner{r_0}} \left( \frac{m_2\inner{r_0}}{r_0^2 \lambda\inner{r_0}} \right)^k.
    \end{equation}

    We now have a lower bound on $G^{\inner{r_0}, 1}_{B(0, r)}(x, z)$ in \eqref{EHI:e:EHIsmallm2-7} and an upper bound on $G^{\inner{r_0}, 0}_{B(0, r)}(x, z)$ in \eqref{EHI:e:EHIsmallm2-12}. By comparing these two bounds,
    \begin{align*}
        G^{\inner{r_0}, 0}_{B(0, r)}(x, z) &\lesssim \frac{r^{2-d}}{m_2\inner{r_0}} \left( \frac{m_2\inner{r_0}}{r_0^2 \lambda\inner{r_0}} \right)^k \frac{(\lambda\inner{r_0})^2}{j(2r)} G^{\inner{r_0}, 1}_{B(0, r)}(x, z) \\
        &= \left(\frac{r}{r_0}\right)^4 \frac{m_2\inner{r_0}}{r^{d+2} j(2r)} \left( \frac{m_2\inner{r_0}}{r_0^2 \lambda\inner{r_0}} \right)^{k-2} G^{\inner{r_0}, 1}_{B(0, r)}(x, z).
    \end{align*}
    Recall that $r_0 \asymp r$. By \eqref{EHI:e:regularJumps}, the fact that $r_0 \asymp r$ implies $j(2r) \asymp j(r_0)$. Thus,
    \begin{equation*}
        G^{\inner{r_0}, 0}_{B(0, r)}(x, z) \lesssim \frac{m_2\inner{r_0}}{r_0^{d+2} j(r_0)} \left( \frac{m_2\inner{r_0}}{r_0^2 \lambda\inner{r_0}} \right)^{k-2} G^{\inner{r_0}, 1}_{B(0, r)}(x, z).
    \end{equation*}
    By \eqref{EHI:e:EHIsmallm2-1},
    \begin{equation*}
        \frac{m_2\inner{r_0}}{r_0^{d+2} j(r_0)} \left( \frac{m_2\inner{r_0}}{r_0^2 \lambda\inner{r_0}} \right)^{k-2} \leq C_1^{k-1} \asymp 1.
    \end{equation*}
    Thus,
    \begin{equation*}
        G^{\inner{r_0}, 0}_{B(0, r)}(x, z) \lesssim G^{\inner{r_0}, 1}_{B(0, r)}(x, z).
    \end{equation*}
    The same holds for $-x$ by the same argument. This completes the proof of \eqref{EHI:e:EHIsmallm2NewGoal2}, which was all we needed to prove.
\end{proof}

\begin{proof}[Proof of Theorem \ref{EHI:t:main}]
   Suppose $M, C, c>0$ and $\eps \in (0, 1]$. Let $X$ be an isotropic unimodal L\'{e}vy jump process satisfying \eqref{EHI:e:regularJumps}. Let $(r_n)$ be a sequence of positive numbers such that $M^{-1} \leq \frac{r_{n+1}}{r_n} \leq M$ for all $n$, and each $n$ satisfies at least one of \eqref{EHI:e:bigm2} and \eqref{EHI:e:smallm2}. Let $E_1$ be the set of all terms $r_n$ in the sequence satisfying \eqref{EHI:e:bigm2}, and let $E_2$ be the set of all $r_n$ satisfying \eqref{EHI:e:smallm2}. Let
   \begin{equation*}
       L:=\max\{L_{\ref{EHI:l:G0lowerbound}}, L_{\ref{EHI:p:EHIsmallm2}}(C, c_j, d, \eps) \}.
   \end{equation*}
   Let
   \begin{equation*}
       S_1 := \bigcup_{r \in E_1} [100Lr, 100MLr] \qquad\mbox{and}\qquad S_2 := \bigcup_{r \in E} [100Lr, 100MLr].
   \end{equation*}
   By Proposition \ref{EHI:p:EHIbigm2}, $X$ satifies $\EHI(r \in S_1)$. By Proposition \ref{EHI:p:EHIsmallm2}, $X$ satisfies $\EHI(r \in S_2)$. Thus, $X$ satisfies $\EHI(r \in S_1 \cup S_2)$.

   If $r_n \to 0^+$, then $S_1 \cup S_2 \supseteq (0, 100MLr_1]$, so $X$ satisfies $\EHI(r \leq 100MLr_1)$. By Proposition \ref{EHI:p:EHI[a,b]forfree}, this means $X$ satisfies $\EHI(r \leq 1)$.
   Similarly, if $r_n \to \infty$, then $S_1 \cup S_2 \supseteq [Lr_1, \infty)$, so $X$ satisfies $\EHI(r \geq L r_1)$ and therefore (by Proposition \ref{EHI:p:EHI[a,b]forfree}) also satisfies $\EHI(r \geq 1)$.
\end{proof}

\subsection{Proof of Corollary \ref{EHI:c:mainNoM2}}

In this subsection, we prove Corollary \ref{EHI:c:mainNoM2}. Parts (a) and (b) are each a simple calculation.

\begin{proof}[Proof of Corollary \ref{EHI:c:mainNoM2}(a)]
    Let $X$ be an isotropic unimodal L\'{e}vy jump process with jump kernel $j(r)$, satisfying \eqref{EHI:e:regularJumps}. Suppose there exist $c, R,\alpha>0$ such that
        \begin{equation}\label{EHI:e:mainNoM2a1}
            \frac{r^{d+2} j(r)}{s^{d+2} j(s)} \geq c \left(\frac{r}{s}\right)^\alpha \qquad\mbox{for all $0 < s \leq r \leq R$},
        \end{equation}
        For all $r \in (0, R]$, by converting the formula for $m_2(r)$ into polar coordinates and then using \eqref{EHI:e:mainNoM2a1} to get an upper bound on $j(s)$ for all $s<r$, 
        \begin{align*}
            m_2(r) &= \int_{B(0, r)} |x|^2 j(|x|) \dee{x} \asymp \int_0^r s^{d+1} j(s) \dee{s} \lesssim \int_0^r s^{-1+\alpha} r^{d+2-\alpha} j(r) \dee{s} \asymp r^{d+2}j(r) \qquad\mbox{as $r \to 0^+$}.
        \end{align*}
        Thus, by Corollary \ref{EHI:c:main}(c) with $\eps=1$, $X$ satisfies $\EHI(r \leq 1)$.
\end{proof}

\begin{proof}[Proof of Corollary \ref{EHI:c:mainNoM2}(b)]
    Let $X$ be an isotropic unimodal L\'{e}vy jump process with jump kernel $j(r)$, satisfying \eqref{EHI:e:regularJumps}. Suppose $j(r) \gtrsim r^{-d-2}$ as $r \to \infty$, and there exist $c, R,\alpha>0$ such that
    \begin{equation}\label{EHI:e:mainNoM2b1}
        \frac{r^{d+2} j(r)}{s^{d+2} j(s)} \geq c \left(\frac{r}{s}\right)^\alpha \qquad\mbox{for all $R \leq s \leq r$}.
    \end{equation}
    Then, for all $r \geq R$,
    \begin{equation}\label{EHI:e:mainNoM2b2}
        m_2(r) = \int_{B(0, R)} |x|^2 j(|x|) \dee{x} + \int_{B(0, r) \setminus B(0, R)} |x|^2 j(|x|) \dee{x}
    \end{equation}
    We handle each of the two integrals in the right-hand side of \eqref{EHI:e:mainNoM2b2} separately. The first of these integrals, $\int_{B(0, R)} |x|^2 j(|x|) \dee{x}$, remains constant as $r \to \infty$. For the second integral, by converting to polar coordinates and using \eqref{EHI:e:mainNoM2b1} to get an upper bound on $j(s)$ for $s<r$,
    \begin{equation*}
        \int_{B(0, r) \setminus B(0, R)} |x|^2 j(|x|) \dee{x} \asymp \int_R^r s^{d+1} j(s) \dee{s} \lesssim \int_R^r s^{-1+\alpha} r^{d+2-\alpha} j(r) \dee{s} \asymp r^{d+2}j(r) \qquad\mbox{as $r \to \infty$}.
    \end{equation*}
    Since $j(r) \gtrsim r^{-d-2}$ as $r \to \infty$, the second integral dominates the first one. Thus, $m_2(r) \lesssim r^{d+2} j(r)$ as $r \to \infty$. By Corollary \ref{EHI:c:main}(d) with $\eps=1$, $X$ satisfies $\EHI(r \leq 1)$.
\end{proof}

To handle the logarithmic term in parts (c) and (d) of Corollary \ref{EHI:c:mainNoM2}, we use the following lemmas.

\begin{lemma} \label{EHI:l:SpecialIntegralAsymptotic}
    For all $\alpha, \beta>0$ and $\delta \in (0, 1)$, we have
    \begin{equation*}
        \int_r^\delta s^{-1-\alpha} (\log(s^{-1}))^{-1-\beta} \dee{s} \asymp r^{-\alpha} \left(\log(r^{-1}) \right)^{-1-\beta} \qquad\mbox{as $r \to 0^+$}.
    \end{equation*}
    (The constants implicit in $\asymp$ depend on $\alpha$, $\beta$, and $\delta$.)
\end{lemma}

\begin{proof}
    Fix $\alpha$, $\beta$, and $\delta$, and let $f(r) := \int_r^\delta s^{-1-\alpha} (\log(s^{-1}))^{-1-\beta} \dee{s}$ for all $0 < r \leq \delta$. The substitution $u=\log(s^{-1})$ yields
    \begin{equation*}
        f(r) = \int_{\log(\delta^{-1})}^{\log(r^{-1})} e^{\alpha u} u^{-1-\beta} \dee{u}.
    \end{equation*}
    For $r \leq \delta^2$, we can split this integral in two as follows:
    \begin{equation}\label{EHI:e:SpecialIntegralAsymptotic1}
        f(r) = \int_{\log(\delta^{-1})}^{\frac12\log(r^{-1})} e^{\alpha u} u^{-1-\beta} \dee{u} + \int_{\frac12\log(r^{-1})}^{\log(r^{-1})} e^{\alpha u} u^{-1-\beta} \dee{u}.
    \end{equation}
    For the first integral in \eqref{EHI:e:SpecialIntegralAsymptotic1}, the $e^{\alpha u}$ term is always at most $r^{-\alpha/2}$, so
    \begin{align*}
        \int_{\log(\delta^{-1})}^{\frac12\log(r^{-1})} e^{\alpha u} u^{-1-\beta} \dee{u} &\leq r^{-\alpha/2} \int_{\log(\delta^{-1})}^{\frac12\log(r^{-1})} u^{-1-\beta} \dee{u} \notag\\
        &= \frac{r^{-\alpha/2}}{\beta} \left( \left(\log(\delta^{-1}) \right)^{-\beta} - \left(\frac12 \log(r^{-1}) \right)^{-\beta} \right) \notag\\
        &\asymp r^{-\alpha/2} \qquad\mbox{as $r \to 0^+$}.
    \end{align*}
    For the second integral in \eqref{EHI:e:SpecialIntegralAsymptotic1}, $u^{-1-\beta}$ is always between $\left(\log(r^{-1}) \right)^{-1-\beta}$ and $2^{1+\beta}\left(\log(r^{-1}) \right)^{-1-\beta}$, so
        \begin{align*}
            \int_{\frac12\log(r^{-1})}^{\log(r^{-1})} e^{\alpha u} u^{-1-\beta} \dee{u} &\asymp \left(\log(r^{-1}) \right)^{-1-\beta} \int_{\frac12\log(r^{-1})}^{\log(r^{-1})} e^{\alpha u} \dee{u} \notag\\
            &= \frac{\left(\log(r^{-1}) \right)^{-1-\beta}}{\alpha} \left( r^{-\alpha}-r^{-\alpha/2} \right) \notag\\
            &\asymp r^{-\alpha} \left(\log(r^{-1}) \right)^{-1-\beta} \qquad\mbox{as $r \to 0^+$}.
        \end{align*}
        The second integral dominates the first one as $r \to 0^+$. Thus,
        \begin{equation*}
            f(r) \asymp r^{-\alpha} \left(\log(r^{-1}) \right)^{-1-\beta} \qquad\mbox{as $r \to 0^+$}.
        \end{equation*}
\end{proof}

\begin{lemma} \label{EHI:l:SpecialIntegralAsymptotic'}
    For all $\alpha \geq 1$ and $\beta>0$, we have
    \begin{equation*}
        \int_r^\infty s^{-2-\alpha} (\log s)^{1+\beta} \dee{s} \asymp r^{-1-\alpha} (\log r)^{1+\beta} \qquad\mbox{as $r \to \infty$}.
    \end{equation*}
    (The constants implicit in $\asymp$ depend on $\alpha$ and $\beta$.)
\end{lemma}

\begin{proof}
    The proof is very similar to that of Lemma \ref{EHI:l:SpecialIntegralAsymptotic}. Fix $\alpha$ and $\beta$, and let $f(r) := \int_r^\infty s^{-2-\alpha} (\log s)^{1+\beta} \dee{s}$. The substitution $u=\log s$ yields
    \begin{align}\label{EHI:e:SpecialIntegralAsymptotic'1}
        f(r) = \int_{\log r}^\infty e^{-(1+\alpha) u} u^{1+\beta} \dee{u} = \int_{\log r}^{2\log r} e^{-(1+\alpha) u} u^{1+\beta} \dee{u} + \int_{2\log r}^\infty e^{-(1+\alpha) u} u^{1+\beta} \dee{u}.
    \end{align}
    We handle each of the two integrals on the right-hand side of \eqref{EHI:e:SpecialIntegralAsymptotic'1} separately. In the first integral, the $u^{1+\beta}$ term is always between $(\log r)^{1+\beta}$ and $2^{1+\beta} (\log r)^{1+\beta}$, so
    \begin{align*}
        \int_{\log r}^{2\log r} e^{-(1+\alpha) u} u^{1+\beta} \dee{u} &\asymp (\log r)^{1+\beta} \int_{\log r}^{2 \log r} e^{-(1+\alpha) u} \dee{u} \\
        &= \frac{(\log r)^{1+\beta}}{1+\alpha} \left( r^{-1-\alpha} - r^{-2-2\alpha} \right) \\
        &\asymp r^{-1-\alpha} (\log r)^{1+\beta} \qquad\mbox{as $r \to \infty$}.
    \end{align*}
    For the second integral, for large enough values of $r$ we have $u^{1+\beta} < e^u$, so
    \begin{align*}
        \int_{2\log r}^\infty e^{-(1+\alpha) u} u^{1+\beta} \dee{u} &\leq \int_{2 \log r}^\infty e^{-\alpha u} \dee{u} = \frac{1}{\alpha} r^{-2\alpha}.
    \end{align*}
    Since $\alpha \geq 1$, the first integral dominates the second one. Thus,
    \begin{equation*}
        f(r) \asymp r^{-1-\alpha} (\log r)^{1+\beta} \qquad\mbox{as $r \to \infty$}.
    \end{equation*}
\end{proof}

In the proof of parts (c) and (d) of Corollary \ref{EHI:c:mainNoM2}, it will be convenient to switch the roles of $r$ and $s$. (In the statement, it was convenient to have $s<r$, to draw an analogy with the complementary conditions from parts (a) and (b), but in the proof it is more convenient to have $r<s$ so that we can draw conclusions about $r^{d+2}$, $m_2(r)$, and $r^2 \lambda(r)$ rather than $s^{d+2}$, $m_2(s)$, and $s^2 \lambda(s)$.) 

\begin{proof}[Proof of Corollary \ref{EHI:c:mainNoM2}(c)]
    Let $X$ be an isotropic unimodal L\'{e}vy jump process with jump kernel $j(r)$, satisfying \eqref{EHI:e:regularJumps}. Suppose $j(r) \gtrsim r^{-d}$ as $r \to 0^+$, and there exist $C, R,\alpha>0$ such that
    \begin{equation}\label{EHI:e:mainNoM2c1}
        \frac{s^{d+2} j(s)}{r^{d+2} j(r)} \leq C \left( \frac{\log(r^{-1})}{\log(s^{-1})} \right)^{1+\alpha} \qquad\mbox{for all $0 < r \leq s \leq R$}.
    \end{equation}
    For all $r \in (0, R]$,
    \begin{equation}\label{EHI:e:mainNoM2c2}
        \lambda(r) = \int_{B(0, R)^c} j(|x|) \dee{x} + \int_{B(0, R) \setminus B(0, r)} j(|x|) \dee{x}
    \end{equation}
    We handle each of the two integrals in the right-hand side of \eqref{EHI:e:mainNoM2c2} separately. The first of these integrals, $\int_{B(0, R)^c} j(|x|) \dee{x}$, remains constant as $r \to 0^+$. By converting the second integral into polar coordinates, then using \eqref{EHI:e:mainNoM2c1} to get an upper bound on $j(s)$ for $s>r$, and then using Lemma \ref{EHI:l:SpecialIntegralAsymptotic} to evaluate the remaining integral,
    \begin{align*}
        \int_{B(0, R) \setminus B(0, r)} j(|x|) \dee{x} &\asymp \int_r^R s^{d-1} j(s) \dee{s} \\
        &\lesssim \left(\log(r^{-1}) \right)^{1+\alpha} r^{d+2} j(r) \int_r^R s^{-3} \left(\log(s^{-1}) \right)^{-1-\alpha} \dee{s} \\
        &\asymp \left(\log(r^{-1}) \right)^{1+\alpha} r^{d+2} j(r) \cdot r^{-2} \left(\log(r^{-1}) \right)^{-1-\alpha} \\
        &=r^d j(r) \qquad\mbox{as $r \to 0^+$}.
    \end{align*}
    Since $j(r) \gtrsim r^{-d}$ as $r \to 0^+$, the second integral dominates the first one. Thus, $\lambda(r) \lesssim r^d j(r)$. Equivalently, $r^2\lambda(r) \lesssim r^{d+2} j(r)$. By Lemma \ref{EHI:l:SmallestOfThreeQuantities}, we also have $r^2 \lambda(r) \gtrsim r^{d+2} j(r)$ and $m_2(r) \gtrsim r^{d+2} j(r)$. Putting these all together, $m_2(r) \gtrsim r^{d+2} j(r) \asymp r^2 \lambda(r)$ as $r \to 0^+$. Therefore, by Corollary \ref{EHI:c:main}(a), $X$ satisfies $\EHI(r \leq 1)$.
\end{proof}

\begin{proof}[Proof of Corollary \ref{EHI:c:mainNoM2}(d)]
    Let $X$ be an isotropic unimodal L\'{e}vy jump process with jump kernel $j(r)$, satisfying \eqref{EHI:e:regularJumps}. Suppose there exist $C, R,\alpha>0$ such that
    \begin{equation}\label{EHI:e:mainNoM2d1}
        \frac{s^{d+2} j(s)}{r^{d+2} j(r)} \leq C \left( \frac{\log s}{\log r} \right)^{1+\alpha} \qquad\mbox{for all $R \leq r \leq s$}.
    \end{equation}
    For $r \geq R$, by converting the integral for $\lambda(r)$ into polar coordinates, then using \eqref{EHI:e:mainNoM2d1} to get an upper bound on $j(s)$ for $s>r$, and then using Lemma \ref{EHI:l:SpecialIntegralAsymptotic'} to evaluate the remaining integral,
    \begin{align*}
        \lambda(r) &=  \int_{B(0, r)^c} j(|x|) \dee{x} \asymp \int_r^\infty s^{d-1} j(s) \dee{s} \\
        &\lesssim (\log r)^{1+\alpha} r^{d+2} j(r) \int_r^\infty s^{-3} (\log s)^{1+\alpha} \dee{s} \\
        &= (\log r)^{1+\alpha} r^{d+2} j(r) \cdot r^{-2} (\log r)^{1+\alpha} \\
        &\asymp r^d j(r) \qquad\mbox{as $r \to \infty$}.
    \end{align*}
    Equivalently, $r^2\lambda(r) \lesssim r^{d+2} j(r)$. By Lemma \ref{EHI:l:SmallestOfThreeQuantities}, we also have $r^2 \lambda(r) \gtrsim r^{d+2} j(r)$ and $m_2(r) \gtrsim r^{d+2} j(r)$. Putting these all together, $m_2(r) \gtrsim r^{d+2} j(r) \asymp r^2 \lambda(r)$ as $r \to \infty$. Therefore, by Corollary \ref{EHI:c:main}(b), $X$ satisfies $\EHI(r \geq 1)$.
\end{proof}

\subsection{Calculating the jump kernel in Examples \ref{EHI:ex:geometric}-\ref{EHI:ex:iterated}}\label{EHI:ss:IteratedRelativistic}

In this subsection, we perform the calculations to asymptotically approximate $j(r)$ for small $r$ for the geometric stable, relativistic geometric stable, and iterated geometric stable processes. The first of these gives us an alternate proof of $\EHI(r \leq 1)$ in Example \ref{EHI:ex:geometric} The latter two of these are necessary to prove the claims we made in Examples \ref{EHI:ex:relativistic} and \ref{EHI:ex:iterated}.

Kim and Mimica \cite[Proposition 4.2]{KM} show that for any subordinate Brownian motion satisfying \ref{EHI:KM:A2} and \ref{EHI:KM:A3},
    \begin{equation*}
        j(r) \asymp r^{-d-2} \phi'(r^{-2}).
    \end{equation*}
For the geometric stable process,
    \begin{equation*}
        \phi'(\lambda) = \frac{\beta}{2} \cdot \frac{\lambda^{\beta/2-1}}{1+\lambda^{\beta/2}} \asymp \lambda^{-1} \qquad\mbox{as $\lambda \to \infty$},
    \end{equation*}
    so
    \begin{equation*} \label{EHI:e:GeometricStableJumpKernel}
        j(r) \asymp r^{-d-2} \phi'(r^{-2}) \asymp r^{-d} \qquad\mbox{as $r \to 0^+$}.
    \end{equation*}

    For the relativistic geometric stable process,
    \begin{equation*}
        \phi'(\lambda) = \frac{2}{\beta} \cdot \frac{(\lambda+m^{\beta/2})^{2/\beta-1}}{1+(\lambda+m^{\beta/2})^{2/\beta}-m} \asymp \lambda^{-1} \qquad\mbox{as $\lambda \to \infty$},
    \end{equation*}
    so
    \begin{equation*}\label{EHI:e:RelativisticGeometricStableJumpKernel}
        j(r) \asymp r^{-d} \qquad\mbox{as $r \to 0^+$},
    \end{equation*}
    similarly.

    For the iterated geometric stable process, $\phi_n$ is the $n$-fold composition of $\phi_1(\lambda)=\log(1+\lambda^{\beta/2})$. For all $n$, let $\log^{(n)}$ denote the $n$-fold composition of $\log$. We will show by induction that
    \begin{equation*}
        \phi_n(\lambda) \asymp \log^{(n)}(\lambda), \qquad \phi_n'(\lambda) \asymp \frac{1}{\lambda \cdot \log\lambda \cdot \log^{(2)}\lambda \cdot \cdots \cdot \log^{(n-1)}(\lambda) }, \qquad \mbox{as $\lambda\to0^+$},
    \end{equation*}
    where the constants implicit in $\asymp$ may depend on $n$. For the base case, see the calculation we have already done for the geometric stable process.
    For the inductive step, if the result holds for $n-1$, then
    \begin{align*}
        \phi_n(\lambda) &= \log(1+(\phi_{n-1}(\lambda))^{\beta/2}) \\
        &\asymp \log(\phi_{n-1}(\lambda)) \qquad\mbox{(since $\phi_{n-1}(\lambda) \to \infty$ when $\lambda \to \infty$)} \\
        & \asymp \log^{(n)}(\lambda) \qquad\mbox{(by the inductive hypothesis)}
    \end{align*}
    and
    \begin{align*}
        \phi_n'(\lambda) &= \phi_1'(\phi_{n-1}(\lambda)) \phi_{n-1}'(\lambda) \\
        &= \frac{1}{\phi_{n-1}(\lambda)} \phi_{n-1}'(\lambda) \qquad\mbox{(since $\phi_{n-1}(\lambda) \to \infty$ when $\lambda\to\infty$)}\\
        &\asymp \frac{1}{\log^{(n-1)}(\lambda)} \cdot \frac{1}{\lambda \cdot \log\lambda \cdot \log^{(2)}\lambda \cdot \cdots \cdot \log^{(n-2)}(\lambda) } \quad \mbox{as $\lambda\to0^+$} \qquad\mbox{(by the inductive hypothesis)}.
    \end{align*}
    Thus, if the subordinator has Laplace exponent $\phi_n$, then
    \begin{align*}
        j(r) &\asymp \frac{r^{-d-2}}{r^{-2} \log(r^{-2}) \log^{(2)}(r^{-2}) \cdots \log^{(n-1)}(r^{-2})} = \frac{r^{-d}}{\log(r^{-2}) \log^{(2)}(r^{-2}) \cdots \log^{(n-1)}(r^{-2})} \qquad \mbox{as $r\to 0^+$}.
    \end{align*}
    
\subsection{\texorpdfstring{Example \ref{EHI:ex:NoA3example}: a subordinate Brownian that does not satisfy \ref{EHI:KM:A3} but satisfies $\EHI(r \leq 1)$}{TEXT}} \label{EHI:ss:NoA3example}

In this subsection, we prove all the claims we made in Example \ref{EHI:ex:NoA3example}.
Let $X=(X_t)_{t \geq 0} = (B_{S_t})_{t \geq 0}$ be a subordinate Brownian motion on $\R^d$, such that the subordinator $S=(S_t)_{t \geq 0}$ has drift $0$ and a L\'{e}vy measure satisfying $\mu(dt) \asymp t^{-2} \left(\log(t^{-1}) \right)^{-2} \dee{t}$ as $t \to 0^+$.
We show that the Laplace exponent of $S$ satisfies $\phi(\lambda) \gtrsim \lambda/\log\lambda$.
We show that $X$ does not satisfy Kim and Mimica's condition \ref{EHI:KM:A3}. Finally, we asymptotically approximate the jump kernel of $X$, its truncated second moments, and its tails, to show that $m_2(r) \gg r^2 \lambda(r) \asymp r^{d+2} j(r)$ as $r \to 0^+$, and therefore $X$ satisfies $\EHI(r \leq 1)$ by Corollary \ref{EHI:c:main}(a).

\begin{lemma}\label{EHI:l:NoA3example1}
    Let $S$ be a subordinator with Laplace exponent $\phi$, drift $\gamma=0$, and L\'{e}vy measure $\mu$.
    If $\mu(dt) \gtrsim t^{-2} (\log(t^{-1}))^{-2} \dee{t}$ as $t \to 0^+$, then
    \begin{equation*}
        \phi(\lambda) \gtrsim \frac{\lambda}{\log \lambda} \qquad \mbox{as $\lambda \to \infty$}.
    \end{equation*}
\end{lemma}

\begin{proof}
    Choose a $\delta>0$ such that $\mu(\mathrm{d}t) \gtrsim t^{-2} (\log(t^{-1}))^{-2} \dee{t}$ for $t \in (0, \delta]$. Fix $\lambda>\delta^{-1}$.
    If we start with \eqref{EHI:e:driftAndLevy}, use the fact that $1-e^{-\lambda t} \asymp \lambda t$ for $\lambda t \leq 1$, and ignore all $t > 1/\lambda$, we see that
    \begin{equation*}
        \phi(\lambda) \geq \int_{(0, 1/\lambda]} (1-e^{-\lambda t}) \, \mu(\mathrm{d}t) \asymp \int_{(0, 1/\lambda]} \lambda t \, \mu(\mathrm{d}t) \gtrsim \int_0^{1/\lambda} \lambda t^{-1} (\log(t^{-1}))^{-2} \dee{t} \qquad\mbox{as $\lambda \to \infty$}.
    \end{equation*}
    By substituting $u=\log(t^{-1})$ and $\dee{u}=-t^{-1}\dee{t}$, we obtain
    \begin{equation*}
        \phi(\lambda) \gtrsim \lambda \int_{\log\lambda}^\infty u^{-2}\dee{u} = \frac{\lambda}{\log\lambda} \qquad \mbox{as $\lambda \to \infty$}.
    \end{equation*}
\end{proof}

\begin{lemma}\label{EHI:l:NoA3example2}
    Let $S$ be a subordinator with Laplace exponent $\phi$. If $\phi(\lambda) \gtrsim \lambda/\log\lambda$ as $\lambda\to\infty$, then \ref{EHI:KM:A3} is not satisfied.
\end{lemma}

\begin{proof}
    Assume for the sake of contradiction that we have both $\phi(\lambda) \gtrsim \lambda/\log\lambda$ and \ref{EHI:KM:A3}. Let $\delta$, $\sigma$, and $\lambda_0$ be the constants from \ref{EHI:KM:A3}.
    First, assume $\delta \in (0, 1)$. Then, for all $\lambda_1 \geq \lambda_0$, by \ref{EHI:KM:A3} and the Fundamental theorem of calculus, we have
    \begin{align*}
        \phi(\lambda_1) - \phi(\lambda_0) &= \int_{\lambda_0}^{\lambda_1} \phi'(\lambda) \dee{\lambda} \\
        &\leq \int_{\lambda_0}^{\lambda_1} \sigma \left( \frac{\lambda}{\lambda_0} \right)^{-\delta} \phi'(\lambda_0) \dee{\lambda} \\
        &= \frac{\sigma \lambda_0^\delta \phi'(\lambda_0)}{1-\delta} (\lambda_1^{1-\delta} - \lambda_0^{1-\delta}).
    \end{align*}
    Therefore, for large $\lambda$, we have both $\phi(\lambda) \lesssim \lambda^{1-\delta}$ and $\phi(\lambda) \gtrsim \lambda/\log\lambda$, a contradiction.
    If $\delta=1$, a similar calculation shows that for large $\lambda$, we have both $\phi(\lambda) \lesssim \log\lambda$ and $\phi(\lambda) \gtrsim \lambda/\log\lambda$, another contradiction.
\end{proof}

\begin{lemma}\label{EHI:l:NoA3example3}
    Let $S$ be a subordinator with Laplace exponent $\phi$, drift $\gamma=0$, and L\'{e}vy measure $\mu$. Let $X=(X_t)_{t \geq 0} = (B_{S_t})_{t \geq 0}$ be a subordinate Brownian motion with subordinator $S$, and let $j(r)$ be the jump kernel of $X$.
    If $\mu(dt) \asymp t^{-2} (\log(t^{-1}))^{-2} \dee{t}$ as $t \to 0^+$, then
    \begin{equation*}
        j(r) \asymp r^{-d-2} \left(\log(r^{-1}) \right)^{-2} \qquad \mbox{as $r \to 0^+$}.
    \end{equation*}
\end{lemma}

\begin{proof}
    Note that $\log(t^{-1}) \asymp \log(e+t^{-1})$ as $t \to 0^+$. Therefore, the hypothesis of this lemma can be written alternatively as $\mu(dt) \asymp t^{-2} (\log(e+t^{-1}))^{-2} \dee{t}$ as $t \to 0^+$. Choose a $\delta>0$ such that $\mu(dt) \asymp t^{-2} (\log(e+t^{-1}))^{-2} \dee{t}$ on $(0, \delta]$. (Now that the logarithmic term is bounded for large $t$, the following proof will be simpler.)
    By \eqref{EHI:e:JumpKernelIntermsofMu}, we have $j(r)=\mathcal{A}(r) + \mathcal{B}(r)$, where
    \begin{align}
        \mathcal{A}(r) &:= \int_{(0, \delta]} (4\pi t)^{-d/2} e^{-\frac{r^2}{4t}} \, \mu(\mathrm{d}t) \asymp \int_0^\delta t^{-d/2} e^{-\frac{r^2}{4t}} t^{-2} (\log(e+t^{-1}))^{-2} \dee{t}, \label{EHI:e:A(r)integral}\\
        \mathcal{B}(r) &:= \int_{(\delta, \infty)} (4\pi t)^{-d/2} e^{-\frac{r^2}{4t}} \, \mu(\mathrm{d}t). \notag\\
    \end{align}
    We will show that $\mathcal{A}(r) \asymp r^{-d-2} \left(\log(r^{-1}) \right)^{-2}$ and $\mathcal{B}(r) \lesssim 1$ for $r \to 0^+$. Therefore, $\mathcal{A}(r)$ is the dominant term, and $j(r) \asymp r^{-d-2} \left(\log(r^{-1}) \right)^{-2}$ as $r \to 0^+$.

    Let us first prove a constant upper bound on $\mathcal{B}(r)$. For all $r>0$,
    \begin{equation*}
        \mathcal{B}(r) \leq \mu((\delta, \infty)) (4\pi\delta)^{-d/2} \lesssim 1.
    \end{equation*}
    For $\mathcal{A}(r)$, let us use the substitution $u=\frac{r^2}{4t}$. Then $\mathrm{d}u = -\frac{r^2}{4} t^{-2} \dee{t}$ and $t=\frac{r^2}{4u}$. Applying this substitution to \eqref{EHI:e:A(r)integral} yields
    \begin{align*}
        \mathcal{A}(r) &\asymp \int_{\frac{r^2}{4\delta}}^\infty \left( \frac{r^2}{u} \right)^{-d/2} e^{-u} \left(\log(e+4ur^{-2}) \right)^{-2} r^{-2} \dee{u} \\
        &= r^{-d-2} \int_0^\infty \indicatorWithSetBrackets{u>\frac{r^2}{4\delta}} \cdot u^{d/2} e^{-u} \left(\log(e+4ur^{-2}) \right)^{-2} \dee{u}.
    \end{align*}
    Dividing both sides by $r^{-d-2}\left(\log(r^{-1}) \right)^{-2}$,
    \begin{equation} \label{EHI:e:A(r)beforeDCT}
        \frac{\mathcal{A}(r)}{r^{-d-2} \left(\log(r^{-1}) \right)^{-2}} \asymp \int_0^\infty \indicatorWithSetBrackets{u>\frac{r^2}{4\delta}} \cdot u^{d/2} e^{-u} \left( \frac{\log(e+4u r^{-2})}{\left(\log(r^{-1}) \right)^{-2}} \right)^{-2} \dee{u}.
    \end{equation}
    Fix $r \in (0, 1)$.
    We would like to apply the Dominated convergence theorem to compute the limit of the right-hand side of \eqref{EHI:e:A(r)beforeDCT} as $r \to 0^+$. First, we must show that the integrand is bounded by an integrable function. (This integrable function is allowed to depend on $r$, since $r$ is fixed at the time we apply the Dominated convergence theorem.) For all $u>0$, because $\log(e+4u r^{-2}) \geq 1$, we have
    \begin{equation*}
        \left( \frac{\log(e+4u r^{-2})}{\left(\log(r^{-1}) \right)^{-2}} \right)^{-2} \leq \left(\log(r^{-1}) \right)^2.
    \end{equation*}
    Thus, the integrand of right-hand side of \eqref{EHI:e:A(r)beforeDCT} is bounded above by the function
    \begin{equation*}
        u \mapsto u^{d/2} e^{-u} \left(\log(r^{-1}) \right)^2,
    \end{equation*}
    which is integrable, since
    \begin{equation*}
        \int_0^\infty u^{d/2} e^{-u} \left(\log(r^{-1}) \right)^2 \dee{u} = \left(\log(r^{-1}) \right)^2 \int_0^\infty u^{d/2} e^{-u} \dee{u} < \infty.
    \end{equation*}
    Therefore, we can apply the Dominated convergence theorem to the right-hand side of \eqref{EHI:e:A(r)beforeDCT}. The result is
    \begin{equation} \label{EHI:e:ApplyDctA(r)}
        \lim_{r \to 0^+} \int_0^\infty \indicatorWithSetBrackets{u>\frac{r^2}{4\delta}} \cdot u^{d/2} e^{-u} \left( \frac{\log(e+4u r^{-2})}{\left(\log(r^{-1}) \right)^{-2}} \right)^{-2} \dee{u} = \int_0^\infty u^{d/2} e^{-u} \cdot 2^{-2} \dee{u}.
    \end{equation}
    The right-hand side of \eqref{EHI:e:ApplyDctA(r)} is a constant depending only on $d$. Therefore, by \eqref{EHI:e:A(r)beforeDCT} and \eqref{EHI:e:ApplyDctA(r)},
    \begin{equation*}
        \frac{\mathcal{A}(r)}{r^{-d-2} \left(\log(r^{-1}) \right)^{-2}} \asymp 1 \qquad \mbox{as $r \to 0^+$}.
    \end{equation*}
    Since $\mathcal{A}(r) \asymp r^{-d-2} \left(\log(r^{-1}) \right)^{-2}, r \to 0^+$ and $\mathcal{B}(r) \lesssim 1$, we have
    \begin{equation*}
        j(r) = \mathcal{A}(r)+\mathcal{B}(r) \asymp r^{-d-2} \left(\log(r^{-1}) \right)^{-2}\qquad \mbox{as $r \to 0^+$}.
    \end{equation*}
\end{proof}

\begin{lemma}\label{EHI:l:NoA3example4}
    Let $X$ be an isotropic unimodal L\'{e}vy jump process on $\R^d$ such that $j(r) \asymp r^{-d-2} \left(\log(r^{-1}) \right)^{-2}$ as $r \to 0^+$. Then
    \begin{equation*}
        m_2(r) \asymp \left(\log(r^{-1})\right)^{-1} \qquad\mbox{as $r \to 0^+$}
    \end{equation*}
    and
    \begin{equation*}
        \lambda(r) \asymp r^{-2}\left(\log(r^{-1})\right)^{-2} \qquad\mbox{as $r \to 0^+$}.
    \end{equation*}
\end{lemma}

\begin{proof}
    Choose a small $\delta$ such that $j(r) \asymp r^{-d-2} \left(\log(r^{-1}) \right)^{-2}$ on $(0, \delta]$.
    Using polar coordinates and the substitution $u=\log(s^{-1})$, we see that for $r \in (0, \delta]$,
    \begin{align*}
        m_2(r) &\asymp \int_0^r s^{d-1} s^2 j(s) \dee{s} \asymp \int_0^r s^{-1} (\log(s^{-1}))^{-2} \dee{s} = \int_{\log(r^{-1})}^\infty u^{-2} \dee{u} = \left(\log(r^{-1}) \right)^{-1}.
    \end{align*}
    A similar calculation tells us that for all $r \in (0, \delta]$,
    \begin{equation}\label{EHI:e:A3counterexLambda}
        \lambda(r) \asymp \int_r^\infty s^{d-1} j(s) \dee{s} \asymp  \int_\delta^\infty s^{d-1} j(s) \dee{s} + \int_r^\delta s^{-3} (\log(s^{-1}))^{-2} \dee{s}.
    \end{equation}
    The first integral in the right-hand side of \eqref{EHI:e:A3counterexLambda} remains constant as $r \to 0^+$. By Lemma \ref{EHI:l:SpecialIntegralAsymptotic}, the second integral is on the order of $r^{-2} \left(\log(r^{-1}) \right)^{-2}$ as $r \to 0^+$. Therefore, the second integral dominates and $\lambda(r) \asymp r^{-2} \left(\log(r^{-1}) \right)^{-2}$ as $r \to 0^+$.
\end{proof}

Let us verify that we have now proved every claim we made in Example \ref{EHI:ex:NoA3example}.
Let $X$ be a subordinate Brownian motion on $\R^d$, such that the subordinator $S$ has drift $0$ and L\'{e}vy measure $\mu(dt) \asymp t^{-2} \left(\log(t^{-1}) \right)^{-2} \dee{t}$ as $t \to 0^+$. Let $\phi(\lambda)$ be the Laplace exponent of $S$.
By Lemma \ref{EHI:l:NoA3example1}, $\phi(\lambda) \gtrsim \lambda/\log\lambda$.
By Lemma \ref{EHI:l:NoA3example2}, $X$ does not satisfy \ref{EHI:KM:A3}.
By Lemma \ref{EHI:l:NoA3example3}, $j(r) \asymp r^{-d-2} \left(\log(r^{-1}) \right)^{-2}$, or equivalently $r^{d+2}j(r) \asymp \left(\log(r^{-1}) \right)^{-2}$ as $r \to 0^+$. By Lemma \ref{EHI:l:NoA3example4}, $m_2(r) \asymp \left(\log(r^{-1}) \right)^{-1}$ and $r^2 \lambda(r) \asymp \left(\log(r^{-1}) \right)^{-2}$ as $r \to 0^+$.

\section{Negative results and counterexample}

In this section, we prove Theorem \ref{EHI:t:mainNegative} and show that the process described in Example \ref{EHI:ex:counterexample} is subordinate Brownian motion with jump kernel satisfying \eqref{EHI:e:regularJumps} for which $\EHI(r \leq 1)$ fails.

\subsection{A recipe for counterexamples}

The following proposition is the basis for all of our negative results.
It gives a general recipe for identifying counterexamples to $\EHI$. Its statement is somewhat abstract and refers to a sequence of Meyer-decoposition-arising stopping times, defined in Definition \ref{EHI:d:meyerStoppingTimes} (stopping times $T^{(r_0)}$, $T^{\inner{r_0}}$, or $T^{\brackets{s}}$). The only properties of Meyer-decoposition-arising stopping times that we use in the proof are that they are memoryless and independent of what $X$ does up until they occur.
Using this general recipe, we will go on to derive more concrete recipes that are easier to directly apply. One of these more concrete recipes (Corollary \ref{EHI:c:j(r)Counterex}) involves directly looking at the jump kernel of $X$, while the other (Proposition \ref{EHI:p:SbmCounterex}) only applies to subordinate Brownian motions and involves looking at the L\'{e}vy measure of the subordinator instead. Recall that given sequences $(a_n)$ and $(b_n)$ of positive numbers, we say $a_n \ll b_n$ (or equivalently, $b_n \gg a_n$) whenever $\lim_{n \to \infty} a_n / b_n = 0$.

\begin{prop}\label{EHI:p:GeneralCounterex}
    Let $X=(X_t)_{t \geq 0}$ be an isotropic unimodal L\'{e}vy jump process. Suppose there exist sequences $(R_{1, n})_{n=1}^\infty, (R_{2, n})_{n=1}^\infty, (R_{3, n})_{n=1}^\infty \subseteq (0, \infty)$ and a sequence $(T_n)$ of Meyer-decomposition-arising exponential stopping times, satisfying the following properties:
    \begin{itemize}
        \item For all $n$,
        \begin{equation*}
            0 < R_{1, n} < R_{2, n} < R_{1, n}+R_{2, n} < R_{3, n}.
        \end{equation*}
        \item There exists a $C>0$ such that for all $n$,
        \begin{equation} \label{EHI:e:GeneralCounterexR2R1}
            R_{2, n} \leq C R_{1, n}.
        \end{equation}
        \item We have
        \begin{equation} \label{EHI:e:GeneralCounterexSandwich}
            1 \gg \P_0 \left( \tau_{B(0, R_{1, n})} < T_n \right) \gg \P_0 \left( |\Delta X(T_n)| \leq R_{1, n}+R_{2, n}+R_{3, n} \right).
        \end{equation}
        \item As $n \to \infty$,
        \begin{equation} \label{EHI:e:GeneralCounterexLeaving1}
            \sup_{x \in B(0, R_{1, n})} \P_x \left( X_{\tau_{B(0, R_{1, n})}} \notin B(0, R_{2, n}) \Big| \tau_{B(0, R_{1, n})} < T_n \right) \to 0. 
        \end{equation}
        \item There exists a $c>0$ such that for all $n$, for all $x \in B(0, R_{1, n}+R_{2, n})$,
        \begin{equation}\label{EHI:e:GeneralCounterexLeaving1+2}
            \P_x \left( X_{\tau_{B(0, R_{1, n}+R_{2, n})}} \in B(0, R_{3, n}) \Big| \tau_{B(0, R_{1, n}+R_{2, n})} < T_n \right) \geq c.
        \end{equation}
    \end{itemize}
    Then $X$ does not satisfy $\EHI$. Furthermore, if $r_n \to 0$ along a subsequence then $X$ does not satisfy $\EHI(r \leq 1)$, and if $r_n \to \infty$ along a subsequence then $X$ does not satisfy $\EHI(r \geq 1)$.
\end{prop}

Before we give a proof, let us explain the significance of some of these conditions. Condition \eqref{EHI:e:GeneralCounterexSandwich} is the most delicate: it says that the process is unlikely to exit a ball of radius $R_{1, n}$ before time $T_n$, but much more unlikely to stay within a distance of $R_{1, n}+R_{2, n}+R_{3, n}$ when the ``scattering jump" (at time $T_n$) occurs. In practice, this will be the hardest condition to satisfy. It will be achieved by constructing processes where for all $n$, most jumps either belong to a much smaller scale than the $R_{j, n}$'s, or a much larger scale. Thus, it is unlikely that the small jumps (before $T_n$) alone are enough to combine to reach this scale, but the first large jump (at $T_n$) far surpasses this scale.
Conditions \eqref{EHI:e:GeneralCounterexLeaving1} and \eqref{EHI:e:GeneralCounterexLeaving1+2} tell us that at when the process exits a ball using only ``small jumps" (ie, before time $T_n$), the first point outside the ball that it reaches is unlikely to be too far from the ball it just exited. In practice, these conditions will follow easily from our constructions.

\begin{proof}[Proof of Proposition \ref{EHI:p:GeneralCounterex}]
    Fix $n$. For the sake of brevity, suppress the subscript and let $R_j := R_{j, n}$ for all $j \in \{1, 2, 3\}$ and $T:=T_n$. In the remainder of the proof, we will construct objects ($y$ and $h$) that depend on $n$. When we use the notation ``$\ll$" or ``$\gg$" for quantities relating to these objects, we mean this when considering sequences indexed by $n$.

    Let $h$ be the function
    \begin{equation*}
        h(x) := \P_x \left( X_{\tau_{B(0, R_1+R_2)}} \in B(0, R_3) \right),
    \end{equation*}
    which is harmonic on $B(0, R_1+R_2)$. Let $y := ((1-\frac{1}{d})R_1+R_2) \vec{e_1}$, where $d$ is the dimension and $\vec{e_1}$ is the unit vector $(1, 0, 0, \dots, 0)$. Note that $y$ is at a distance of $R_1/d$ from the boundary of $B(0, R_1+R_2)$. By \eqref{EHI:e:GeneralCounterexR2R1},
    \begin{equation*}
        \frac{|y|}{R_1+R_2} = \frac{1-\frac{1}{d}+\frac{R_2}{R_1}}{1+\frac{R_2}{R_1}} \leq \frac{1-\frac{1}{d}+C}{1+C}.
    \end{equation*}
    Let $\kappa := (1-\frac{1}{2d}+C)/(1+C)$. Then both $0$ and $y$ are contained in the ball $B(0, \kappa(R_1+R_2))$. Note that $\kappa \in (0, 1)$ does not depend on $n$. We will show that $h(0) \ll h(y)$, contradicting $\EHI$.

    Let us start by putting a lower bound on $h(y)$. For all $x \in B(0, R_1+R_2)$, by \eqref{EHI:e:GeneralCounterexLeaving1+2},
    \begin{align*}
        h(x) &= \P_x \left( X_{\tau_{B(0, R_1+R_2)}} \in B(0, R_3) \right) \\
        &\geq \P_x \left( \tau_{B(0, R_1+R_2)} < T \right) \P_x \left( X_{\tau_{B(0, R_1+R_2)}} \in B(0, R_3) \Big| \tau_{B(0, R_1+R_2)} < T \right) \\
        &\geq c \P_x \left( \tau_{B(0, R_1+R_2)} < T \right).
    \end{align*}
    In particular,
    \begin{equation} \label{EHI:e:GeneralCounterexh(y)1}
        h(y) \geq c \P_y \left( \tau_{B(0, R_1+R_2)} < T \right).
    \end{equation}
    We will use a symmetry argument to show that if the process has initial value $y$ and exits $B(y, R_1)$ before time $T$, then the conditional probability that it has exited $B(0, R_1+R_2)$ by this time is at least $\frac{1}{2d}$. Consider the vector $X_{\tau_{B(X_0, R_1)}}-X_0$, whose distribution is radially symmetric. The distribution of this vector remains radially symmetric when we condition on the event $\{ \tau_{B(X_0, R_1)} < T \}$. For all $1 \leq j \leq d$, let $A^+_j$ (resp, $A^-_j$) be the event that the coordinate of $X_{\tau_{B(X_0, R_1)}}-X_0$ with the largest magnitude is the $j$th coordinate, and that this coordinate is positive (resp, negative). By radial symmetry, for all $1 \leq j \leq d$,
    \begin{equation*}
        \P_y(A^+_j) = \P_y(A^-_j) = \P_y \left(A^+_j \Big| \tau_{B(X_0, R_1)}<T \right) = \P_y \left(A^-_j \Big| \tau_{B(X_0, R_1)}<T \right) = \frac{1}{2d}.
    \end{equation*}
    Note that at least one coordinate of $X_{\tau_{B(X_0, R_1)}}-X_0$ must have magnitude at least $R_1/d$, because the vector itself has magnitude at least $R_1$. Therefore, if the process begins at $X_0=y$ and the events $\{\tau_{B(y, R_1)} < T \}$ and $A^+_1$ both happen, then the first coordinate of $X_{\tau_{B(y, R_1)}}$ is at least $(1-\frac{1}{d})R_1 + R_2 + \frac{R_1}{d} = R_1+R_2$, which means $X_{\tau_{B(y, R_1)}}$ is outside of $B(0, R_1+R_2)$, so $\tau_{B(0, R_1+R_2)} < T$. In other words, assuming $X_0=y$, we have
    \begin{equation*}
        \{\tau_{B(y, R_1)} < T \} \cap A^+_1 \subseteq \{ \tau_{B(0, R_1+R_2)} < T \}.
    \end{equation*}
    Therefore,
    \begin{align}
        \P_y \left( \tau_{B(0, R_1+R_2)} < T \right) &\geq \P_y \left( {\tau_{B(y, R_1)} < T} \} \cap A^+_1 \right) \notag\\
        &= \P_y \left( \tau_{B(y, R_1)} < T \right) \P_y \left( A^+_1 \Big| \tau_{B(y, R_1)} < T \right) \notag\\
        &= \P_0 \left( \tau_{B(0, R_1)} < T \right) \cdot \frac{1}{2d}. \label{EHI:e:GeneralCounterexSymmetry}
    \end{align}
    By \eqref{EHI:e:GeneralCounterexh(y)1} and \eqref{EHI:e:GeneralCounterexSymmetry},
    \begin{equation}\label{EHI:e:GeneralCounterexh(y)2}
        h(y) \geq \frac{c}{2d} \P_0 \left( \tau_{B(0, R_1)} < T \right).
    \end{equation}

    Next, let us put an upper bound on $h(0)$. If the process begins at $X_0=0$, there are three ways that the event $X_{\tau_{B(0, R_1+R_2)}} \in B(0, R_3)$ can occur:
    \begin{itemize}
        \item First, the process exits $B(0, R_1)$ with $X_{\tau_{B(0, R_1)}} \in B(0, R_2)$, then travels another distance of at least $R_1$ to exit $B(0, R_1+R_2)$ with $X_{\tau_{B(0, R_1+R_2)}} \in B(0, R_3)$, all before time $T$. This involves the process exiting a ball of radius $R_1$ twice before time $T$, so by the Strong Markov property and the stationary property of $X$, the probability of this happening is at most $\left[ \P_0 \left( \tau_{B(0, R_1) < T} \right) \right]^2$.
        \item The process exits $B(0, R_1)$ with $X_{\tau_{B(0, R_1)}} \notin B(0, R_2)$, and either is already in $B(0, R_3) \setminus B(0, R_1+R_2)$ at time $\tau_{B(0, R_1)}$, or is still in $B(0, R_1+R_2)$ and goes on to have $X_{\tau_{B(0, R_1+R_2)}} \in B(0, R_3)$, all before time $T$. This involves both $\tau_{B(0, R_1)} < T$ and $X_{\tau_{B(0, R_1)}} \notin B(0, R_2)$, so the probability of this happening is at most
        \begin{equation*}
            \P_0 \left( \tau_{B(0, R_1)} < T \right) \P \left( X_{\tau_{B(0, R_1)}} \notin B(0, R_2) \Big| \tau_{B(0, R_1)} < T\right) .          
        \end{equation*}
        \item The stopping time $T$ occurs before or at time $\tau_{B(0, R_1+R_2)}$, and $X_{\tau_{B(0, R_1+R_2)}} \in B(0, R_3)$. In this case, in order to have $X_{\tau_{B(0, R_1+R_2)}} \in B(0, R_3)$, we must have both $X(T-) \in B(0, R_1+R_2)$ and $X(T) \in B(0, R_3)$, so $|\Delta X(T)| = |X(T)-X(T-)| \leq R_1+R_2+R_3$. Thus, the probability of this happening is at most $\P_0 \left( |\Delta X(T)| \leq R_1+R_2+R_3 \right)$.
    \end{itemize}
    Since these are the only three ways we can have $X_{\tau_{B(0, R_1+R_2)}} \in B(0, R_3)$ when $X_0=0$,
    \begin{align} \label{EHI:e:GeneralCounterexh(0)}
        h(0) &= \P_0 \left( X_{\tau_{B(0, R_1+R_2)}} \in B(0, R_3) \right) \notag\\
        &\leq \left[ \P_0 \left( \tau_{B(0, R_1)} < T \right) \right]^2 + \P_0 \left( \tau_{B(0, R_1)} < T \right) \P_0 \left( X_{\tau_{B(0, R_1)}} \notin B(0, R_2) \Big| \tau_{B(0, R_1)} < T\right) \notag\\
        & \qquad +  \P_0 \left( |\Delta X(T)| \leq R_1+R_2+R_3 \right).
    \end{align}
    By \eqref{EHI:e:GeneralCounterexh(y)2}, \eqref{EHI:e:GeneralCounterexh(0)}, \eqref{EHI:e:GeneralCounterexLeaving1}, and \eqref{EHI:e:GeneralCounterexSandwich},
    \begin{align*}
        \frac{h(0)}{h(y)} &\leq \frac{2d}{c} \left[ \P_0 \left( \tau_{B(0, R_1)} < T \right) + \P_0\left( X_{\tau_{B(0, R_1)}} \notin B(0, R_2) \Big| \tau_{B(0, R_1)} < T\right) + \frac{\P_0 \left( |\Delta X(T)| \leq R_1+R_2+R_3 \right)}{\P_0 \left( \tau_{B(0, R_1)} < T \right)} \right] \\
        &\to 0.
    \end{align*}
    If we re-apply the subscripts that we have been suppressing, $h_n(0) \ll h_n(y_n)$, even though $0$ and $y_n$ belong to $B(0, \kappa (R_{1, n}+R_{2, n}))$ and $h_n$ is non-negative everywhere and harmonic on $B(0, R_{1, n}+R_{2, n})$. This contradicts $\EHI$. If $r_n \to 0$ along a subsequence, this contradicts $\EHI(r \leq 1)$. If $r_n \to \infty$ along a subsequence, this contradicts $\EHI(r \geq 1)$.
\end{proof}

Our first application of Proposition \ref{EHI:p:GeneralCounterex} is the following corollary.

\begin{corollary} \label{EHI:c:j(r)Counterex}
    Let $X=(X_t)_{t \geq 0}$ be an isotropic unimodal L\'{e}vy jump process on $\mathbb{R}^d$, and let $j(r)$ be its jump kernel. For all $r>0$, let $T^{(r)} := \inf \{ t>0 : |\Delta X(t)| \geq r\}$.
    Suppose there exist sequences $(r_n)$ and $(\rho_n)$ of positive numbers such that
    \begin{itemize}
        \item There exists a $c>0$ (that does not depend on $n$) such that for all $n$,
        \begin{equation}\label{EHI:e:RhoOnOrderOfR}
            \rho_n \geq c r_n.
        \end{equation}
        \item We have \begin{equation}\label{EHI:e:Hypothesis}
            1 \gg \mathbb{P}_0 \left( \tau_{B(0, \rho_n)} < T^{(r_n)} \right) \gg \P_0 (|\Delta X(T^{(r_n)})| \leq 4\rho_n+3r_n) .
        \end{equation}
    \end{itemize}
    Then $X$ does not satisfy $\EHI$. Furthermore, if $r_n \to 0$ along a subsequence then $X$ does not satisfy $\EHI(r \leq 1)$, and if $r_n \to \infty$ along a subsequence then $X$ does not satisfy $\EHI(r \geq 1)$.
\end{corollary}

\begin{proof}
    We would like to apply Proposition \ref{EHI:p:GeneralCounterex} with
    \begin{align*}
        R_{1, n} &:= \rho_n , \\
        R_{2, n} &:= \rho_n+r_n, \\
        R_{1, n}+R_{2, n} &:= 2\rho_n+r_n, \\
        R_{3, n} &:= 2\rho_n + 2r_n, \\
        T_n &:= T^{(r_n)}.
    \end{align*}
    We must verify that conditions \eqref{EHI:e:GeneralCounterexR2R1}-\eqref{EHI:e:GeneralCounterexLeaving1+2} hold. Condition \eqref{EHI:e:GeneralCounterexR2R1} follows from \eqref{EHI:e:RhoOnOrderOfR}:
    \begin{equation*}
        \frac{R_{2, n}}{R_{1, n}} = 1+\frac{r_n}{\rho_n} \leq 1+c^{-1}.
    \end{equation*}
    Condition \eqref{EHI:e:GeneralCounterexSandwich} is exactly \eqref{EHI:e:Hypothesis}. Conditions \eqref{EHI:e:GeneralCounterexLeaving1} and \eqref{EHI:e:GeneralCounterexLeaving1+2} hold trivially, because the process $X$ never takes any jumps with magnitude greater than $r_n$ before $T^{(r_n)}$, by the definition of $T^{(r_n)}$. Indeed,
    \begin{equation*}
        \sup_{x \in B(0, \rho_n)} \P_x \left( X_{\tau_{B(0, \rho_n)}} \notin B(0, \rho_n + r_n) \Big| \tau_{B(0, \rho_n)} < T^{(r_n)} \right) = 0
    \end{equation*}
    and
    \begin{equation*}
        \P_x \left( X_{\tau_{B(0, 2\rho_n+r_n}} \in B(0, 2\rho_n+2r_n) \Big| \tau_{B(0, 2\rho_n+r_n)} < T^{(r_n)} \right) = 1.
    \end{equation*}
\end{proof}

\begin{proof}[Proof of Theorem \ref{EHI:t:mainNegative}]
    Let $X=(X_t)_{t \geq 0}$ be an isotropic unimodal L\'{e}vy jujmp process on $\R^d$ with jump kernel $j(r)$. Suppose there exists a sequence $(r_n) \subseteq (0, \infty)$ satisfying \eqref{EHI:e:m2intermedCondition}. By re-arranging \eqref{EHI:e:m2intermedCondition}, this is equivalent to the existance of constant $c_0>0$ such that
    \begin{itemize}
        \item For all $n$,
        \begin{equation} \label{EHI:e:m2>=lambda}
            m_2(r_n) \geq c_0 r_n^2 \lambda(r_n).
        \end{equation}
        \item We have
        \begin{equation} \label{EHI:e:ComparingLevelsOfDominance}
            \left( \frac{m_2(r_n)}{r_n^2 \lambda(r_n)} \right)^d \ll \left( \frac{r_n^2 \lambda(r_n)}{r_n^{d+2} j(r_n)} \right)^2.
        \end{equation}
    \end{itemize}
    We will show that $X$ does not satisfy $\EHI$. If $r_n \to 0^+$, along a subsequence, then it can be gleaned from our proof that $\EHI(r \leq 1)$ fails. Similarly, if $r_n \to \infty$ along a subsequence, then $\EHI(r \geq 1)$ fails.

    Let $s_n := 2\sqrt{m_2(r_n) / \lambda(r_n)}$. By \eqref{EHI:e:m2>=lambda},
    \begin{equation*} 
        s_n \geq 2\sqrt{c_0} r_n.
    \end{equation*}
    (In other words, $s_n$ is at least on the order of $r_n$.) By Lemmas \ref{EHI:l:FactorOf2ForExits} and \ref{EHI:l:m2isvariance} applied to $X^{(r_n)}$,
    \begin{equation} \label{EHI:e:PExitSHalf}
        \mathbb{P}_0 \left( \tau_{B(0, s_n)} < T^{(r_n)} \right) \leq 2 \mathbb{P}_0 \left( \left| X^{(r_n)}_{T^{(r_n)}} \right| \geq s_n \right) \leq \frac{2}{s_n^2} \mathbb{E}_0 \left[ \left| X^{(r_n)}_{T^{(r_n)}} \right|^2 \right] = \frac{2 m_2(r_n)}{s_n^2 \lambda(r_n)} = \frac12.
    \end{equation}
    We would like to apply Lemma \ref{EHI:l:CloseToSubmultiplicative} to get an exponentially-decaying upper bound on $\mathbb{P}_0 \left( \tau_{B(0, \rho)} < T^{(r_n)} \right)$ for $\rho \geq s_n$. For a given $\rho \geq s_n$, the largest integer $m$ such that $\rho \geq ms_n + (m-1)r_n$ is $m := \floor{(\rho+r_n)/(s_n+r_n)}$. Thus, applying Lemma \ref{EHI:l:CloseToSubmultiplicative} gives us
    \begin{equation}\label{EHI:e:PexitRhoExponential}
        \mathbb{P}_0 \left( \tau_{B(0, \rho)} < T^{(r_n)} \right) \leq 2^{-\floor{(\rho+r_n)/(s_n+r_n)}} \qquad\mbox{for all $\rho \geq s_n$}.
    \end{equation}
    On the other hand, we can use Lemma \ref{EHI:l:LeaveBallConstant} to get a lower bound. For all $\rho \geq s_n$, if we let $C:=C_{\ref{EHI:l:LeaveBallConstant}}(1/2)$ and $t := C\rho^2 / m_2(r_n)$, then
    \begin{align}
        \mathbb{P}_0 \left( \tau_{B(0, \rho)} < T^{(r_n)} \right) &\geq \mathbb{P}_0 \left( \tau_{B(0, \rho)} < t < T^{(r_n)} \right) = \mathbb{P}_0 \left( \tau_{B(0, \rho)} < t \right) \mathbb{P}_0 \left( T^{(r_n)}>t \right) \notag\\
        &\geq \frac12 e^{-\lambda(r_n) t} = \frac12 \exp(-\frac{C \rho^2 \lambda(r_n)}{m_2(r_n)}) = \frac12 \exp(-\frac{4C\rho^2}{s_n^2}) \qquad\mbox{for all $\rho \geq s_n$}. \label{EHI:e:PhoExitAtLeastSquareExponential}
    \end{align}
    Let $\xi_n := s_n^d j(r_n) / \lambda(r_n)$. By \eqref{EHI:e:ComparingLevelsOfDominance},
    \begin{equation*}
        \xi_n = 2^d \left(\frac{m_2(r_n)}{\lambda(r_n)} \right)^{d/2} \frac{j(r_n)}{\lambda(r_n)} = 2^d \left(\frac{m_2(r_n)}{r_n^2 \lambda(r_n)} \right)^{d/2} \frac{r_n^d j(r_n)}{\lambda(r_n)} \ll 1.
    \end{equation*}
    Assume without loss of generality that $\xi_n^{1/3} \leq \frac12 e^{-4C}$ for all $n$. (Since $\xi_n \ll 1$, this is true for all but finitely many $n$, and we can simply disregard all $n$ for which this is not true.)

    We would like to choose a sequence $(\rho_n)$ such that $\rho_n \gg s_n$ but $\rho_n/s_n$ is not too large, so that we still have $\frac12 \exp(-4C \rho_n^2/s_n^2) \gg \rho_n^d j(r_n)/\lambda(r_n)$. Let $\tilde{\rho}_n \geq s_n$ and $\hat{\rho}_n \geq s_n$ be the solutions to
    \begin{equation*}
        \frac12 \exp(-\frac{4C(\tilde{\rho}_n)^2}{s_n^2}) = \xi_n^{1/3} \qquad\mbox{and}\qquad \frac{(\hat{\rho}_n)^d j(r_n)}{\lambda(r_n)} = \xi^{2/3}.
    \end{equation*}
    Let us briefly justify the existence and uniqueness of $\tilde{\rho}_n$ and $\hat{\rho}_n$. The function $\rho \mapsto \frac12 \exp(4C\rho^2/s_n^2)$ is continuous and decreasing on $[s_n, \infty]$, takes value $\frac12 e^{-4C} \geq \xi^{1/3}$ at $\rho=s_n$ and approaches $0$ as $\rho \to \infty$. Therefore, there exists some unique $\tilde{\rho}_n \geq s_n$ where this function achieves the value $\xi_n^{1/3}$. Similarly, the function $\rho \mapsto \rho^d j(r_n) / \lambda(r_n)$ is continuous and increasing on $[s_n, \infty)$, takes value $\xi_n \leq \xi_n^{2/3}$ at $s_n$, and approaches $\infty$ as $\rho \to \infty$, so there exists a unique $\hat{\rho}_n \geq s_n$ where this function takes the value $\xi_n^{2/3}$.
    Let
    \begin{equation*}
        \rho_n := \min\{ \tilde{\rho}_n, \hat{\rho}_n \}.
    \end{equation*}
    We claim that $\rho_n \gg s_n$. To prove this, it is enough to prove $\tilde{\rho}_n \gg s_n$ and $\hat{\rho}_n \gg s_n$. We have $\tilde{\rho}_n \gg s_n$, because
    \begin{equation*}
        \exp(-4C \left( \left( \frac{\tilde{\rho}_n}{s_n} \right)^2 - 1 \right)) = \frac{\frac12\exp(-4C(\tilde{\rho}_n)^2/s_n^2)}{\frac12 e^{-4C}} = \frac{\xi_n^{1/3}}{\frac12 e^{-4C}} \ll 1.
    \end{equation*}
    Similarly, we have $\hat{\rho}_n \gg s_n$, because
    \begin{equation*}
        \left( \frac{\hat{\rho}_n}{s_n} \right)^d = \frac{\xi_n^{2/3}}{\xi_n} \gg 1.
    \end{equation*}
    
    We will show that the sequences $(r_n)$ and $(\rho_n)$ together satisfy the conditions of Corollary \ref{EHI:c:j(r)Counterex}. Condition \eqref{EHI:e:RhoOnOrderOfR} is met, since $\rho_n \geq s_n \geq 2\sqrt{c_0} r_n$. All that remains is to show \eqref{EHI:e:Hypothesis} (that $1 \gg \mathbb{P}_0 \left( \tau_{B(0, \rho_n)} < T^{(r_n)} \right) \gg \P_0 (|\Delta X(T^{(r_n)})| \leq 4\rho_n+3r_n)$).
    Since $\rho_n \gg s_n \geq 2\sqrt{c_0} r_n$, by \eqref{EHI:e:PexitRhoExponential},
    \begin{equation*}
        1 \gg 2^{-\floor{(\rho_n+r_n)/(s_n+r_n)}} \geq \mathbb{P}_0 \left( \tau_{B(0, \rho_n)} < T^{(r_n)} \right).
    \end{equation*}
    By \eqref{EHI:e:PhoExitAtLeastSquareExponential} and the definitions of $\tilde{\rho}_n$, $\hat{\rho}_n$, and $\rho_n$,
    \begin{align*}
        \mathbb{P}_0 \left( \tau_{B(0, \rho_n)} < T^{(r_n)} \right) &\geq \mathbb{P}_0 \left( \tau_{B(0, \tilde{\rho}_n)} < T^{(r_n)} \right) \\
        &\geq \frac{1}{8} \exp(-\frac{4C (\tilde{\rho}_n)^2}{s_n^2}) = \xi_n^{1/3} \gg \xi_n^{2/3} = \frac{(\hat{\rho}_n)^d j(r_n)}{\lambda(r_n)} \geq \frac{\rho_n^d j(r_n)}{\lambda(r_n)}.
    \end{align*}
    Finally,
    \begin{align*}
        \P_0(|\Delta X(T^{(r_n)})| \leq 4\rho_n+3r_n) &= \frac{\int_{B(0, 4\rho_n+3r_n) \setminus B(0, r_n)} j(|x|) \, dx}{\int_{B(0, r_n)^c } j(|x|) \, dx} \leq \frac{|B(0, 1)| (4\rho_n + 3r_n)^d j(r_n)}{\lambda(r_n)} \\
        &\lesssim \frac{\rho_n^d j(r_n)}{\lambda(r_n)}.
    \end{align*}\
    Thus,
    \begin{equation*}
        1 \gg \mathbb{P}_0 \left( \tau_{B(0, \rho_n)} < T^{(r_n)} \right) \gg \frac{\rho_n^d j(r_n)}{\lambda(r_n)} \gtrsim \P_0(|\Delta X(T^{(r_n)})| \leq 4\rho_n+3r_n),
    \end{equation*}
    completing the proof.
\end{proof}

\subsection{Negative results for subordinate Brownian motions}

We now turn to subordinate Brownian motions. In this subsection, we prove the following proposition.

\begin{prop} \label{EHI:p:SbmCounterex}
    Let $X=(X_t)_{t \geq 0} = (B_{S_t})_{t \geq 0}$ be a subordinate Brownian motion. For all $s>0$, let $T^{\{s\}} := \inf\{ t : \Delta S(t) > s \}$. Suppose there exist sequences $(s_n)$ and $(r_n)$ of positive numbers such that
    \begin{itemize}
        \item $s_n \ll r_n^2$.
        \item $1 \gg \P_0(\tau_{B(0, r_n)} < T^{\{s_n\}}) \gg \P_0 (|\Delta X(T^{\{s_n\}})| \leq 36 r_n)$.
    \end{itemize}
    Then $X$ does not satisfy $\EHI$. Furthermore, if $r_n \to 0$ along a subsequence then $X$ does not satisfy $\EHI(r \leq 1)$, and if $r_n \to \infty$ along a subsequence then $X$ does not satisfy $\EHI(r \geq 1)$.
\end{prop}

We will use Proposition \ref{EHI:p:GeneralCounterex} to prove Proposition \ref{EHI:p:SbmCounterex}, but first we must prove a few lemmas.

\begin{lemma} \label{EHI:l:234}
    Let $Z \sim N(0, I_d)$ be a standard $d$-dimensional Gaussian. There exist positive constants $C$ and $c$ (which may depend on $d$) such that for all $r \geq 1$,
    \begin{equation*}
        \frac{\P(|Z| \geq 4r)}{\P(2r \leq |Z| \leq 3r)} \leq Ce^{-c r^2}.
    \end{equation*}
\end{lemma}

\begin{proof}
    For all $r \geq 1$, let $g(r) := \P(|Z| \geq 4r) / \P(2r \leq |Z| \leq 3r)$. The density of $Z$ is $(2\pi)^{-d/2}e^{-|x|^2/2}$, so by the change of variables $\{ x=ry; dx=r^d dy \}$, we have
    \begin{equation}\label{EHI:e:g(r)ChangeOfVariables}
        g(r) = \frac{\int_{\{ x \in \R^d : |x| > 4r \}} e^{-|x|^2/2} \, dx}{\int_{\{ x \in \R^d : 2r \leq |x| \leq 3r \}} e^{-|x|^2} \, dx} = \frac{\int_{\{y \in \R^d:|y| \geq 4\}} e^{-r^2 |y|^2 /2} \, dy}{\int_{\{ y \in \R^d: 2 \leq |y| \leq 3 \}} e^{-r^2 |y|^2 /2} \, dy}.
    \end{equation}
    For both the numerator and the denominator of the right-hand side of \eqref{EHI:e:g(r)ChangeOfVariables}, let us compare the integrand to what it would be for $r=1$. This comparison will give us an upper bound for the ratio $g(r)/g(1)$. For the numerator, note that for all $|y| \geq 4$, we have
    \begin{equation*}
        \frac{e^{-r^2|y|^2/2}}{e^{-|y|^2/2}} = e^{-(r^2-1)|y|^2/2} \leq e^{-8(r^2-1)},
    \end{equation*}
    or equivalently,
    \begin{equation}\label{EHI:e:g(r)Numerator}
        e^{-r^2|y|^2/2} \leq e^{-8(r^2-1)} e^{-|y|^2/2} \qquad\mbox{for all $r \geq 1, |y| \geq 4$}.
    \end{equation}
    For the denominator, note that for all $|y| \leq 3$, we have
    \begin{equation*}
        \frac{e^{-r^2|y|^2/2}}{e^{-|y|^2/2}} = e^{-(r^2-1)|y|^2/2} \geq e^{-\frac92 (r^2-1)},
    \end{equation*}
    or
    \begin{equation}\label{EHI:e:g(r)Denominator}
        e^{-r^2|y|^2/2} \geq e^{-\frac92 (r^2-1)} e^{-|y|^2/2} \qquad\mbox{for all $r \geq 1, |y| \leq 3$}.
    \end{equation}
    By applying the bounds \eqref{EHI:e:g(r)Numerator} and \eqref{EHI:e:g(r)Denominator} to the integrands from \eqref{EHI:e:g(r)ChangeOfVariables},
    \begin{equation*}
        g(r) \leq \frac{e^{-8(r^2-1)} \int_{\{ y : |y| \geq 4 \} } e^{-|y|^2/2} \, dy}{e^{-\frac{9}{2}(r^2-1)} \int_{\{ y : 2 \leq |y| \leq 3 \} } e^{-|y|^2/2} \, dy} = e^{-(8-\frac{9}{2})(r^2-1)} g(1) = g(1) e^{7/2} e^{-\frac72 r^2}.
    \end{equation*}
    This upper bound is of the form we desired, with $C=g(1)e^{7/2}$ and $c=7/2$.
\end{proof}

\begin{lemma} \label{EHI:l:SbmExitsNotTooFar}
    Let $X=(X_t)_{t \geq 0} = (B_{S_t})_{t \geq 0}$ be a subordinate Brownian motion on $\mathbb{R}^d$. Fix $r>0$ and $s \leq r^2/2$. Let $T^{\{s\}} := \inf\{ t>0 : \Delta S(t)>s \}$ be the first time that the subordinator $S$ takes a jump of size greater than $s$. Then for all $x \in B(0, r)$,
    \begin{equation*}
        \mathbb{P}_x \left( X_{\tau_{B(0, r)}} \notin B(0, 5 r) \Big| \tau_{B(0, r)} < T^{\{s\}} \right) \leq Ce^{-cr^2/s},
    \end{equation*}
    where $C, c>0$ are constants that depend only on $d$.
\end{lemma}

\begin{proof}
    Recall from the definition of the SBM decomposition (Definition \ref{EHI:d:sbmDecomposition}) that $S^{\brackets{s}}=\left(S^{\brackets{s}}_t\right)_{t \geq 0}$ is the process that results from removing all the jumps of size greater than $s$ from $S$. That is,
    \begin{equation*}
        S^{\brackets{s}}_t := S_t - \sum_{0 <t' \leq t : \Delta S(t')>s} \Delta S(t').
    \end{equation*}
    Moreover, $X^{\brackets{s}} = (X^{\brackets{s}}_t)_{t \geq 0} := (B_{S^{\brackets{s}}_t})_{t \geq 0}$, and $j^{\brackets{s}}$ is the jump kernel of $X^{\brackets{s}}$, which can be expressed as
    \begin{equation} \label{EHI:e:JtildeFormula}
        j^{\brackets{s}}(r) = \int_{(0, s]} (4 \pi u)^{-d/2} \exp(-\frac{r^2}{4u}) \, \mu(du),
    \end{equation}
    where $\mu$ is the L\'{e}vy measure of $S$. Let us also use the notation $j^{\brackets{s}}(x, y) := j^{\brackets{s}}(|x-y|)$ and $j^{\brackets{s}}(x, E) = \int_E j^{\brackets{s}}(x, y) \, dy$. For all measurable sets $E$, let $\tau^{\brackets{s}}_E := $

    Let $\mathcal{A} := \mathbb{P}_x \left( X_{\tau_{B(0, r)}} \notin B(0, A r) \Big| \tau_{B(0, r)} < T^{\{s\}} \right)$, the quantity that we are interested in upper-bounding. Note that $S$ and $S^{\brackets{s}}$ agree up until time $T^{\{s\}}$, so an equivalent expression for $\mathcal{A}$ is
    \begin{equation} \label{EHI:e:CondProbMathcalAsRatio}
        \mathcal{A} = \P_x \left( X^{\brackets{s}}_{\tau^{\brackets{s}}_{B(0, r)}} \notin B(0, 5r) \Big| \tau^{\brackets{s}}_{B(0, r)} < T^{\{s\}} \right) = \frac{\P_x \left( \tau^{\brackets{s}}_{B(0, r)} < T^{\{s\}}, X^{\brackets{s}}_{\tau^{\brackets{s}}_{B(0, r)}} \notin B(0, 5r) \right)}{\P_x \left( \tau^{\brackets{s}}_{B(0, r)} < T^{\{s\}} \right)}.
    \end{equation}
    For all measurable $E \subseteq B(0, r)^c$, by the L\'{e}vy system formula (cf. \cite{BL}, \cite{ck1}),
    \begin{align*}
        \P_x \left( \tau^{\brackets{s}}_{B(0, r)} < T^{\{s\}}, X^{\brackets{s}}_{\tau^{\brackets{s}}_{B(0, r)}} \in E \right) = \E_x \left[ \int_0^{\tau^{\brackets{s}}_{B(0, r)} \wedge T^{\{s\}}} j^{\brackets{s}}(X^{\brackets{s}}_t, E) \, dt \right].
    \end{align*}
    Thus, another way of writing \eqref{EHI:e:CondProbMathcalAsRatio} is
    \begin{align} \label{EHI:e:CondProbMathcalAsJratio}
        \mathcal{A} &= \frac{\P_x \left( \tau^{\brackets{s}}_{B(0, r)} < T^{\{s\}}, X^{\brackets{s}}_{\tau^{\brackets{s}}_{B(0, r)}} \in B(0, 5r)^c \right)}{\P_x \left( \tau^{\brackets{s}}_{B(0, r)} < T^{\{s\}}, X^{\brackets{s}}_{\tau^{\brackets{s}}_{B(0, r)}} \in B(0, r)^c \right)} \notag\\
        &= \frac{\E_x \left[ \int_0^{\tau^{\brackets{s}}_{B(0, r)} \wedge T^{\{s\}}} j^{\brackets{s}}(X^{\brackets{s}}_t, B(0, 5r)^c) \, dt \right]}{\E_x \left[ \int_0^{\tau^{\brackets{s}}_{B(0, r)} \wedge T^{\{s\}}} j^{\brackets{s}}(X^{\brackets{s}}_t, B(0, r)^c) \, dt \right]} \leq \sup_{z \in B(0, r)} \frac{j^{\brackets{s}}(z, B(0, 5r)^c)}{j^{\brackets{s}}(z, B(0, r)^c)}.
    \end{align}
    All that remains is to show that for all $z \in B(0, r)$, we have $j^{\brackets{s}}(z, B(0, 5r)^c) / j^{\brackets{s}}(z, B(0, r)^c) \leq Ce^{-cr^2/s}$. Fix $z \in B(0, r)$. Note that by the triangle inequality, we have $B(0, 5r)^c \subseteq B(z, 4r)^c$ and $B(0, r)^c \supseteq B(z, 2r)^c \supseteq B(z, 3r) \setminus B(0, 2r)$, so
    \begin{equation*}
        \frac{j^{\brackets{s}}(z, B(0, 5r)^c)}{j^{\brackets{s}}(z, B(0, r)^c)} \leq \frac{j^{\brackets{s}}(z, B(z, 4r)^c)}{j^{\brackets{s}}(z, B(z, 3r) \setminus B(z, 2r))} = \frac{j^{\brackets{s}}(0, B(0, 4r)^c)}{j^{\brackets{s}}(0, B(0, 3r) \setminus B(0, 2r))}.
    \end{equation*}
    Since this holds for all $z \in B(0, r)$, \eqref{EHI:e:CondProbMathcalAsJratio} implies
    \begin{equation}\label{EHI:e:JzIntoJ0}
        \mathcal{A} \leq \frac{j^{\brackets{s}}(0, B(0, 4r)^c)}{j^{\brackets{s}}(0, B(0, 3r) \setminus B(0, 2r))}.
    \end{equation}
    Let us apply \eqref{EHI:e:JtildeFormula} to both the numerator and denominator of the right-hand side of \eqref{EHI:e:JzIntoJ0}, to compare their ratio. Let $Z \sim N(0, I_d)$ be a standard $d$-dimensional Gaussian. Note that for all $u>0$, the density of $\sqrt{2u}Z$ is $(4\pi u)^{-d/2} \exp(-\frac{|x|^2}{4u})$
    Then by \eqref{EHI:e:JtildeFormula} and Tonelli's theorem,
    \begin{align}
        j^{\brackets{s}}(0, B(0, 4r)^c) &= \int_{B(0, 4r)^c} j^{\brackets{s}}(|x|) \, dx = \int_{B(0, 4r)^c} \int_{(0, s]} (4\pi u)^{-d/2} \exp(-\frac{|x|^2}{4u}) \, \mu(du) \, dx \notag\\
        &= \int_{(0, s]} \int_{B(0, 4r)^c} (4\pi u)^{-d/2} \exp(-\frac{|x|^2}{4u}) \, dx \, \mu(du) \notag\\
        &= \int_{(0, s]} \P \left( \sqrt{2u} Z \in B(0, 4r)^c \right) \, \mu(du) = \int_{(0, s]} \P \left( |Z| \geq \frac{4r}{\sqrt{2u}} \right) \, \mu(du). \label{EHI:e:JtildeUsingTonelli1}
    \end{align}
    By a similar calculation,
    \begin{equation}\label{EHI:e:JtildeUsingTonelli2}
        j^{\brackets{s}}(0, B(0, 3r) \setminus B(0, 2r)) = \int_{(0, s]} \P \left( \frac{2r}{\sqrt{2u}} \leq |Z| \leq \frac{3r}{\sqrt{2u}} \right) \, \mu(du).
    \end{equation}
    Let $C$ and $c$ be the constants from Lemma \ref{EHI:l:234}.
    By plugging \eqref{EHI:e:JtildeUsingTonelli1} and \eqref{EHI:e:JtildeUsingTonelli2} into \eqref{EHI:e:JzIntoJ0}, and then applying Lemma \ref{EHI:l:234},
    \begin{align*}
        \mathcal{A} &\leq \frac{\int_{(0, s]} \P \left( |Z| \geq \frac{4r}{\sqrt{2u}} \right) \, \mu(du)}{\int_{(0, s]} \P \left( \frac{2r}{\sqrt{2u}} \leq |Z| \leq \frac{3r}{\sqrt{2u}} \right) \, \mu(du)} \leq \sup_{u \in (0, s]} \frac{\P \left( |Z| \geq \frac{4r}{\sqrt{2u}} \right)}{\P \left( \frac{2r}{\sqrt{2u}} \leq |Z| \leq \frac{3r}{\sqrt{2u}} \right)} \leq C e^{-cr^2/(2s)},
    \end{align*}
    completing the proof. (Note that the $c$ in the statement of this lemma is half of the $c$ from Lemma \ref{EHI:l:234}. Also note that the assumption that $s \leq r^2/2$ was necessary so that we would have $\frac{r}{\sqrt{2u}} \geq 1$ for all $u \in (0, s]$, which allows us to apply Lemma \ref{EHI:l:234}.)
\end{proof}

\begin{proof}[Proof of Proposition \ref{EHI:p:SbmCounterex}]
    Let $X=(X_t)_{t \geq 0} = (B_{S_t})_{t \geq 0}$ be a subordinate Brownian motion, and suppose there exist sequences $(s_n), (r_n) \subseteq (0, \infty)$ such that $s_n \ll r_n^2$ and $1 \gg \P_0(\tau_{B(0, r_n)} < T^{\{s_n\}}) \gg \P_0 (|\Delta X(T^{\{s_n\}})| \leq 36 r_n)$. We would like apply Proposition \ref{EHI:p:GeneralCounterex} with
    \begin{align*}
        R_{1, n} &:= r_n, \\
        R_{2, n} &:= 5r_n, \\
        R_{1, n}+R_{2, n} &:= 6r_n, \\
        R_{3, n} &:= 30r_n, \\
        T_n &:= T^{\{s_n\}}.
    \end{align*}
    We must verify conditions \eqref{EHI:e:GeneralCounterexR2R1}-\eqref{EHI:e:GeneralCounterexLeaving1+2}. Condition \eqref{EHI:e:GeneralCounterexR2R1} clearly holds, with $C=5$. Condition \eqref{EHI:e:GeneralCounterexSandwich} holds, since we assumed $1 \gg \P_0(\tau_{B(0, r_n)} < T^{\{s_n\}}) \gg \P_0 (|\Delta X(T^{\{s_n\}})| \leq 36 r_n)$. By Lemma \ref{EHI:l:SbmExitsNotTooFar},
    \begin{equation*}
        \sup_{x \in B(0, r_n)} \P_x \left( \tau_{B(0, r_n)} \notin B(0, 5r_n) \Big| \tau_{B(0, r_n)} < T^{\{s_n\}} \right) \leq Ce^{-cr_n^2/s_n} \to 0
    \end{equation*}
    and
    \begin{equation*}
        \sup_{x \in B(0, 6r_n)} \P_x \left( \tau_{B(0, 6r_n)} \notin B(0, 30r_n) \Big| \tau_{B(0, 6r_n)} < T^{\{s_n\}} \right) \leq Ce^{-c(6r_n)^2/s_n} \to 0,
    \end{equation*}
    so conditions \eqref{EHI:e:GeneralCounterexLeaving1} and \eqref{EHI:e:GeneralCounterexLeaving1+2} hold, completing the proof.
\end{proof}

\subsection{A highly regular counterexample}\label{EHI:ss:Counterexample}

In this section, we construct a process $X$ that has many regularity properties, and yet does not satisfy $\EHI$. This process is a subordinate Brownian motion, and has the property that $j(2r) \asymp j(r)$ for all $r>0$. Better yet, the L\'{e}vy measure $\mu$ of its subordinator $S$ is absolutely continuous with respect to the Lebesgue measure on $(0, \infty)$, and can be written as $\mu(du) = f(u) \, du$, where $f : (0, \infty) \to (0, \infty)$ is a continuous, decreasing function, and $f(2u) \asymp f(u)$ for all $u>0$. (From this, it follows easily that $j(2r) \asymp j(r)$.) In this sense, the process feels very far from pathological. As it turns out, its unexpected behavior results from $f$ being too flat (rather than too fast-decaying) on certain intervals (in fact, it is close to constant on some intervals).

Before we introduce this highly regular counterexample, let us speak about a distribution that will feature prominently in its construction. Let $\alpha=3$. Let $Y$ be a positive random variable with density
\begin{equation}\label{EHI:e:Ydef}
    f_Y(y) := \left\{
    \begin{matrix}
        \frac{\alpha-1}{\alpha} &:& \mbox{if $0 \leq y \leq 1$}\\
        \frac{\alpha-1}{\alpha}y^{-\alpha} &:& \mbox{if $y \geq 1$}.
    \end{matrix} \right.
\end{equation}
In other words, the density of $Y$ is continuous, non-increasing, constant on $[0, 1]$, and decays like $y^{-\alpha}$ on $[1, \infty)$.
For all $y \in [0, 1]$,
\begin{equation}\label{EHI:e:ProbYsmall}
    \P(Y<y) = \frac{\alpha-1}{\alpha}y.
\end{equation}
For all $y \geq 1$, 
\begin{equation}\label{EHI:e:ProbYbig}
    \P(Y>y) = \frac{1}{\alpha} y^{-(\alpha-1)}.
\end{equation}

\begin{lemma}\label{EHI:l:YZlemma}
    Let $Y$ have distribution \eqref{EHI:e:Ydef}, and let $Z$ be a standard Gaussian in $\mathbb{R}^d$, independent of $Y$. Then there exists a constant $C$ such that for all $a \in (0, 1)$,
    \begin{equation*}
        \P(Y|Z| \leq a) \leq Ca^{\frac{2d}{d+2}}.
    \end{equation*}
\end{lemma}

\begin{proof}
    In order to have $\sqrt{Y}|Z| \le a$, we must have either $\sqrt{Y} \leq a^{\frac{d}{d+2}}$ or $|Z| \leq a^{\frac{2}{d+2}}$. Therefore,
    \begin{align*}
        \P(\sqrt{Y}|Z| \leq a) &\leq \P(Y \leq a^{\frac{2d}{d+2}}) + \P(|Z| \leq a^{\frac{2}{d+2}}) \\
        &= \frac{\alpha-1}{\alpha} a^{\frac{2d}{d+2}} + \int_{B(0, a^{\frac{2}{d+2}})} (2\pi)^{-d/2} \exp(-\frac{|x|^2}{2}) \, dx \\
        & \leq \left( \frac{\alpha-1}{\alpha} + (2\pi)^{-d/2} |B(0, 1)| \right) a^{\frac{2d}{d+2}}.
    \end{align*}
\end{proof}

For all $n \in \mathbb{N}$, let $H_n := 2^{n^2}$ and $A_n = 2^{-2n^2}$. The description of the subordinator $S$ is as follows. For each $n \in \mathbb{N}$, there is a Poisson clock which rings with rate $H_n$. Whenever the $n$th clock rings, $S$ takes a jump forward with size equal in distribution to $A_n Y$ (where $Y$ is as described in \eqref{EHI:e:Ydef}). We call this a ``type-$n$ jump." More formally, for all $n$, let $N^{(n)} = (N^{(n)}_t)_{t\geq0}$ be a Poisson point process with rate $H_n$. Let the collection $(N^{(n)})_{n \in \mathbb{N}}$ be independent. Then let $(Y_{n, i})_{n \in \mathbb{N}, i \in \mathbb{N}}$ be an iid collection of copies of $Y$.
For all $t\geq0$, let
\begin{equation*}
    S_t := \sum_{n=1}^\infty \sum_{i=1}^{N^{(n)}_t} A_n Y_{n, i}.
\end{equation*}
Note that $N^{(n)}_t$ is the number of type-$n$ jumps that have occurred by time $t$.
As usual, let $B$ be a Brownian motion on $\mathbb{R}^d$ with twice the standard speed, and let $X_t := B(S_t)$.

In addition to the probabilistic description above, we can describe $S$ in terms of its L\'{e}vy measure:
\begin{equation*}
    \mu(\mathrm{d}t) = \frac{\alpha-1}{\alpha}\sum_{n=1}^\infty H_n A_n^{-1} \min\{1, (t/A_n)^{-3} \} \dee{t} = \frac{2}{3} \sum_{n=1}^\infty 2^{3n^2} \min\left\{1, \left(2^{2n^2} t \right)^{-3} \right\} \dee{t}.
\end{equation*}

\begin{theorem} \label{EHI:t:HighlyRegularCounterex}
    If $X=(X_t)_{t \geq 0} = (B(S_t))_{t\geq0}$ is the process constructed above, then $X$ does not satisfy $\EHI$.
\end{theorem}

We will usually write $\alpha$, $H_n$, and $A_n$ (rather than their exact values $3$, $2^{n^2}$, and $2^{-2n^2}$) to illustrate that the idea behind this proof is easily adaptable to different instantiations of these parameters. Generally, $\alpha$ must be a sufficiently large exponent, $(H_n)$ must be a sequence that grows faster than exponentially (with $H_n \ll H_{n+1}$), and $(A_n)$ must be a sequence that decays much faster than $H_n$ grows (with $H_n A_n \gg H_{n+1} A_{n+1}$). Throughout the proof, when we use notation like $o(1)$, $\lesssim$, or $\ll$, we mean as $n \to \infty$.

The key idea is that $S$ usually takes jumps on the order of $A_n$ for some $n$, which corresponds to $X$ taking jumps with magnitude on the order of $\sqrt{A_n}$. At a scale between $\sqrt{A_{n+1}}$ and $\sqrt{A_n}$, the type-$m$ jumps for $m>n$ are all very small, so the process looks like a Brownian motion until the first jump of type $n$ or less. (The exponent $\alpha$ being larger than $2$ is enough to achieve this.) We construct sequences $(s_n)$ and $(r_n)$ such that $A_{n+1} \ll s_n \ll r_n^2 \ll A_n$. Then the stopping time $T^{\{s_n\}}$ (the first time $S$ takes a jump of size greater than $s_n$) will usually be the time of the first jump of type $n$ or less. At this time, because $A_n$ is so much larger than $r_n^2$, the process will usually jump somewhere very far away. The exact choices of $(s_n)$ and $(r_n)$ are reverse-engineered so that the conditions of Proposition \ref{EHI:p:SbmCounterex} will be satisfied.

\begin{proof}[Proof of Theorem \ref{EHI:t:HighlyRegularCounterex}]
    We will apply Proposition \ref{EHI:p:SbmCounterex}, with
    \begin{align*}
        s_n &:= \frac{H_{n+1}}{H_n} A_{n+1} = 2^{-(2n^2+2n+1)}, \\
    r_n^2 &:= \frac{H_{n+1}}{H_n} A_{n+1} \log(n) = 2^{-(2n^2+2n+1)} \log n.
    \end{align*}
    Then $A_{n+1} \ll s_n \ll r_n^2 \ll A_n$. We must show that $1 \gg \P_0(\tau_{B(0, r_n)} < T^{\{s_n\}}) \gg \P_0(|\Delta X(T^{\{s_n\}})| \leq 36 r_n)$.

    Let $\lambda_{\{s_n\}}$ denote the rate at which jumps such that $\Delta S(t) > s_n$ occur. Since $A_{n+1} \ll s_n \ll A_n$, by \eqref{EHI:e:ProbYsmall} and \eqref{EHI:e:ProbYbig},
    \begin{equation*}
        \lambda_{\{s_n\}} = \sum_{m=1}^\infty H_m \P(A_m Y > s_n) = (1-o(1))\sum_{m=1}^n H_m + \sum_{m=n+1}^\infty \frac{1}{\alpha} H_m \left( \frac{A_m}{s_n} \right)^{\alpha-1}.
    \end{equation*}
    One may verify that the dominant term of this series is $H_n$, so
    \begin{equation}\label{EHI:e:LambdasnRegularCounterex}
        \lambda_{\{s_n\}} = (1\pm o(1)) H_n.
    \end{equation}

    Next, let us consider the rate at which jumps such that $\frac{A_{n+1}}{2} \leq \Delta S(t) \leq s_n$ occur. Let $\lambda^*_n$ denote this rate. By putting a lower bound on $\lambda^*_n$, we will put a lower bound on the probability that enough of these jumps occur before time $T^{\{s_n\}}$ that $S_t$ is already on the order of $r_n^2$. Note that more than half of the type-$(n+1)$ jumps are of this size (at least for sufficiently large $n$), since by \eqref{EHI:e:ProbYsmall} and \eqref{EHI:e:ProbYbig},
    \begin{equation*}
        \P \left( \frac{A_{n+1}}{2} \leq A_{n+1}Y \leq s_n \right) = 1 - \P \left( Y \leq \frac12 \right) - \P \left( Y \geq \frac{s_n}{A_{n-1}} \right) = 1-\frac{\alpha-1}{\alpha} \cdot\frac12 - o(1) > \frac12.
    \end{equation*}
    Therefore, for sufficiently large $n$,
    \begin{equation} \label{EHI:e:Lambdastarn}
        \lambda^*_n \geq \frac12 H_{n+1}.
    \end{equation}
    Let $c$ be a small constant. The probability of getting at least $c\frac{r_n^2}{A_{n+1}}$ jumps with $\frac{A_{n+1}}{2} \leq s_n \leq s_n$ before time $T^{\{s_n\}}$ is equal to
    \begin{align*}
        \left( \frac{\lambda^*_n}{\lambda^*_n + \lambda_{\{s_n\}}} \right)^{\ceil{c \frac{r_n^2}{A_{n+1}}}} &= \left(1 + \frac{\lambda_{\{s_n\}}}{\lambda^*_n} \right)^{-\ceil{c \frac{r_n^2}{A_{n+1}}}} \\
        &\geq \exp(- \frac{\lambda_{\{s_n\}}}{\lambda^*_n} \ceil{ c \frac{r_n^2}{A_{n+1}}}).
    \end{align*}
    Since $r_n^2 \gg A_{n+1}$, $\lambda_{\{ s_n \} } = (1+o(1)) H_n$, and $\lambda^*_n \geq \frac12 H_{n+1}$, we can easily choose a $c$ so that
    \begin{equation*}
        \frac{\lambda_{\{s_n\}}}{\lambda^*_n} \ceil{ c \frac{r_n^2}{A_{n+1}}} \leq \frac{H_n}{H_{n+1}} \cdot \frac{r_n^2}{A_{n+1}}.
    \end{equation*}
    Therefore, with probability at least $\exp(-(H_n / H_{n+1}) (r_n^2 / A_{n+1}))$, the subordinator $S$ takes at least $cr_n^2/A_{n+1}$ jumps of size at least $A_{n+1}/2$ before time $T^{ \{ s_n \} }$. If this happens, then we already have $S_t \geq \frac{c}{2} r_n^2$ before time $T^{\{s_n \}}$, so the conditional probability that $|X_t| \geq r_n$ for some $t<T^{\{s_n\}}$ is on the order of $1$. Therefore,
    \begin{equation} \label{EHI:e:RegularCounterexLowerbound}
        \P_0 (\tau_{B(0, r_n)} < T^{\{s_n\}}) \gtrsim \exp(-\frac{H_n}{H_{n+1}} \cdot \frac{r_n^2}{A_{n+1}}) = \frac{1}{n}.
    \end{equation}

    We would like to put an upper bound on $\P_0 (|\Delta X(T^{\{s_n\}})| \leq 36 r_n)$ that decays much faster than $1/n$, to show that $\P_0 (\tau_{B(0, r_n)} < T^{\{s_n\}}) \gg \P_0 (|\Delta X(T^{\{s_n\}})| \leq 36 r_n)$. There are two ways we could have $|\Delta X(T^{\{s_n\}})| \leq 36 r_n$. First, $T^{\{s_n\}}$ might (surprisingly) be a type-$m$ jump for some $m>n$. Second, $T^{\{s_n\}}$ could be type $n$ or less, but the Gaussian factor in the displacement of $X$ resulting from this jump might be surprisingly small. We will put upper bounds on the probabilities each of these occurrences.

    The rate at which jumps in $S$ of size greater than $s_n$ due to type-$m$ jumps for $m>n$ occur is $\sum_{m=n+1}^\infty H_m \P(A_m Y > s_n)$. The rate at which jumps in $S$ with size greater than $s_n$ due to type-$n$ jumps is $(1-o(1)) H_n$. Therefore,
    \begin{align*}
        \P_0 \left( \mbox{$T^{\{s_n\}}$ is type-$m$ for some $m<n$} \right) &\leq \frac{\sum_{m=n+1}^\infty H_m \frac{1}{\alpha} \left( \frac{A_m}{s_n} \right)^{\alpha-1}}{(1-o(1)) H_n} \\
        &\asymp \frac{H_{n+1}}{H_n} \left( \frac{A_{n+1}}{s_n} \right)^{\alpha-1} = \left( \frac{H_n}{H_{n+1}} \right)^{\alpha-2}.
    \end{align*}
    (We obtained the $\asymp$ in the above equation by considering the dominant term in the series.)
    On the other hand, if $T^{\{ s_n \}}$ is type $m$ for some $m \leq n$, then the displacement of $X$ from this jump is equal in distribution to $\sqrt{2 A_m Y} Z$, where $Z \sim N(0, I_d)$ is independent of $Y$, so by Lemma \ref{EHI:l:YZlemma},
    \begin{align*}
        \P_0 \left( |\Delta X(T^{\{s_n\}})| \Big| \mbox{$T^{\{s_n\}}$ is type-$m$ for some $m \geq n$} \right) &\leq \P \left( \sqrt{2 A_n Y} |Z| \leq 36 r_n \Big| Y \geq s_n \right) \\
        &\leq \P \left( \sqrt{2 A_n Y} |Z| \leq 36 r_n \right) \\
        &\lesssim \left( \frac{r_n^2}{A_n} \right)^{\frac{2d}{d+2}}.
    \end{align*}
    Thus,
    \begin{equation} \label{EHI:e:RegularCounterexScattering}
        \P_0 \left( |\Delta X(T^{\{s_n\}})| \leq 36 r_n \right) \lesssim \left( \frac{H_n}{H_{n+1}} \right)^{\alpha-2} + \left( \frac{r_n^2}{A_n} \right)^{\frac{2d}{d+2}} = 2^{-(2n+1)(\alpha-2)} + \left( 2^{-(2n+1)} \log n \right)^{\frac{2d}{d+2}} \ll \frac{1}{n}.
    \end{equation}

    We must still show that $\P_0 (\tau_{B(0, r_n)} < T^{\{s_n\}}) \ll 1$. Consider the truncated processes $\tilde{S}^{\{s_n\}}$ and $\tilde{X}^{\{s_n \}}$, which do not include the jumps of size greater than $S_n$ in $S$.
    We use the notation $\tilde{\tau}^{\{s_n\}}_U$ to denote exit times for $\tilde{X}^{\{s_n\}}$
    By Lemma \ref{EHI:l:FactorOf2ForExits}, Markov's inequality, and conditioning on $\tilde{S}^{\{ s_n \}}_{T^{\{ s_n \} }}$,
    \begin{align*}
        \P_0 \left( \tau_{B(0, r_n)} < T^{\{s_n\}} \right) &= \P_0 \left( \tilde{\tau}^{\{s_n\}}_{B(0, r_n)} < T^{\{s_n\}} \right) \leq 2 \P_0 \left( \left| \tilde{X}^{\{s_n\}}_{T^{\{ s_n \}}} \right| \geq r_n \right) \leq 2r_n^{-2} \E_0 \left[ \left| \tilde{X}^{\{s_n\}}_{T^{\{ s_n \}}} \right|^2 \right] \notag\\
        & \asymp r_n^{-2} \E_0 \left[ \tilde{S}^{\{s_n\}}_{T^{\{ s_n \}}} \right] = r_n^{-2} \frac{\E_0 \left[ \tilde{S}^{\{s_n\}}_{1}  \right]}{\lambda_{\{s_n\}}}. \notag\\
    \end{align*}
    To compute $\E_0 \left[ \tilde{S}^{\{s_n\}}_{1}  \right]$,
    \begin{align*}
        \E_0 \left[ \tilde{S}^{\{s_n\}}_{1}  \right] &= \sum_{m=1}^\infty H_m \E[A_m Y \cdot \indicatorWithSetBrackets{A_m Y \leq s_n}] \\
        &= \sum_{m=1}^\infty H_m A_m \E \left[Y \indicatorWithSetBrackets{Y \leq \frac{s_n}{A_m}} \right].
    \end{align*}
    For $m > n$, let us just replace the indicator with $1$ for an upper bound. For $m \leq n$, let us replace $Y \indicatorWithSetBrackets{Y \leq s_n/A_m}$ with $(s_n / A_m ) \indicatorWithSetBrackets{Y \leq s_n/A_m}$. The result is
    \begin{align*}
        \E_0 \left[ \tilde{S}^{\{s_n\}}_{1}  \right] &\leq \sum_{m=1}^n H_m s_n \P_0 \left( Y \leq \frac{s_n}{A_m} \right) + \sum_{m=n+1}^\infty H_m A_m \E[Y] \\
        &\lesssim s_n \sum_{m=1}^n H_m + \sum_{m=n+1}^\infty H_m A_m.
    \end{align*}
    This series has two dominant terms, corresponding to $m=n$ and $m=n+1$, which are both equal to $H_{n+1} A_{n+1}$. Therefore,
    \begin{equation} \label{EHI:e:RegularCounterex<<1}
        \P_0 \left( \tau_{B(0, r_n)} < T^{\{s_n\}} \right) \leq \frac{H_{n+1} A_{n+1}}{r_n^2 \lambda_{\{ s_n \}}} \asymp \frac{H_{n+1} A_{n+1}}{r_n^2 H_n} = \frac{s_n}{r_n^2} = \frac{1}{\log n } \ll 1.
    \end{equation}
    By \eqref{EHI:e:RegularCounterex<<1}, \eqref{EHI:e:RegularCounterexLowerbound}, and \eqref{EHI:e:RegularCounterexScattering},
    \begin{equation*}
        1 \gg \frac{1}{\log n} \gtrsim \P_0 \left( \tau_{B(0, r_n)} < T^{\{s_n\}} \right) \gtrsim \frac{1}{n} \gg \P_0 \left( |\Delta X(T^{\{s_n\}})| \leq 36 r_n \right),
    \end{equation*}
    as desired.
\end{proof}

\bibliographystyle{alpha}
\bibliography{biblio} 

@article{Rai,
author = {Rai\v{c}, Martin},
year = {2019},
month = {09},
pages = {2824-2853},
title = {A multivariate {B}erry-{E}sseen theorem with explicit constants},
volume = {25},
journal = {Bernoulli},
doi = {10.3150/18-BEJ1072}
}

@article{KM,
  author  = {Kim, Panki and Mimica, Ante},
  title   = {{H}arnack inequalities for subordinate {B}rownian motions},
  journal = {Electron. J. Probab.},
  volume  = {17},
  pages   = {1--23},
  year    = {2012},
  doi     = {10.1214/EJP.v17-1930}
}

@article{G,
  author  = {Grzywny, Tomasz},
  title   = {On {H}arnack inequality and {H}\"{o}lder regularity for isotropic unimodal {L}\'{e}vy processes},
  journal = {Potential Anal.},
  volume  = {41},
  number  = {1},
  pages   = {1--29},
  year    = {2014},
  doi     = {10.1007/s11118-013-9355-7}
}

@article{CkwElliptic,
  author = {Chen, Zhen-Qing and Kumagai, Takashi and Wang, Jinghai},
  title = {Elliptic {H}arnack inequalities for symmetric non-local {D}irichlet forms},
  journal = {J. Math. Pures Appl.},
  volume = {125},
  number = {9},
  year = {2019},
  pages = {1--42},
  doi = {10.1016/j.matpur.2018.06.003}
}

@article{ckw,
  author = {Chen, Zhen-Qing and Kumagai, Takashi and Wang, Jinghai},
  title = {Stability of heat kernel estimates for symmetric jump processes on metric measure spaces},
  journal = {Memoirs Amer. Math. Soc.},
  volume = {289},
  number = {1417},
  year = {2023},
  doi = {10.1090/memo/1417}
}

@article{ckw2,
  author = {Chen, Zhen-Qing and Kumagai, Takashi and Wang, Jinghai},
  title = {Stability of parabolic {H}arnack inequalities for symmetric non-local Dirichlet forms},
  journal = {J. Eur. Math. Soc.},
  volume = {24},
  number = {10},
  pages = {3545--3615},
  year = {2022},
  doi = {10.4171/JEMS/1229}
}

@book{FOT,
  author = {Fukushima, Masatoshi and Oshima, Yoichi and Takeda, Masayoshi},
  title = {{D}irichlet Forms and Symmetric {M}arkov Processes},
  edition = {2nd, revised and extended},
  series = {De Gruyter Studies in Mathematics},
  volume = {19},
  publisher = {Walter de Gruyter \& Co.},
  address = {Berlin},
  year = {2011}
}

@article{ck1,
  author = {Chen, Zhen-Qing and Kumagai, Takashi},
  title = {Heat kernel estimates for stable-like processes on {$d$}-sets},
  journal = {Stochastic Process. Appl.},
  volume = {108},
  year = {2003},
  pages = {27--62},
  doi = {10.1016/S0304-4149(03)00022-7}
}

@article{hadamard,
  author = {Hadamard, Jacques},
  title = {Extension {\`a} l'\'{e}quation de la chaleur d'un th\'{e}or\`{e}me de {A}. {H}arnack},
  journal = {Rendiconti del Circolo Matematico di Palermo. Serie II},
  volume = {3},
  year = {1955},
  pages = {337--346},
  note = {Originally presented in 1954}
}

@book{harnack,
  author = {Harnack, Carl Gustav Axel},
  title = {Die Grundlagen der Theorie des logarithmischen Potentiales und der eindeutigen Potentialfunktion in der Ebene},
  publisher = {Teubner},
  address = {Leipzig, Germany},
  year = {1887},
  note = {See also The Cornell Library Historical Mathematics Monographs}
}

@article{kassman,
  author = {Kassmann, Moritz},
  title = {{Ha}rnack inequalities: An introduction},
  journal = {Boundary Value Problems},
  volume = {2007},
  year = {2007},
  pages = {81415},
  note = {21 pages},
  doi = {10.1155/2007/81415}
}

@article{ms3,
  author = {Murugan, Mathav and Saloff-Coste, Laurent},
  title = {Heat kernel estimates for anomalous heavy-tailed random walks},
  journal = {Ann. Inst. Henri Poincar\'e Probab. Stat.},
  volume = {55},
  number = {2},
  year = {2019},
  pages = {697--719},
  doi = {10.1214/18-AIHP894}
}

@article{pini,
  author = {Pini, {B}.},
  title = {Sulla soluzione generalizzata di {W}iener per il primo problema di valori al contorno nel caso parabolico},
  journal = {Rendiconti del Seminario Matematico della Universit\`a di Padova},
  volume = {23},
  year = {1954},
  pages = {422--434}
}

@article{W,
  author = {Watanabe, Takashi},
  title = {The {I}soperimetric {I}nequality for {I}sotropic {U}nimodal {L}\'{e}vy {P}rocesses},
  journal = {Z. Wahrsch. Verw. Gebiete},
  volume = {63},
  year = {1983},
  pages = {487--499},
  doi = {10.1007/BF00531835}
}

@article{SSV,
  author  = {{\v{S}}iki\'{c}, Hrvoje and Song, Renming and Vondra\v{c}ek, Zoran},
  title   = {Potential theory of geometric stable processes},
  journal = {Probab. Theory Relat. Fields},
  volume  = {135},
  number  = {4},
  pages   = {547--575},
  year    = {2006},
  doi     = {10.1007/s00440-005-0470-3}
}

@article{IW,
  author  = {Ikeda, Nobuyuki and Watanabe, Shinzo},
  title   = {On some relations between the harmonic measure and the {L}\'evy measure for a certain class of {M}arkov processes},
  journal = {J. Math. Kyoto Univ.},
  volume  = {2},
  year    = {1962},
  pages   = {79--95},
  note    = {MR0142153}
}

@article{BL,
  author  = {Bass, Richard F. and Levin, David A.},
  title   = {Harnack inequalities for jump processes},
  journal = {Potential Anal.},
  volume  = {17},
  year    = {2002},
  pages   = {375--388},
  doi     = {10.1023/A:1020382320612}
}

@article{meyer,
  author  = {Meyer, Paul-Andr\'e},
  title   = {Renaissance, recollements, m\'elanges, ralentissement de processus de {M}arkov},
  journal = {Ann. Inst. Fourier},
  volume  = {25},
  number  = {3--4},
  year    = {1975},
  pages   = {464--497},
  note    = {MR0415784}
}

@article{chen-kumagai-wang2020,
  author  = {Chen, Xia and Kumagai, Takashi and Wang, Jinghai},
  title   = {Random conductance models with stable-like jumps: heat kernel estimates and {H}arnack inequalities},
  journal = {J. Funct. Anal.},
  volume  = {279},
  number  = {7},
  year    = {2020},
  pages   = {108656},
  doi     = {10.1016/j.jfa.2020.108656}
}

@article{Malm,
  author    = {Malmquist, Jens},
  title     = {Stability Results for Symmetric Jump Processes on Metric Measure Spaces with Atoms},
  journal   = {Potential Analysis},
  volume    = {59},
  pages     = {167--235},
  year      = {2023},
  doi       = {10.1007/s11118-021-09965-6},
  url       = {https://doi.org/10.1007/s11118-021-09965-6}
}

@article{BC,
  author  = {Bass, Richard F. and Chen, Zhen-Qing},
  title   = {Regularity of harmonic functions for a class of singular stable-like processes},
  journal = {Math. Z.},
  volume  = {266},
  number  = {3},
  pages   = {489--503},
  year    = {2010},
  doi     = {10.1007/s00209-009-0537-1}
}

@article{GK,
  author  = {Grzywny, Tomasz and Kwa{\'s}nicki, Mateusz},
  title   = {Potential kernels, probabilities of hitting a ball, harmonic functions and the boundary {H}arnack inequality for unimodal {L}{\'e}vy processes},
  journal = {Stochastic Process. Appl.},
  volume  = {128},
  pages   = {1--38},
  year    = {2018},
  doi     = {10.1016/j.spa.2017.04.006}
}
\end{document}